\documentclass{aic}

\aicAUTHORdetails{%  
  title = {The Bandwidth Theorem in Sparse Graphs}, %% please capitalize all significant words
  author = {Peter Allen, Julia B\"ottcher, Julia Ehrenm\"uller, Anusch Taraz},
  plaintextauthor = {Peter Allen, Julia Boettcher, Julia Ehrenmueller, Anusch Taraz},
}   %%% END \aicAUTHORdetails

\aicEDITORdetails{%
   year={2020},
   number={6},
   received={21 February 2019},
   published={15 May 2020},
   doi={10.19086/aic.12849},
}   %%% END \aicEDITORdetails

\usepackage{amsmath,amsfonts,amssymb,amsthm}
\usepackage[utf8]{inputenc}
\usepackage{ifthen}
\usepackage[style=american]{csquotes}
\usepackage{color}
\usepackage{graphics}
\usepackage{enumitem}
\usepackage[english,algoruled,lined,noresetcount,norelsize]{algorithm2e}

\setitemize{leftmargin=\parindent}
\setenumerate{leftmargin=*}

\newtheorem{theorem}{Theorem}
\newtheorem{lemma}[theorem]{Lemma}
\newtheorem{proposition}[theorem]{Proposition}
\newtheorem{corollary}[theorem]{Corollary}

\newtheorem{claim}[theorem]{Claim}
\newtheorem{conjecture}[theorem]{Conjecture}

\theoremstyle{definition}
\newtheorem{definition}[theorem]{Definition}

\theoremstyle{remark}

\newcommand{\oldqed}{}
\def\endofClaim{\hfill\scalebox{.6}{$\Box$}}

\newenvironment{claimproof}[1][Proof]{
  \renewcommand{\oldqed}{\qedsymbol}
  \renewcommand{\qedsymbol}{\endofClaim}
  \begin{proof}[#1]
}{
  \end{proof}
  \renewcommand{\qedsymbol}{\oldqed}
}

\def\subset{\subseteq}

\newcommand{\Ex}{\mathbb{E}}
\newcommand{\tX}{\widetilde{X}}
\newcommand{\tW}{\widetilde{W}}
\newcommand{\tcW}{\widetilde{\mathcal{W}}}
\newcommand{\tV}{\widetilde{V}}
\newcommand{\eps}{\varepsilon}
\renewcommand{\epsilon}{\varepsilon}
\newcommand{\NGa}{N_{\Gamma}}
\newcommand{\Vij}{V_{i,j}}
\newcommand{\epsa}{\eps^{\ast}}
\newcommand{\epsaa}{\eps^{\ast\ast}}
\newcommand{\Ca}{C^{\ast}}

\newcommand{\symd}{\triangle}
\renewcommand{\Pr}{\mathbb{P}}

\newcommand{\scalefactor}{0.5}

\newcommand{\ao}{\alpha_{\scalebox{\scalefactor}{$\mathrm{OSRIL}$}}}
\newcommand{\at}{\alpha_{\scalebox{\scalefactor}{$\mathrm{TSRIL}$}}}
\newcommand{\eo}{\eps_{\scalebox{\scalefactor}{$\mathrm{OSRIL}$}}}
\newcommand{\et}{\eps_{\scalebox{\scalefactor}{$\mathrm{TSRIL}$}}}

\newcommand{\eBL}{\eps_{\scalebox{\scalefactor}{$\mathrm{BL}$}}}
\newcommand{\CBL}{C_{\scalebox{\scalefactor}{$\mathrm{BL}$}}}

\newcommand{\im}{\mathrm{im}}
\newcommand{\dom}{\mathrm{dom}}
\newcommand{\dist}{\mathrm{dist}}
\newcommand{\Bs}{B^{\mathrm{sw}}}

\def\itm#1{\rm ({#1})} 
\def\itmit#1{\itm{\it #1\,}} 
\def\rom{\itmit{\roman{*}}}

\def\itmarab#1{\mbox{\itm{{\it #1\,}\arabic{*}\hspace{.05em}}}}
\def\itmarabp#1#2{\mbox{\itm{{\it #1\,}\arabic{*}#2\hspace{.05em}}}}
\def\itmsol#1#2{\mbox{\itm{{\it #1\,}#2\hspace{.05em}}}}

\newcommand{\By}[2]{\overset{\mbox{\tiny{#1}}}{#2}} 
\newcommand{\ByRef}[2]{   \By{\eqref{#1}}{#2} }

\newcommand{\leBy}[1]{    \By{#1}{\le} } 
\newcommand{\geBy}[1]{    \By{#1}{\ge} }

\newcommand{\leByRef}[1]{ \ByRef{#1}{\le} }

\newcommand{\cE}{\mathcal{E}}
\newcommand{\cH}{\mathcal{H}}
\newcommand{\cI}{\mathcal I}
\newcommand{\cJ}{\mathcal J}
\newcommand{\cL}{\mathcal L}
\newcommand{\cP}{\mathcal{P}}
\newcommand{\cV}{\mathcal V}
\newcommand{\cU}{\mathcal U}
\newcommand{\cW}{\mathcal W}

\newcommand{\cX}{\mathcal X}
\newcommand{\tcX}{\widetilde{\mathcal{X}}}
\newcommand{\tcV}{\widetilde{\mathcal{V}}}
\newcommand{\pitau}{\pi^\tau}
\newcommand{\tD}{\tilde{D}}

\newcommand{\dcup}{\mathbin{\text{\mbox{\makebox[0mm][c]{\hphantom{$\cup$}$\cdot$}$\cup$}}}}

\newcommand\restr[2]{{% we make the whole thing an ordinary symbol
  \left.\kern-\nulldelimiterspace % automatically resize the bar with \right
  #1 % the function
  \vphantom{\big|} % pretend it's a little taller at normal size
  \right|_{#2} % this is the delimiter
  }}

\begin{document}

%%%%% FRONTMATTER %%%%%%%%%%%%%%%%%%%%%%%%%%%%%%%%%%%%%%%%%%

\begin{frontmatter}[classification=text]
%% EDITOR: this will force the keywords to appear right after the Abstract.
%%   If the abstract is too long and would force the keywords off the
%%   front page, please comment out % [classification=text] above
%%   This way the keywords will be floated on the bottom of the first page
%%   even though the Abstract spills over to the next page.

%%% AUTHOR: Title goes here.  This line is optional.  You must use it
%%   if title has footnote attached or requires nontrivial typesetting,
%%   e.g., inclusion of linebreaks to force nice layout.
\title{The Bandwidth Theorem in Sparse Graphs} %% please capitalize all significant words

%%% AUTHOR:
%%% List all authors. If you wish, place grant acknowledgements in \thanks.
%%% In brackets include a short tag for each author, e.g.:
%   \author[pgom]{Paul Erd\H{o}s\thanks{Supported by...}}

\author[PA]{Peter Allen}
\author[JB]{Julia B\"ottcher}
\author[JE]{Julia Ehrenm\"uller}
\author[AT]{Anusch Taraz}

%%% AUTHOR: Abstract goes here
\begin{abstract}
The bandwidth theorem [Mathematische Annalen, 343(1):175--205, 2009] states
that any $n$-vertex graph~$G$ with minimum degree $\big(\tfrac{k-1}{k}+o(1)\big)n$ contains all $n$-vertex $k$-colourable graphs~$H$ with bounded maximum degree and bandwidth $o(n)$.
We provide sparse analogues of this statement in random graphs as well as pseudorandom graphs.

More precisely, we show that for $p\gg \big(\tfrac{\log n}{n}\big)^{1/\Delta}$ asymptotically almost surely each spanning subgraph~$G$ of
$G(n,p)$ with minimum degree $\big(\tfrac{k-1}{k}+o(1)\big)pn$ contains all
$n$-vertex $k$-colourable graphs~$H$ with maximum degree $\Delta$, bandwidth
$o(n)$, and at least $C p^{-2}$ vertices not contained in any triangle. A
similar result is shown for sufficiently bijumbled graphs, which, to the best of
our knowledge, is the first resilience result in pseudorandom graphs for a rich
class of \emph{spanning }subgraphs.
Finally, we provide improved results for~$H$ with small degeneracy, which in particular imply a resilience result in $G(n,p)$ with respect to the containment of spanning bounded degree trees for $p\gg \big(\tfrac{\log n}{n}\big)^{1/3}$.
\end{abstract}

\end{frontmatter}

%%%%% PAPER BODY %%%%%%%%%%%%%%%%%%%%%%%%%%%%%%%%%%%%%%%%%%

%%% AUTHOR: body of paper starts here

%%%%%%%%%%%%%%%%%%%%%%%%%%%%%%%%%%%%%%%%%%%%%%%%%%%%%%%%%

% Introduction

%%%%%%%%%%%%%%%%%%%%%%%%%%%%%%%%%%%%%%%%%%%%%%%%%%%%%%%%%

\section{Introduction}

A central topic in extremal graph theory is to determine minimum degree
conditions which force a graph $G$ to contain a copy of some large or even
spanning subgraph $H$. The prototypical example of such a theorem is Dirac's
theorem~\cite{dirac1952}, which states that if $\delta(G)\ge\tfrac12 v(G)$ and
$v(G)\ge 3$, then $G$ is Hamiltonian. 
Analogous results were established for a wide range of spanning subgraphs~$H$ with bounded maximum degree such as powers of Hamilton cycles, trees, or $F$-factors for any fixed graph~$F$ (see e.g.~\cite{kuhnsurvey} for a survey). One feature that all these subgraphs~$H$ have in common is that their \emph{bandwidth} is small. The bandwidth of a graph~$H$ is the minimum $b$ such that there is a labelling of the vertex set of~$H$ by integers $1, \ldots, n$ with $|i-j| \leq b$ for every edge $ij$ of~$H$. And indeed, it was shown in~\cite{bottcher2009proof} that a more general result holds, which provides a minimum degree condition forcing any spanning bounded degree subgraphs of small bandwidth.

\begin{theorem}[Bandwidth Theorem~\cite{bottcher2009proof}]
\label{thm:bandwidth}
For every $\gamma >0$, $\Delta \geq 2$, and $k\geq 1$, there exist $\beta>0$ and $n_0 \geq 1$ such that for every $n\geq n_0$ the following holds. If $G$ is a graph on $n$ vertices with minimum degree $\delta(G) \geq \left(\frac{k-1}{k}+\gamma\right)n$ and if $H$ is a $k$-colourable graph on $n$ vertices with maximum degree $\Delta(H) \leq \Delta$ and bandwidth at most $\beta n$, then $G$ contains a copy of $H$.
\end{theorem}

We remark that in contrast to the above mentioned earlier results for specific bounded degree spanning subgraphs the minimum degree condition in this theorem has an error term~$\gamma n$, and it is known that this cannot completely be omitted in this general statement. In that sense the minimum degree condition in Theorem~\ref{thm:bandwidth} is best-possible. It is also known that the bandwidth condition cannot be dropped completely (see~\cite{bottcher2009proof} for further explanations). Moreover, this condition does not limit the class of graphs under consideration unreasonably, because many interesting classes of graphs have sublinear bandwidth. Indeed, it was shown in \cite{bottcher2010bandwidth} that for bounded degree $n$-vertex graphs, restricting the bandwidth to $o(n)$ is equivalent to restricting the treewidth to $o(n)$ or forbidding linear sized expanding subgraphs, which implies that bounded degree planar graphs, or more generally classes of bounded degree graphs defined by forbidding some fixed minor have bandwidth $o(n)$. Generalisations of Theorem~\ref{thm:bandwidth} were obtained in~\cite{BoeHeiTar,BoeTarWue,KnoTre,Lee:degenerateBandwidth}.

In this paper we are interested in the transference of
Theorem~\ref{thm:bandwidth} to sparse graphs.  Such transference results
recently received much attention, including for example the breakthrough result
on the transference of Tur\'an's theorem to random graphs by Conlon and
Gowers~\cite{ConGow} and Schacht~\cite{Schacht}. The random graph model we shall
consider here is the binomial random graph $G(n,p)$, which has $n$ vertices, and
each pair of vertices forms an edge independently with probability $p$. We shall
study the asymptotic appearance of spanning subgraphs~$H_n$ in~$G(n,p)$ under
adversarial edge deletions. We denote the sequence of graphs we consider by
$H=(H_n)$, and, abusing notation slightly also often write~$H$ for the
graph~$H_n$, when it is clear from the context what~$n$ is.

The appearance of large or spanning subgraphs of $G(n,p)$ was studied since the
early days of probabilistic combinatorics and by now many important results were
obtained. Gems include the theorem of Riordan~\cite{Riordan} which gives a very
good, and in many cases tight, upper bound on the threshold for $G(n,p)$ to
contain a general sequence of spanning graphs~$H=(H_n)$, and the
Johansson--Kahn--Vu theorem~\cite{JohKahVu} on $F$-factors, which we state for
the case $F=K_k$.

\begin{theorem}[Johansson, Kahn and Vu~\cite{JohKahVu}]\label{thm:JKV}
 For each $k\ge 3$ there exists $C>0$ such that the following holds for $p=p(n)\ge C n^{-2/k}(\log n)^{1/\binom{k}{2}}$. Asymptotically almost surely, $G(n,p)$ contains a $K_k$-factor, that is, a collection of $\big\lfloor\tfrac{n}{k}\big\rfloor$ pairwise vertex-disjoint copies of $K_k$.
\end{theorem} 

For a sequence of spanning graphs $H=(H_n)$ with maximum degree $\Delta(H)\le\Delta$, Riordan's theorem implies that $G(n,p)$ \emph{asymptotically almost surely} (a.a.s.), that is, with probability tending to~$1$ as~$n$ tends to infinity, contains~$H$ as a subgraph if $p\cdot n^{\frac{2}{\Delta+1}-\frac{2}{\Delta(\Delta+1)}}\to\infty$. This is not believed to be best possible. Indeed, Theorem~\ref{thm:JKV} states that the threshold for $G(n,p)$ to contain a $K_{\Delta+1}$-factor is $n^{-2/(\Delta+1)}(\log n)^{1/\binom{\Delta+1}{2}}$, and it is conjectured in~\cite{FerLuhNgu} that above this probability
we also get any other sequence of spanning graphs $H=(H_n)$ with $\Delta(H)\le\Delta$. This was proved, using Theorem~\ref{thm:JKV}, to be true for almost spanning graphs by Ferber, Luh, and Nguyen~\cite{FerLuhNgu}.

\begin{theorem}[Ferber, Luh, and Nguyen~\cite{FerLuhNgu}]
  \label{thm:FerLuhNgu}
  For every $\eps>0$ and $\Delta\ge 1$, and every sequence $H=(H_n)$ of graphs with $v(H)\le(1-\eps)n$ and $\Delta(H)\le\Delta$, the random graph $G(n,p)$ a.a.s.\ contains~$H$ if
  $p\cdot n^{2/(\Delta+1)}/(\log n)^{1/\binom{\Delta+1}{2}} \to\infty$.
\end{theorem}

Better bounds are available if we further know that the degeneracy of~$H$ is bounded by a constant much smaller than $\Delta(H)$. The \emph{degeneracy} of~$H$ is the smallest integer~$D$ such that any subgraph of~$H$ has a vertex of degree at most~$D$. Surprisingly, for this class of graphs~$H$ already Riordan's theorem implies an essentially optimal bound.

\begin{corollary}[of Riordan's theorem~\cite{Riordan}]
  \label{cor:Riordan}
  For every $\Delta\ge 1$ and $D\ge 3$, and every sequence $H=(H_n)$ of graphs with $v(H)\le n$ and $\Delta(H)\le\Delta$ and degeneracy at most~$D$, the random graph 
  $G(n,p)$ a.a.s.\ contains~$H$ if $p\cdot n^{1/D} \to\infty$.
\end{corollary}

This is best possible because a simple first moment calculation shows that if $p\cdot n^{1/D}\to 0$ then $G(n,p)$ a.a.s.\ does not contain the $D$-th power of a Hamilton path, which is a $D$-degenerate graph with maximum degree $2D$.

Two features that both Riordan's theorem and Theorem~\ref{thm:JKV} (and consequently all results which rely on them, such as Theorem~\ref{thm:FerLuhNgu} and Corollary~\ref{cor:Riordan}) have in common is that their proofs are non-constructive, and the proof techniques do not allow for so-called universality results.  A graph~$G$ is said to be \emph{universal} for a family~$\cH$ of graphs if~$G$ contains copies of all graphs in~$\cH$ simultaneously.  The random graph~$G(n,p)$ is known to be universal for various families of graphs, but in almost all cases we only know an upper bound on the threshold for universality, which we do not believe is the correct answer.

The reason why probabilistic existence results such as Corollary~\ref{cor:Riordan} do not imply universality is that in $G(n,p)$ the failure probability for containing any given spanning graph $H$ without isolated vertices is at least $(1-p)^{n-1}$, the probability that a fixed vertex of $G(n,p)$ is isolated. This probability is too large to apply a union bound. Thus, to prove universality results, one needs to show that any graph $G$ with some collection of properties that $G(n,p)$ a.a.s.\ possesses must contain any given $H\in\cH$. Using this approach, and improving on a series of earlier results, Dellamonica, Kohayakawa, R\"odl and Ruci\'nski~\cite{dellamonica2014} obtained the following universality result for the family $\cH(n,\Delta)$ of $n$-vertex graphs with maximum degree $\Delta$.

\begin{theorem}[Dellamonica, Kohayakawa, R\"odl and Ruci\'nski~\cite{dellamonica2014}]
\label{thm:universal}
  For all $\Delta\ge3$ there is~$C$ such that
  if $p\ge C\big(\frac{\log n}{n}\big)^{1/\Delta}$ then $G(n,p)$ is a.a.s.\ universal for $\cH(n,\Delta)$.
\end{theorem}

However, it is conjectured that universality and the appearance of a $K_{\Delta+1}$-factor occur together, at the threshold given in Theorem~\ref{thm:FerLuhNgu}. A probability bound which is better, but still far from the conjectured truth, was so far only established for almost spanning graphs by Conlon, Ferber, Nenadov and \v{S}kori\'c~\cite{CFNS}, who showed that for $\Delta\ge 3$, if $p\gg n^{-1/(\Delta-1)}\log^5 n$ then $G(n,p)$ is a.a.s.\ universal for $\cH\big((1-o(1))n,\Delta)\big)$. For graphs with small degeneracy, again, the following better bound exists, but this also is far away from the threshold in Corollary~\ref{cor:Riordan}, which is a plausible candidate for the correct answer.

\begin{theorem}[Allen, B\"ottcher, H\`an, Kohayakawa, Person~\cite{blowup}]
  \label{thm:Duniversal}
  For all $\Delta, D\ge 1$ there is~$C$ such that if $p\ge C\big(\frac{\log n}{n}\big)^{1/(2D+1)}$ then $G(n,p)$ is a.a.s.\ universal for all graphs in~$\cH(n,\Delta)$ with degeneracy at most~$D$.
\end{theorem}
Very recently Conlon and Nenadov~\cite{ConNen} established an essentially best possible bound on~$p$ for an almost spanning analogue: They 
showed that $G(n, p)$ is a.a.s.\ universal for all graphs in $\cH\big((1-\eps)n,\Delta\big)$ with degeneracy at most~$D$ for $p\ge\big(\frac{C\log^2 n}{n\log\log n}\big)^{1/D}$.

\smallskip

Furthermore, one may ask how robustly $G(n,p)$ contains certain large subgraphs~$H$. Questions of this type were considered by Alon, Capalbo, Kohayakawa, R\"odl, Ruci\'nski and Szemer\'edi~\cite{ACKRRS}, and further popularised by Sudakov and Vu~\cite{SudVu}, who introduced the term local resilience.  Let $\cP$ be a monotone increasing graph property and let $G$ be a graph in $\cP$. The \emph{local resilience} of $G$ with respect to $\cP$ is defined to be the minimum $r \in \mathbb R$ such that by deleting at each vertex $v\in V(G)$ at most $r\deg(v)$ edges one can obtain a graph not in $\cP$.  In this language, for example, Theorem~\ref{thm:bandwidth} says that the local resilience of $K_n=G(n,1)$ with respect to being universal for all $k$-colourable graphs in $\cH(n,\Delta)$ with sublinear bandwidth is $\frac1k-o(1)$, and a sparse analogue asks for a similar local resilience result to hold a.a.s.\ for $G(n,p)$.

Lee and Sudakov~\cite{lee2012} obtained a very strong local resilience result for Hamilton cycles.  Improving on~\cite{SudVu}, they showed that the local resilience of $G(n,p)$ with respect to Hamiltonicity is a.a.s.\ at least $\tfrac12-o(1)$ when $p=\Omega\big(\tfrac{\log n}{n}\big)$. This is optimal up to the constant factor, as below this probability $G(n,p)$ is itself a.a.s.\ not Hamiltonian. Triangle factors were investigated by Balogh, Lee and Samotij~\cite{balogh2012corradi}, who proved that the local resilience of $G(n,p)$ with respect to the containment of a triangle factor on $n-O\big(p^{-2}\big)$ vertices is a.a.s.\ $\frac13-o(1)$ if $p\gg (\frac{\log n}{n})^{1/2}$. It was observed by Huang, Lee, and Sudakov~\cite{huang2012} that we cannot hope to cover all vertices in the host graph with triangles: Already for constant~$p$ it is easy to delete all edges in the neighbourhood of $\Theta(p^{-2})$ vertices in~$G(n,p)$ without violating the resilience condition.  Very recently Noever and Steger~\cite{NoeSte} showed that the local resilience of $G(n,p)$ with respect to containing a $(1-o(1))n$-vertex squared cycle (a cycle with all edges between vertices at distance $2$ added) is a.a.s.\ $\tfrac13-o(1)$ provided $p\gg n^{-1/2+o(1)}$. Even more recently, Nenadov and \v{S}kori\'{c}~\cite{NenSko} removed the $\log$-factor from the probability bound of~\cite{balogh2012corradi}. These results are notable in that the bounds on~$p$ are close to optimal: for $p\ll n^{-1/2}$, a.a.s.\ most edges of~$G(n,p)$ are not in triangles and one can obtain a triangle-free graph by deleting only a tiny fraction of edges at each vertex, so that the local resilience of $G(n,p)$ with respect to containing triangles is $o(1)$. Furthermore, for $p\ll n^{-1/2}$ the random graph $G(n,p)$ itself does not contain any $(1-o(1))n$-vertex squared cycle.

More general subgraphs~$H$ were considered by Huang, Lee and
Sudakov~\cite{huang2012}, who proved an analogue of the bandwidth theorem
(Theorem~\ref{thm:bandwidth}) in $G(n,p)$ with $0<p<1$ constant (for
subgraphs~$H$ with certain vertices not in triangles). A version which works for
much smaller probabilities $p\gg(\frac{\log n}{n})^{1/\Delta}$ in the special
case of bipartite graphs~$H$ on $(1-o(1))n$ vertices (with maximum
degree~$\Delta$ and sublinear bandwidth) was established
in~\cite{bottcher2013almost}. In~\cite[Theorem~1.9]{blowup} the so-called sparse
blow-up lemma is used to prove a similar result for graphs~$H$ which are not
necessarily bipartite.
An even better bound on~$p$ was obtained when we
restrict the problem to the class of almost spanning trees~$H$: Balogh, Csaba
and Samotij~\cite{balogh2011} proved that the local resilience of $G(n,p)$ with
respect to containing copies of all trees $T$ on $(1-o(1))n$ vertices with
$\Delta(T) \leq \Delta$ is asymptotically almost surely at least~$1/2 - o(1)$ if
$p \gg 1/n$, which is optimal. Finally, returning to $H$-factors, Conlon,
Gowers, Samotij and Schacht~\cite{CGSS} gave resilience results for
almost-spanning $H$-factors which work down to the optimal probability
threshold, but leave a linear number of vertices uncovered; Nenadov and
\v{S}kori\'c~\cite{NenSko} substantially improved this, but (for most graphs)
the number of vertices left uncovered in their result is still not the correct
order of magnitude.

\subsection*{Our results.}

We prove several sparse analogues of the bandwidth theorem
(Theorem~\ref{thm:bandwidth}). Our first result is a version for sparse random
graphs, extending the resilience results of Huang, Lee and
Sudakov~\cite{huang2012}, \cite{bottcher2013almost}, and~\cite{blowup}.

\begin{theorem}
\label{thm:main}
For each $\gamma >0$, $\Delta \geq 2$, and $k \geq 1$, there exist constants $\beta^\ast >0$ and $\Ca >0$ such that the following holds asymptotically almost surely for $\Gamma = G(n,p)$ if $p \geq \Ca\big(\frac{\log n}{n}\big)^{1/\Delta}$. Let $G$ be a spanning subgraph of $\Gamma$ with $\delta(G) \geq\left(\frac{k-1}{k}+ \gamma\right)pn$, and let $H$ be a $k$-colourable graph on $n$ vertices with $\Delta(H) \leq \Delta$, bandwidth at most $\beta^\ast n$, and with at least $\Ca p^{-2}$ vertices which are not contained in any triangles of $H$. Then $G$ contains a copy of $H$.  
\end{theorem}

Observe that the bound on~$p$ achieved in this result matches the bound in the universality result in Theorem~\ref{thm:universal}. Hence, though we do not believe it to be optimal, improving it will most likely be hard.  Moreover, as explained in conjunction with Theorem~\ref{thm:bandwidth}, the minimum degree of $G$ cannot be decreased, nor can the bandwidth restriction be removed.  As indicated above, it is also necessary that $\Theta(p^{-2})$ vertices of~$H$ are not in triangles.

If in addition the subgraph~$H$ is also $D$-degenerate, we can prove a variant of Theorem~\ref{thm:main} for $p\gg (\log n/n)^{1/(2D+1)}$. 
Again, this probability bound matches the one in the currently best universality result for $D$-degenerate graphs given in Theorem~\ref{thm:Duniversal}.
As before we require a certain number of vertices which are not in triangles of~$H$. But, due to technicalities of our proof method, in addition these vertices are now also required not to be in four-cycles.

\begin{theorem}\label{thm:degenerate}
For each $\gamma >0$, $\Delta \geq 2$, and $D, k \geq 1$, there exist constants $\beta^\ast >0$ and $\Ca >0$ such that the following holds asymptotically almost surely for $\Gamma = G(n,p)$ if $p \geq \Ca\big(\frac{\log n}{n}\big)^{1/(2D+1)}$. Let $G$ be a spanning subgraph of $\Gamma$ with $\delta(G) \geq\left(\frac{k-1}{k}+ \gamma\right)pn$ and let $H$ be a $D$-degenerate, $k$-colourable graph on $n$ vertices with $\Delta(H) \leq \Delta$, bandwidth at most $\beta^\ast n$ and with at least $\Ca p^{-2}$ vertices which are not contained in any triangles or four-cycles of $H$. Then $G$ contains a copy of $H$.
\end{theorem}

Since trees are $1$-degenerate this implies a resilience result for trees when $p\gg (\frac{\log n}{n})^{1/3}$. This probability bound is much worse than that obtained by Balogh, Csaba, and Samotij~\cite{balogh2011} for almost-spanning trees, and unlikely to be optimal, but it is the first resilience result for bounded degree
\emph{spanning} trees in $G(n,p)$.

Finally, we also establish a sparse analogue of Theorem~\ref{thm:bandwidth} in bijumbled graphs, one of the most widely studied classes of pseudorandom graphs.
A graph $\Gamma$ is \emph{$(p,\nu)$-bijumbled} if for all disjoint sets $X,Y\subset V(\Gamma)$ we have
\[\big|e(X,Y)-p|X||Y|\big|\le\nu\sqrt{|X||Y|}\,.\]
This definition goes back to an equivalent notion introduced by Thomason~\cite{Tho87} who initiated the study of pseudorandom graphs. It is also related to the well investigated class of $(n,d,\lambda)$-graphs in that an $(n,d,\lambda)$-graph is $\big(\tfrac{d}{n},\lambda\big)$-bijumbled.  

Only very recently a universality result similar to Theorem~\ref{thm:universal} was established for bijumbled graphs in~\cite{blowup}, where it was shown that $(p,\nu)$-bijumbled graphs~$G$ with $\delta(G)\ge\frac12pn$ and $\nu\ll p^{\max(4,(3\Delta+1)/2)}n$ are universal for $\cH(n,\Delta)$. Our resilience result works for the same bijumbledness condition, though we do not believe it to be optimal.
Local resilience results in bijumbled graphs were so far only obtained for special subgraphs~$H$: Dellamonica, Kohayakawa, Marciniszyn, and Steger~\cite{dellamonica2008} considered cycles~$H$
of length $(1-o(1))n$, the results of Conlon, Fox and Zhao~\cite{CFZ} imply resilience for $F$-factors covering $(1-o(1))n$ vertices, and Krivelevich, Lee and Sudakov~\cite{KriLeeSud}
established a resilience result for pancyclicity. Hence, previous to this work only little was known about the resilience of bijumbled (or indeed any other common notion of pseudorandom) graphs.

\begin{theorem}
\label{thm:jumbled}
 For each $\gamma >0$, $\Delta \geq 2$, and $k \geq 1$, there exists a constant $c >0$ such that the following holds for any $p>0$. Given $\nu\le cp^{\max(4,(3\Delta+1)/2)}n$, suppose $\Gamma$ is a $\big(p,\nu\big)$-bijumbled graph, $G$ is a spanning subgraph of $\Gamma$ with $\delta(G) \geq\big(\tfrac{k-1}{k}+\gamma\big)pn$, and $H$ is a $k$-colourable graph on $n$ vertices with $\Delta(H) \leq \Delta$ and bandwidth at most $c n$. Suppose further that there are at least $c^{-1}p^{-6} \nu^2n^{-1}$ vertices in $V(H)$ that are not contained in any triangles of $H$. Then $G$ contains a copy of $H$.
\end{theorem}

We remark that the requirement of $\Ca p^{-6}\nu^2 n^{-1}$ vertices of $H$ not
being in triangles comes from our use of a so-called regularity inheritance
lemma proved in~\cite{ABSS}; this bound is not believed to be optimal (see
Section~\ref{sec:remarks} for further details).

The proofs of our results rely on sparse versions of the so-called blow-up
lemma. The blow-up lemma is an important tool in extremal graph theory, proved
by Koml\'os, S\'ark\"ozy and Szemer\'edi~\cite{komlos1997blow} and was for
example instrumental in the proof of the bandwidth theorem and its analogue in
$G(n,p)$ for constant~$p$ by Huang, Lee and Sudakov~\cite{huang2012}. However it
applies only to dense graphs. Several of the earlier resilience results in
sparse random graphs developed sparse blow-up type results handling special
classes of graphs: Balogh, Lee and Samotij~\cite{balogh2012corradi} proved a
sparse blow-up lemma for embedding triangle factors, and
in~\cite{bottcher2013almost} a blow-up lemma for embedding almost spanning
bipartite graphs in sparse graphs was used. Full versions of the blow-up lemma
in sparse random graphs and pseudorandom graphs were established only very
recently in~\cite{blowup}. We will use these here.
We remark that a simple use of these blow-up lemmas gives almost
spanning versions of our main results (as already noted in~\cite{blowup}), and
the main work here is to extend this to spanning embedding results, which turns
out to be much harder.

Further, we note that we actually prove somewhat stronger statements than Theorem~\ref{thm:main}, Theorem~\ref{thm:degenerate}, and Theorem~\ref{thm:jumbled} in the same sense in that a stronger statement than Theorem~\ref{thm:bandwidth} was proven in~\cite{bottcher2009proof}: 
we allow~$H$ in fact to be $(k+1)$-colourable, where the additional colour may only be assigned to very few well distributed vertices (for details see, e.g., Theorem~\ref{thm:maink} below). Thus, for instance, even though Theorem~\ref{thm:main} only implies that the local resilience of $G(n,p)$ with respect to Hamiltonicity is a.a.s.\ $\tfrac12-o(1)$ when $n$ is even, Theorem~\ref{thm:maink} implies it also for $n$ odd.

\subsection*{Organisation.}

The remainder of this paper is organised as follows. In Section~\ref{sec:preliminaries} we introduce necessary definitions and collect some known results which we need in our proofs. Next, in Section~\ref{sec:mainlemmas}, we outline the proof of the bandwidth theorem in sparse random graphs, Theorem~\ref{thm:main}, and state the four technical lemmas we require. Their proofs are given in Sections~\ref{sec:prooflemG}--\ref{sec:prooflembalancing}, and the proof of Theorem~\ref{thm:main} is presented in Section~\ref{sec:proofmain}. We provide the modifications required to obtain Theorem~\ref{thm:degenerate} in Section~\ref{sec:proofdegen}, and those required for Theorem~\ref{thm:jumbled} in Section~\ref{sec:proofjumbled}. Finally, Section~\ref{sec:remarks} contains some concluding remarks, and Appendix~\ref{app:tools} contains proofs of a few results which are more or less standard but which we could not find in the form we need in the literature.

%%%%%%%%%%%%%%%%%%%%%%%%%%%%%%%%%%%%%%%%%%%%%%%%%%%%%%%%%

% Preliminaries

%%%%%%%%%%%%%%%%%%%%%%%%%%%%%%%%%%%%%%%%%%%%%%%%%%%%%%%%%
\section{Preliminaries}
\label{sec:preliminaries}

Throughout the paper $\log$ denotes the natural logarithm.
We assume that the order $n$ of all graphs tends to infinity and therefore is sufficiently large whenever necessary. 
For reals $a, b >0$ and integer $k\in \mathbb N$, we use the notation $(a\pm b) = [a-b, a+b]$ and $[k] = \{1, \ldots, k\}$.  
Our graph-theoretic notation is standard and follows \cite{bollobas1998modern}. In particular, given a graph $G$ its vertex set 
is denoted by $V(G)$ and its edge set by $E(G)$. Let $A,B\subseteq V$ be disjoint vertex sets. We denote the number of edges between $A$ and $B$ by $e(A,B)$. 
For a vertex $v \in V(G)$ we write $N_G(v)$ for the neighbourhood of $v$ in $G$ and $N_G(v,A):= N_G(v) \cap A$ for the neighbourhood of $v$ restricted to $A$ in $G$. 
Given vertices $v_1, \ldots, v_k \in V(G)$ we denote the joint neighbourhood of $v_1, \ldots, v_k$ restricted to a set $A$ by $N_G(v_1, \ldots, v_k; A) = \bigcap_{i\in[k]} N_G(v_i, A)$. 
Finally, we use the notation $\deg_G(v) := |N_G(v)|$ and $\deg_G(v, A) := |N_G(v,A)|$, as well as $\deg_G(v_1, \ldots, v_k; A) := |N_G(v_1, \ldots, v_k;A)|$ for the degree of $v$ in $G$, the degree of $v$ restricted to $A$ in $G$ and the size of the joint neighbourhood of $v_1, \ldots, v_k$ restricted to $A$ in $G$. 
Finally, let $\deg_G(v) := |N_G(v)|$ be the degree of $v$ in $G$. 
For the sake of readability, we do not intend to optimise the constants in our theorems and proofs.

Now we introduce some definitions and results of the regularity method as well as related tools that are essential in our proofs. In particular, we state a minimum degree version of the sparse regularity lemma (Lemma~\ref{lem:regularitylemma}) and the sparse blow up lemma (Lemma~\ref{thm:blowup}). Both lemmas use the concept of regular pairs. Let $G= (V,E)$ be a graph, $\eps, d >0$, and $p \in (0,1]$. Moreover, let $X,Y \subseteq V$ be two disjoint nonempty sets. The \emph{$p$-density} of the pair $(X,Y)$ is defined as \[d_{G,p}(X,Y) := \frac{e_G(X,Y)}{p|X||Y|}.\]
For most of this paper, when we work with random graphs, we will be interested in the regularity concept called \emph{lower-regularity}. When we work with bijumbled graphs, on the other hand, we will need the stronger concept \emph{regularity}. The difference is that in the former we impose only lower bounds on $p$-densities, whereas in the latter we impose in addition upper bounds. The main reason for this difference is that our `regularity inheritance lemmas' below have different requirements in random and in bijumbled graphs; we do not otherwise make use of the extra strength of `regular' as opposed to `lower-regular'. 

We also need to define super-regularity, for which we require $G$ to be a subgraph of a graph $\Gamma$, which will be the random or bijumbled graph whose resilience properties we are establishing.

\begin{definition}[$(\eps,d,p)$-(super-)(lower-)regular pairs]
  \label{def:regular}
  Let~$G$ and~$\Gamma$ be graphs with $G\subset\Gamma$.
The pair $(X,Y)$ is called \emph{$(\eps,d,p)_G$-lower-regular} if for every $X'\subseteq X$ and $Y'\subseteq Y$ with $|X'|\geq \eps|X|$ and $|Y'|\geq \eps |Y|$ we have  $d_{G,p}(X',Y') \geq d- \eps$.

It is called \emph{$(\eps,d,p)_G$-regular} if there exists $d'\ge d$ such that for every $X'\subseteq X$ and $Y'\subseteq Y$ with $|X'|\geq \eps|X|$ and $|Y'|\geq \eps |Y|$ we have  $d_{G,p}(X',Y') = d'\pm \eps$.

If $(X,Y)$ is either $(\eps,d,p)_G$-lower-regular or $(\eps,d,p)_G$-regular, and in addition we have
\begin{align*}
 |N_G(x,Y)| &\geq (d-\eps)\max\big(p|Y|,\deg_\Gamma(x,Y)/2\big)\quad\text{and}\\
 |N_G(y,X)| &\geq (d-\eps)\max\big(p|X|,\deg_\Gamma(y,X)/2\big)
\end{align*}
for every $x \in X$ and $y \in Y$, then the pair $(X,Y)$ is called \emph{$(\eps,d,p)_G$-super-regular}. When we use super-regularity it will be clear from the context whether $(X,Y)$ is lower-regular or regular.
\end{definition}

Note that a regular pair is by definition lower-regular, though the converse does not hold. Furthermore, although the definition of super-regularity of $G$ contains a reference to $\Gamma$, at each place in this paper where we use super-regularity, we will see that the first term in the maximum is larger than the second. When it is clear from the context, we may omit the subscript $G$ in $(\eps,d,p)_G$-\mbox{(super-)}regular which is used to indicate with respect to which graph a pair is \mbox{(super-)}regular. A direct consequence of the definition of $(\eps,d,p)$-lower-regular pairs is the following proposition about the sizes of neighbourhoods in lower-regular pairs. 

\begin{proposition}
\label{prop:neighbourhood}
Let $(X,Y)$ be $(\eps, d,p)$-lower-regular. Then there are less than $\eps |X|$ vertices $x\in X$ with $|N(x,Y)| < (d-\eps)p|Y|$. \qed
\end{proposition} 

The next proposition asserts that altering the vertex sets in an
$(\eps,d,p)$-(lower-)regular pair slightly does not destroy (lower-)regularity.   
\begin{proposition}
\label{prop:subpairs3}
Let $(X,Y)$ be an $(\eps,d,p)$-lower-regular pair in a graph $G$ and let $\hat{X}$ and $\hat Y$ be two subsets of $V(G)$ such that $|X\symd \hat{X}| \leq \mu |X|$ and $|Y \symd \hat Y| \leq \nu |Y|$ for some $0 \leq \mu, \nu \leq 1$. Then $(\hat X, \hat Y)$ is $(\hat \eps, d, p)$-lower-regular, where $\hat \eps := \eps + 2\sqrt{\mu} + 2 \sqrt{\nu}$. Furthermore, if for any disjoint $A,A'\subset V(G)$ with $|A|\ge\mu|X|$ and $|A'|\ge\nu|Y|$ we have $e(A,A')\le (1+\mu+\nu)p|A||A'|$, and $(X,Y)$ is $(\eps,d,p)$-regular, then $(\hat X, \hat Y)$ is $(\hat \eps, d, p)$-regular.
\end{proposition}
We defer the proof of this to Appendix~\ref{app:tools}.

In order to state the sparse regularity lemma, we need some more definitions. A partition $\cV = \{V_i\}_{i\in\{0,\ldots,r\}}$ of the vertex set of $G$ is called an \emph{$(\eps,p)_G$-regular partition} of $V(G)$ if $|V_0|\leq \eps |V(G)|$ and $(V_i,V_{i'})$ forms an $(\eps,0,p)_G$-regular pair for all but at most $\eps\binom{r}{2}$ pairs $\{i,i'\}\in \binom{[r]}{2}$. It is called an \emph{equipartition} if $|V_i| = |V_{i'}|$ for every $i,i'\in[r]$.
The partition $\cV$ is called \emph{$(\eps,d,p)$-(lower-)regular} on a graph $R$ with vertex set $[r]$ if $(V_i, V_{i'})$ is $(\eps,d,p)_G$-(lower-)regular for every $\{i,i'\} \in E(R)$. The graph $R$ is referred to as the \emph{$(\eps,d,p)_G$-reduced graph} of $\cV$, the partition classes $V_i$ with $i \in [r]$ as \emph{clusters}, and $V_0$ as the \emph{exceptional set}. We also say that $\cV$ is \emph{$(\eps,d,p)_G$-super-regular} on a graph $R'$ with vertex set $[r]$ if $(V_i, V_{i'})$ is $(\eps,d,p)_G$-super-regular for every $\{i,i'\}\in E(R')$. Again, when we talk about reduced graphs or super-regularity, whether we are using lower-regularity or regularity will be clear from the context. We will however always specify whether a partition is regular or only lower-regular on $R$.

Analogously to Szemer\'edi's regularity lemma for dense graphs, the sparse regularity lemma, proved by Kohayakawa and Rödl~\cite{kohayakawa1997, kohayakawa2003}, asserts the existence of an $(\eps,p)$-regular partition of constant size of any sparse graph. We state a minimum degree version of this lemma, whose proof (following~\cite{bottcher2013almost}) we defer to Appendix~\ref{app:tools}. 

\begin{lemma}[Minimum degree version of the sparse regularity lemma]
\label{lem:regularitylemma}
For each $\eps >0$, each  $\alpha \in [0,1]$, and $r_0\geq 1$ there exists $r_1\geq 1$ with the following property. For any $d\in[0,1]$, any $p>0$, and any $n$-vertex graph $G$ with minimum degree $\alpha p n$ such that for any disjoint $X,Y\subset V(G)$ with $|X|,|Y|\ge\tfrac{\eps n}{r_1}$ we have $e(X,Y)\le \big(1+\tfrac{1}{1000}\eps^2\big)p|X||Y|$, there is an $(\eps,p)_G$-regular equipartition of $V(G)$ with $(\eps,d,p)_G$-reduced graph $R$ satisfying $\delta(R) \geq (\alpha-d-\eps)|V(R)|$ and $r_0 \leq |V(R)| \leq r_1$.
\end{lemma}

A key ingredient in the proof of our main theorem is the so-called sparse blow up lemma developed by H{\`a}n, Kohayakawa, Person, and two of the current authors in~\cite{blowup}. Given a subgraph $G \subseteq \Gamma =G(n,p)$ with $p \gg (\log n/n)^{1/\Delta}$ and an $n$-vertex graph $H$ with maximum degree at most $\Delta$ with vertex partitions $\cV$ and $\cW$, respectively, the sparse blow up lemma guarantees under certain conditions a spanning embedding of $H$ in $G$ which respects the given partitions. In order to state this lemma we need to introduce some definitions. 

\begin{definition}[$(\vartheta, R')$-buffer]
\label{def:buffer}
Let $R'$ be a graph on $r$ vertices and let $H$ be a graph with vertex partition $\cW=\{W_i\}_{i\in[r]}$. We say
that the family $\tcW=\{\tW_i\}_{i\in[r]}$ of subsets $\tW_i\subseteq W_i$ is an \emph{$(\vartheta,R')$-buffer} for $H$ if
 \begin{enumerate}[label=\rom]
  \item $|\tW_i|\geq\vartheta |W_i|$ for all $i\in[r]$,  and 
  \item for each $i\in[r]$ and each $x\in\tW_i$, the first and second neighbourhood of $x$ go along $R'$, i.e.,\
  for each $\{x,y\},\{y,z\}\in E(H)$ with $y\in W_j$ and $z\in W_k$ we have $\{i,j\}\in E(R')$ and $\{j,k\}\in E(R')$.
 \end{enumerate}
\end{definition}

Let $G$ and $H$ be graphs on $n$ vertices with partitions $\cV=\{V_i\}_{i\in[r]}$ of $V(G)$ and $\cW=\{W_i\}_{i\in[r]}$ of $V(H)$. We say that $\cV$ and $\cW$ are \emph{size-compatible} if  $|V_i|=|W_i|$ for all $i\in[r]$. If there exists an integer $m \geq 1$ such that $m \leq |V_i| \leq \kappa m$ for every $i\in [r]$, then we say that $(G,\cV)$ is $\kappa$-balanced. Given a graph $R$ on $r$ vertices, we call $(G, \cV)$ an \emph{$R$-partition} if for every edge $\{x,y\}\in E(G)$ with $x \in V_i$ and $y\in V_{i'}$ we have $\{i,i'\}\in E(R)$. 

We will actually need a little more than just an embedding of $H$ into $G$ respecting given partitions: we will need to restrict the images of some vertices of $H$ to subsets of the clusters of $G$. The following definition encapsulates the properties we have to guarantee for the sparse blow-up lemma to obtain such an embedding.

\begin{definition}[Restriction pair]
\label{def:restrict} 
 Let $\eps,d>0$, $p \in [0,1]$, and let $R$ be a graph on $r$ vertices. Furthermore, let $G$ be a (not necessarily spanning) subgraph of $\Gamma = G(n,p)$ and let $H$ be a graph given with vertex partitions $\cV= \{V_i\}_{i\in[r]}$ and $\cW = \{W_i\}_{i\in[r]}$, respectively, such that $(G,\cV)$ and $(H,\cW)$ are size-compatible $R$-partitions.
  Let $\cI=\{I_x\}_{x\in V(H)}$ be a collection of subsets of $V(G)$, called
  \emph{image restrictions}, and $\cJ=\{J_x\}_{x\in V(H)}$ be a collection of
  subsets of $V(\Gamma)\setminus V(G)$, called \emph{restricting vertices}.
   For each $i\in [r]$ we define $R_i\subseteq W_i$ to be the set of all vertices $x \in W_i$ for which $I_x \neq V_i$. 
  We say that $\cI$ and $\cJ$ are a
  \emph{$(\rho,\zeta,\Delta,\Delta_J)$-restriction pair} if the
  following properties hold for each $i\in[r]$ and $x\in W_i$.
  \begin{enumerate}[label=\itmarab{RP}]
    \item\label{itm:restrict:numres} We have $|R_i|\leq\rho|W_i|$.
    \item\label{itm:restrict:sizeIx} If $x\in R_i$, then $I_x\subseteq
    \bigcap_{u\in J_x} N_\Gamma(u, V_i)$ is of size at least $\zeta(dp)^{|J_x|}|V_i|$.
    \item\label{itm:restrict:Jx} If $x\in R_i$, then $|J_x|+\deg_H(x)\leq\Delta$ and 
    if $x\in W_i\setminus R_i$, then $J_x=\varnothing$.
    \item\label{itm:restrict:DJ} Each vertex in $V(G)$ appears in at most $\Delta_J$ of the sets of $\cJ$.
    \item\label{itm:restrict:sizeGa} We have
    $\big|\bigcap_{u\in J_x} N_\Gamma(u, V_i)\big| = (p\pm\eps p)^{|J_x|}|V_i|$.
    \item\label{itm:restrict:Ireg} If $x\in R_i$, for each $xy\in E(H)$ with $y\in W_j$, 
    \[\text{the pair }\quad\Big( V_i \cap \bigcap_{u\in J_x}N_\Gamma(u), V_j \cap \bigcap_{v\in J_y}N_\Gamma(v)\Big)\quad\text{ is
    $(\eps,d,p)_G$-lower-regular.}\] 
  \end{enumerate}
\end{definition}
Suppose $\cV$ is an $(\eps,d,p)_G$-lower-regular partition of $V(G)$ with reduced graph $R$, and let $R'$ be a subgraph of $R$. We say $(G,\cV)$ has \emph{one-sided inheritance on $R'$} if for every 
 $\{i,j\}, \{j,k\}\in E(R')$ and every $v\in V_i$ the pair $\big(N_\Gamma(v, V_j),V_k\big)$ is $(\eps,d,p)_G$-lower-regular. Given a $(\vartheta,R')$-buffer $\tcW$, we say that $(G,\cV)$ has \emph{two-sided inheritance on $R'$ for $\tcW$} if whenever there is a triangle $w_iw_jw_k\in H$ with $w_i\in\tW_i$, $w_j\in W_j$ and $w_k\in W_k$, it follows that for every $v\in V_i$ the pair $\big(N_\Gamma(v, V_j),N_\Gamma(v, V_k)\big)$ is $(\eps,d,p)_G$-lower-regular.

Now we can finally state the sparse blow up lemma.

\begin{lemma}[{\cite[Lemma 1.21]{blowup}}]
\label{thm:blowup}
  For each $\Delta$, $\Delta_{R'}$, $\Delta_J$, $\vartheta,\zeta, d>0$, $\kappa>1$
  there exist $\eBL,\rho>0$ such that for all $r_1$ there is a $\CBL$ such that for
  $p\geq\CBL(\log n/n)^{1/\Delta}$ the random graph $\Gamma=G_{n,p}$ asymptotically
  almost surely satisfies the following.
   
  Let $R$ be a graph on $r\le r_1$ vertices and let $R'\subseteq R$ be a spanning
  subgraph with $\Delta(R')\leq \Delta_{R'}$.
  Let $H$ and $G\subseteq \Gamma$ be graphs given with $\kappa$-balanced,
  size-compatible vertex partitions 
  $\cW=\{W_i\}_{i\in[r]}$ and $\cV=\{V_i\}_{i\in[r]}$ with parts of size at
  least $m\geq n/(\kappa r_1)$. 
  Let $\cI=\{I_x\}_{x\in V(H)}$ be a family of image restrictions, and
  $\cJ=\{J_x\}_{x\in  V(H)}$  be a family of restricting vertices.
  Suppose that
  \begin{enumerate}[label=\itmarab{BUL}]
  \item\label{itm:blowup:H} $\Delta(H)\leq \Delta$, 
	for every edge $\{x,y\}\in E(H)$ with $x\in W_i$ and $y\in W_j$ we have $\{i,j\}\in E(R)$ and $\tcW=\{\tW_i\}_{i\in[r]}$ is an
    $(\vartheta,R')$-buffer for $H$,
	\item\label{itm:blowup:G} $\cV$ is $(\eBL,d,p)_G$-lower-regular on $R$, $(\eBL,d,p)_G$-super-regular on $R'$, has one-sided inheritance on $R'$, and two-sided inheritance on $R'$ for $\tcW$,
  \item\label{itm:blowup:restrict} $\cI$ and $\cJ$ form
    a $(\rho,\zeta,\Delta,\Delta_J)$-restriction pair.
  \end{enumerate}
  Then there is an embedding $\phi\colon V(H)\to V(G)$ such that $\phi(x)\in
  I_x$ for each $x\in H$.
\end{lemma}	

Observe that in the blow up lemma for dense graphs, proved by Koml{\'o}s, S{\'a}rk{\"o}zy, and Szemer{\'e}di~\cite{komlos1997blow}, one does not need to explicitly ask for one- and two-sided inheritance properties since they are always fulfilled by dense regular partitions. This is, however, not true in general in the sparse setting. The following two lemmas will be very useful whenever we need to choose vertices whose neighbourhoods inherit lower-regularity.

\begin{lemma}[One-sided lower-regularity inheritance,~\cite{blowup}]
\label{lem:OSRIL}
For each $\eo, \ao >0$ there exist $\eps_0 >0$ and $C >0$ such that for any $0 < \eps < \eps_0$ and $0 < p <1$  asymptotically almost surely $\Gamma= G(n,p)$ has the following property. For any disjoint sets $X$ and $Y$ in $V(\Gamma)$ with $|X|\geq C\max\big(p^{-2}, p^{-1} \log n\big)$ and $|Y| \geq C p^{-1} \log n$, and any subgraph $G$ of $\Gamma[X,Y]$ which is $(\eps, \ao,p)_G$-lower-regular, there are at most $C p^{-1}\log (en/|X|)$ vertices $z \in V(\Gamma)$ such that $(X \cap N_{\Gamma}(z),Y)$ is not $(\eo,\ao,p)_G$-lower-regular.
\end{lemma}

\begin{lemma}[Two-sided lower-regularity inheritance,~\cite{blowup}]
\label{lem:TSRIL}
For each $\et,\at>0$ there exist $\eps_0>0$ and
$C >0$ such that for any $0<\eps<\eps_0$ and $0<p<1$, asymptotically almost surely
$\Gamma=G_{n,p}$ has the following property. For any disjoint sets $X$
and $Y$ in $V(\Gamma)$ with $|X|,|Y|\ge C\max\{p^{-2},p^{-1}\log n\}$, and any
subgraph $G$ of $\Gamma[X,Y]$ which is $(\eps,\at,p)_G$-lower-regular, there are
at most $C\max\{p^{-2},p^{-1}\log (en/|X|)\}$ vertices $z \in V(\Gamma)$
such that $\big(X\cap N_\Gamma(z),Y\cap N_\Gamma(z)\big)$ is not
$(\et,\at,p)_G$-lower-regular.
\end{lemma}

We close this section with some probabilistic tools.
We start with the following useful observation. Roughly speaking, it states that a.a.s.~nearly all vertices in $G(n,p)$ have approximately the expected number of neighbours within large enough subsets. 

\begin{proposition}[]
\label{prop:chernoff}
For each $\eps>0$ there exists a constant $C >0$ such that for every $0<p<1$ asymptotically almost surely $\Gamma=G(n,p)$ has the following properties. For any disjoint $X,Y\subset V(\Gamma)$ with $|X|\ge Cp^{-1}\log n$ and $|Y|\ge Cp^{-1}\log (en/|X|)$, we have $e(X,Y)=(1\pm\eps)p|X||Y|$ and $e(X)\le 2p|X|^2$. Furthermore, for every $X \subseteq V(\Gamma)$ with $|X| \geq C p^{-1} \log n$, the number of vertices $v \in V(\Gamma)$ with $\big||\NGa(v,X)| - p |X|\big| > \eps p |X|$ is at most $C p^{-1} \log (en/|X|)$.
\end{proposition}

Note that in most of this paper we will use the upper bound $\log(en/|X|)\le\log n$ when applying this proposition, and Lemmas~\ref{lem:OSRIL} and~\ref{lem:TSRIL}, valid since (in all applications) we have $|X|\ge e$. We will only need the full strength of these three results when proving the Lemma for $G$ (Lemma~\ref{lem:G}).

In the proof of Proposition~\ref{prop:chernoff} we use the following version of Chernoff's Inequalities (see e.g.~\cite[Chapter~2]{janson2011random} for a proof).

\begin{theorem}[Chernoff's Inequality,~\cite{janson2011random}]
\label{thm:chernoff}
Let $X$ be a random variable which is the sum of independent Bernoulli random variables. Then we have for $\eps\leq 3/2$
\[\Pr\big[|X-\Ex[X]| > \eps \Ex[X]\big] < 2e^{-\eps^2\Ex[X]/3}\,.\]
Furthermore, if $t\ge 6\Ex[X]$ then we have
\[\Pr\big[X\ge\Ex[X]+t\big]\le e^{-t}\,.\]
\end{theorem}

\begin{proof}[Proof of Proposition~\ref{prop:chernoff}]
Since the statement of the proposition is stronger when $\eps$ is smaller, we may assume that $0<\eps\le 1$. We set $C'=100\eps^{-2}$ and $C=1000C'\eps^{-1}$.

We first show that $\Gamma=G(n,p)$ a.a.s.\ has the following two properties. For any disjoint $A,B\subset V(\Gamma)$, with $|A|\ge C'p^{-1}\log n$ and $|B|\ge C'p^{-1}\log(en/|A|)$, we have $e(A,B)=\big(1\pm\tfrac{\eps}{2}\big)p|A||B|$. For any $A\subset V(\Gamma)$, we have $e(A)\le 4p|A|^2+2|A|\log n$, and if $|A|\ge C'p^{-1}\log n$ then $e(A)\le 2p|A|^2$. Note that these properties imply the first two conclusions of the proposition.

We estimate the failure probability of the first property using Theorem~\ref{thm:chernoff} and the union bound. Assuming without loss of generality that $|A|\ge|B|$, this probability is at most
\begin{align*}
 \sum_{|A|,|B|\le n} \binom{n}{|A|}^2\cdot 2e^{-\eps^2p|A||B|/12}&\le 2n\sum_{|A|}\Big(\frac{en}{|A|}\Big)^{2|A|}e^{-\eps^2C'|A|\log (en/|A|)/12}\\
 &<2n\sum_{|A|}\Big(\frac{en}{|A|}\Big)^{-2|A|}\,.
\end{align*}

For the second property, observe that $4p|A|^2>7p\binom{|A|}{2}$, so that for any given $A$ by Theorem~\ref{thm:chernoff} we have
\[\Pr\big[e(A)\ge 4p|A|^2+2|A|\log n\big]\le e^{-2|A|\log n}=n^{-2|A|}\,.\]
Taking a union bound over the at most $n^{|A|}$ choices of $A$ given $|A|$, we see that the failure probability of the second property is at most $\sum_{a=1}^nn^{-a}$.

Finally, the failure probability of the last property is at most
\[\sum_{|A|\ge C'p^{-1}\log n}n^{|A|}\cdot 2e^{-p\binom{|A|}{2}/3}\le\sum_{|A|}2n^{|A|}e^{-C'|A|\log n/12}\le 2n^{-2}\,,\]
and since all three failure probabilities tend to zero as $n\to\infty$, we conclude that a.a.s.\ $G(n,p)$ enjoys both properties.

Now suppose $\Gamma$ has these properties, and let $X\subset V(\Gamma)$ have size at least $Cp^{-1}\log n$. We first show that there are at most $C'p^{-1}\log (en/|X|)$ vertices in $\Gamma$ which have less than $(1-\eps)p|X|$ neighbours in $X$. If this were false, then we could choose a set $Y$ of $C'p^{-1}\log (en/|X|)$ vertices in $\Gamma$ which have less than $(1-\eps)p|X|$ neighbours in $X$. By choice of $C$ and since $|X|>e$, we have $(1-\eps)p|X|\le \big(1-\tfrac\eps2\big)p|X\setminus Y|$, so we see that $e(Y,X\setminus Y)<\big(1-\tfrac{\eps}{2}\big)p|Y||X\setminus Y|$. This is a contradiction since $|X\setminus Y|\ge C'p^{-1}\log n$.

Next we show that there are at most $2C'p^{-1}\log (en/|X|)$ vertices of $\Gamma$ which have more than $(1+\eps)p|X|$ neighbours in $X$. Again, if this is not the case we can let $Y$ be a set of $2C'p^{-1}\log (en/|X|)$ vertices of $\Gamma$ with more than $(1+\eps)p|X|$ neighbours in $X$. Now $e(Y)\le 4p|Y|^2+2|Y|\log n=8C'|Y|\log(en/|X|)+2|Y|\log n\le 10C'|Y|\log n$, so there are at most $|Y|/2$ vertices in $Y$ which have $40C'\log n$ or more neighbours in $Y$. Let $Y'\subset Y$ consist of those vertices with at most $40C'\log n$ neighbours in $Y$. For each $v\in Y'$ we have
\[(1+\eps)p|X|\le\deg(v;X)\le \deg(v;Y)+\deg(v;X\setminus Y)\,,\]
and so, by choice of $C$, each vertex of $Y'$ has at least $\big(1+\tfrac\eps2\big)p|X\setminus Y|$ neighbours in $X\setminus Y$. Since $|Y'|\ge C'p^{-1}\log(2en/|X|)$ and $|X\setminus Y|\ge |X|/2\ge C'p^{-1}\log n$, this is a contradiction. Finally, since by choice of $C$ we have $3C'p^{-1}\log n<Cp^{-1}\log n$ we conclude that all but at most $Cp^{-1}\log (en/|X|)$ vertices of $\Gamma$ have $(1\pm\eps)p|X|$ neighbours in $X$, as desired.
\end{proof}

Now let $N$, $m$, and $s$ be positive integers and let $S$ and $S' \subseteq S$ be two sets with $|S| = N$ and $|S'| = m$. The \emph{hypergeometric distribution} is the distribution of the random variable $X$ that is defined by drawing $s$ elements of $S$ without replacement and counting how many of them belong to $S'$. It can be shown that Theorem~\ref{thm:chernoff} still holds in the case of hypergeometric distributions (see e.g.~\cite{janson2011random}, Chapter~2 for a proof) with $\Ex[X]:= ms/N$.  

\begin{theorem}[Hypergeometric inequality,~\cite{janson2011random}]
\label{thm:hypergeometric}
Let $X$ be a random variable that follows the hypergeometric distribution with parameters $N$, $m$, and $s$. Then for any $\eps>0$ and $t\ge\eps ms/N$ we have
\[\Pr\big[|X - ms/N| > t \big] < 2e^{-\eps^2t/3}\,.\]
\end{theorem}

We require the following technical lemma, which is a consequence of the hypergeometric inequality stated in Theorem~\ref{thm:hypergeometric}. 

\begin{lemma}\label{lem:hypgeo}
 For each $\eta>0$ and $\Delta$ there exists $C$ such that the following holds. Let $W\subset [n]$, let $t\le 100n^\Delta$, and let $T_1,\ldots,T_t$ be subsets of $W$. For each $m\le |X|$ there is a set $S\subset W$ of size $m$ such that
 \[|T_i\cap S|=\frac{m}{|W|}|T_i|\pm \big(\eta|T_i|+C\log n\big)\text{ for every }i\in[t]\,.\]
\end{lemma}
\begin{proof}
 Set $C=30\eta^{-2}\Delta$. Observe that for each $i$, the size of $T_i\cap S$ is hypergeometrically distributed.
 By Theorem~\ref{thm:hypergeometric}, for each $i$ we have
 \[\Pr\Big[|T_i\cap S|\neq \frac{m}{|W|}|T_i|\pm \big(\eta|T_i|+C\log n\big)\Big]<2e^{-\eta^2C\log n/3}<\frac{2}{n^{1+\Delta}}\,,\]
 so taking the union bound over all $i\in[t]$ we conclude that the probability of failure is at most $2t/n^{1+\Delta}\le 200/n\to 0$ as $n\to\infty$, as desired.
\end{proof}

We shall also use McDiarmid's Inequality.

\begin{lemma}[McDiarmid's Inequality~\cite{McDiarmid}]\label{lem:McDiarmid}
Let $X_1, \ldots, X_k$ be independent random variables, where $X_i$ takes values in a finite set $A_i$ for each $i\in [k]$. Suppose that a function $g: A_1 \times \ldots \times A_k \to \mathbb R$ satisfies for each $i\in [k]$ 
$$\sup_{x_1,\ldots, x_k, \hat x_i}|g(x_1,x_2,\ldots,x_k)-g(x_1,x_2,\ldots, x_{i-1}, \hat{x}_i, x_{i+1}, \ldots, x_k)| \leq c_i.$$ 
Then, for any $\eps >0$, we have 
$$\Pr\big[|\Ex[g(X_1,\ldots, X_k)]- g(X_1,\ldots, X_k)| \geq \eps \big]\leq 2\exp\left\{-\frac{2\eps^2}{\sum_{i\in[k]} c_i^2}\right\}\,.$$
\end{lemma}

%%%%%%%%%%%%%%%%%%%%%%%%%%%%%%%%%%%%%%%%%%%%%%%%%%%%%%%%%

% Main Lemmas

%%%%%%%%%%%%%%%%%%%%%%%%%%%%%%%%%%%%%%%%%%%%%%%%%%%%%%%%%

\section{Proof overview and main lemmas}
\label{sec:mainlemmas}

Theorem~\ref{thm:main} is a corollary of the following more general Theorem~\ref{thm:maink}, which we prove in Section~\ref{sec:proofmain}. We require one preliminary definition.
\begin{definition}[Zero-free colouring]\label{def:zerofree}
  Let $H$ be a $(k+1)$-colourable graph on $n$ vertices and let $\cL$ be a
  labelling of its vertex set by integers $1, \ldots, n$ of bandwidth at most
  $\beta n$. A proper $(k+1)$-colouring $\sigma:V(H) \to \{0,\ldots,k\}$ of its
  vertex set is said to be \emph{$(z,\beta)$-zero-free} with respect to $\cL$ if
  any $z$ consecutive blocks contain at most one block with colour zero, where a
  block is defined as a set of the form $\{(t-1)4k\beta n +1, \ldots, t4k\beta
  n\}$ with some $t \in [1/(4k\beta)]$,
  and a block
  with colour zero is a block in which at least one vertex is coloured with zero.
\end{definition}

\begin{theorem}
\label{thm:maink}
For each $\gamma>0$, $\Delta \geq 2$, and $k\geq 2$, there exist constants $\beta >0$, $z>0$, and $C>0$ such that the following holds asymptotically almost surely for $\Gamma = G(n,p)$ if $p\geq C\left(\frac{\log n}{n}\right)^{1/\Delta}$. Let $G$ be a spanning subgraph of $\Gamma$ with $\delta(G) \geq\left(\frac{k-1}{k}+\gamma\right)pn$ and let $H$ be a graph on $n$ vertices with $\Delta(H) \leq \Delta$ that has a labelling $\cL$ of its vertex set of bandwidth at most $\beta n$, a $(k+1)$-colouring that is $(z,\beta)$-zero-free with respect to $\cL$ and where the first $\sqrt{\beta} n$ vertices in $\cL$ are not given colour zero and the first $\beta n$ vertices in $\cL$ include $C p^{-2}$ vertices that are not contained in any triangles of $H$. Then $G$ contains a copy of $H$.
\end{theorem}

\subsection{Proof overview}\label{subsec:over}
We now give a brief sketch of the proof of Theorem~\ref{thm:maink}. Ultimately,
our goal is to apply the sparse blow-up lemma, Lemma~\ref{thm:blowup}, to find
an embedding of $H$ into $G$. Thus, the proof boils down to obtaining the
required conditions. But there is a catch: this is not as such possible, as for
any lower-regular partition of $G$ there can be $O(p^{-2})$ exceptional vertices which
are `badly behaved' with respect to the partition. These vertices will never
satisfy the conditions of the sparse blow-up lemma, and we will have to deal
with them beforehand. We will do this by `pre-embedding' some vertices of $H$ to
cover the exceptional vertices, and then apply the sparse blow-up lemma to
complete the embedding of $H$ into $G$, using image restrictions to ensure we
really obtain an embedding of $H$. Let us now fill in a few more details.
 
We start by obtaining, in the lemma for $G$, Lemma~\ref{lem:G}, a lower-regular
partition of $G$ into parts $V_0$ and $V_{i,j}$ for $i\in[r]$ (where $r$ may be
large but is bounded above by a constant) and $j\in [k]$ with several extra
properties. The most important properties are that $|V_0|=O\big(p^{-2}\big)$,
that the corresponding reduced graph, which we call $R^k_r$, on the vertex set
$[r]\times[k]$ has high minimum degree and contains a so-called backbone graph
on which the partition does not merely provide lower-regular pairs but
super-regular pairs, and that all vertices outside $V_0$ have inheritance
properties with respect to all lower-regular pairs. Here, a \emph{backbone
  graph} has all edges between $(i,j)$ and $(i',j')$ with $|i-i'|\le 1$ and $j\neq j'$. One
should think of the backbone graph as consisting of copies of $K_k$ (one for
each $i\in[r]$) connected in a linear order; and the high minimum degree of
$R^k_r$ ensures that each $K_k$ extends to $K_{k+1}$ in $R^k_r$. In short, if
the exceptional vertices $V_0$ did not exist, this partition, together with a
corresponding partition of $V(H)$, would be what we need to apply the sparse
blow-up lemma.
 
Passing over for now the inconvenient existence of $V_0$, our next task is to
find the corresponding partition of $V(H)$, for which we use the lemma for $H$,
Lemma~\ref{lem:H2}. The basic idea is then to split $H$ into intervals in the bandwidth
order. We assign the first interval to the first $K_k$ of the backbone graph
according to the given colouring of $H$, with the few vertices of colour zero
assigned to a vertex extending this clique of the backbone graph to $K_{k+1}$,
and so on. Using the bandwidth property and zero-freeness of the colouring one
can do this in such a way as to obtain a graph homomorphism from $H$ to $R^k_r$,
which is what we need. In addition, we need the number of vertices assigned to
each $(i,j)\in V(R^k_r)$ to be very close to $|V_{i,j}|$. We cannot guarantee
exact equality, but we can get very close by making further use of bandwidth,
zero-freeness, and the fact that $K_k$s in $R^k_r$ extend to $K_{k+1}$s.
 
 Now we have to deal with the exceptional set $V_0$. We do this as follows. We choose a vertex $v$ in the exceptional set, and `pre-embed' to it a vertex $x$ picked from the first $\sqrt{\beta}n$ vertices of $\cL$ which is not in any triangle of $H$. Using the common neighbourhood lemma, Lemma~\ref{lem:common}, we choose $\Delta$ neighbours of $v$ which are `well-behaved' with respect to the clusters $V_{i,j}$ for some $i\in[r]$, and pre-embed the neighbours of $x$ to these vertices. The `well-behaved' properties are what we need to generate image restrictions for the second neighbours of $x$ (which we will embed using the sparse blow-up lemma) satisfying the restriction pair properties. We also need to change the assignment from the Lemma for $H$ locally (up to a large but constant distance from $x$) to accommodate this: the vertex $x$, and its first and second neighbours, might have been assigned somewhere quite different previously. We repeat this until we have pre-embedded to all exceptional vertices, and let $H'$ and $G'$ be respectively the unembedded vertices of $H$ and the vertices of $G$ to which we did not pre-embed.
 
 At this point we have all the conditions we need to apply the sparse blow-up lemma to complete the embedding, except that the partitions of $H'$ and $G'$ we have do not quite have parts of matching sizes. We use the balancing lemma, Lemma~\ref{lem:balancing}, to deal with this. The idea is simple: we take some carefully selected vertices in clusters of $G$ which are too big (compared to the assigned part of $H$) and move them to other clusters, first in order to make sure that the total number of vertices in $\bigcup_i V_{i,j}$ is correct for each $j$ (using the high minimum degree of $R^k_r$) and then (using the structure of the backbone graph) to give each cluster the correct size.
 
 At last, applying the sparse blow-up lemma, Lemma~\ref{thm:blowup}, we complete the embedding of $H$ into $G$.

 We note that this proof sketch glosses over some subtleties. In particular, at the two places where `we choose' vertices onto which to pre-embed, we have to be quite careful to choose vertices correctly so that this strategy can be completed and we do not destroy good properties obtained earlier. We will return to this point immediately before the proof of Theorem~\ref{thm:maink} in Section~\ref{sec:proofmain} to explain how we do this.
 
\subsection{Main lemmas}

In this subsection we formulate the four main lemmas that we use in the proof of Theorem~\ref{thm:maink} mentioned in the above overview. We defer the proofs of these lemmas to later sections. Before stating these lemmas, we need some more definitions. 

Let $r, k \geq 1$ and let $B^k_r$ be the backbone graph on $kr$ vertices. That is, we have
\[V(B^k_r) := [r] \times [k]\] and for every $j \neq j' \in [k]$ we have $\{(i,j),(i',j')\} \in E(B^k_r)$ if and only if $|i-i'|\le1$.

Let $K^k_r \subseteq B^k_r$ be the spanning subgraph of $B^k_r$ that is the disjoint union of $r$ complete graphs on $k$ vertices given by the following components: the complete graph $K^k_r[\{(i,1),\ldots, (i,k)\}]$ is called the \emph{$i$-th component} of $K^k_r$ for each $i\in [r]$. 

A vertex partition $\cV' = \{V_{i,j}\}_{i\in[r],j\in[k]}$ is called \emph{$k$-equitable} if $\big||V_{i,j}| -|V_{i,j'}|\big|\leq 1$ for every $i\in [r]$ and $j,j'\in[k]$. Similarly, an integer partition $\{n_{i,j}\}_{i\in[r],j\in[k]}$ of $n$ (meaning that $n_{i,j} \in \mathbb Z_{\geq 0}$ for every $i\in [r],j\in[k]$ and $\sum_{i\in[r]j\in[k]} n_{i,j} = n$) is \emph{$k$-equitable} if $|n_{i,j}-n_{i,j'}| \leq 1$ for every $i\in[r]$ and $j,j'\in[k]$. 

The lemma for $G$ says that a.a.s.~$\Gamma = G(n,p)$ satisfies the following property if $p \gg (\log n/n)^{1/2}$. For any spanning subgraph $G\subset\Gamma$ with minimum degree a sufficiently large fraction of $pn$, there exists an $(\eps,d,p)_G$-lower-regular vertex partition $\cV$ of $V(G)$ whose reduced graph $R^k_r$ contains a clique factor $K^k_r$ on which the corresponding vertex sets of $\cV$ are pairwise $(\eps,d,p)$-super-regular. Furthermore, $(G,\cV)$ has one-sided and two-sided inheritance with respect to $R^k_r$, and the $\Gamma$-neighbourhoods of all vertices but the ones in the exceptional set of $\cV$ have almost exactly their expected size in each cluster.
The proof of Lemma~\ref{lem:G} is given in Section~\ref{sec:prooflemG}.

\begin{lemma}[Lemma for $G$] 
\label{lem:G}
For each $\gamma > 0$ and integers $k \geq 2$ and $r_0 \geq 1$ there exists $d > 0$ such that for every $\eps \in \left(0, \frac{1}{2k}\right)$ there exist $r_1\geq 1$ and $\Ca>0$ such that the following holds a.a.s.~for $\Gamma = G(n,p)$ if $p \geq \Ca \left(\log n/n\right)^{1/2}$. Let $G=(V,E)$ be a spanning subgraph of $\Gamma$ with $\delta(G) \geq  \left(\frac{k-1}{k} + \gamma\right)pn$. Then there exists an integer $r$ with $r_0\leq kr \leq r_1$, a subset $V_0 \subseteq V$  with $|V_0| \leq \Ca p^{-2}$, 
a $k$-equitable vertex partition $\cV = \{\Vij\}_{i\in[r],j\in[k]}$ of $V(G)\setminus V_0$,
and a graph $R^k_r$ on the vertex set $[r] \times [k]$ with $K^k_r \subseteq B^k_r \subseteq R^k_r$, with $\delta(R^k_r) \geq \left(\frac{k-1}{k} + \frac{\gamma}{2}\right)kr$, and such that the following are true.
\begin{enumerate}[label=\itmarab{G}]
\item \label{lemG:size} $\frac{n}{4kr}\leq |\Vij| \leq \frac{4n}{kr}$ for every $i\in[r]$ and $j\in[k]$,
\item \label{lemG:regular} $\cV$ is $(\eps,d,p)_G$-lower-regular on $R^k_r$ and $(\eps,d,p)_G$-super-regular on $K^k_r$,
\item \label{lemG:inheritance} both $\big(\NGa(v, V_{i,j}),V_{i',j'}\big)$ and $\big(\NGa(v, V_{i,j}),\NGa(v, V_{i',j'})\big)$ are $(\eps,d,p)_G$-lower-regular pairs for every $\{(i,j),(i',j')\} \in E(R^k_r)$ and $v\in V\setminus V_0$,
\item \label{lemG:gamma} $|\NGa(v,V_{i,j})| = (1 \pm \eps)p|V_{i,j}|$ for every $i \in [r]$, $j\in [k]$ and every $v \in V \setminus V_0$.
\end{enumerate}
Furthermore, if we replace~\ref{lemG:inheritance} with the weaker
\begin{enumerate}[label=\itmsol{G}{3'}]
 \item \label{lemG:inheritancep} $\big(\NGa(v, V_{i,j}),V_{i',j'}\big)$ is an $(\eps,d,p)_G$-lower-regular pair for every $\{(i,j),(i',j')\} \in E(R^k_r)$ and $v\in V\setminus V_0$,
\end{enumerate}
then we have the stronger bound $|V_0|\le\Ca p^{-1}$.
\end{lemma}

After Lemma~\ref{lem:G} has constructed a lower-regular partition $\cV$ of $V(G)$, the second main lemma deals with the graph $H$ that we would like to find as a subgraph of $G$. More precisely, Lemma~\ref{lem:H2} provides a homomorphism $f$ from the graph $H$ to the reduced graph $R^k_r$ given by Lemma~\ref{lem:G} which has among others the following properties. The edges of $H$ are mapped to the edges of $R^k_r$, and the vast majority of the edges of $H$ are assigned to edges of the clique factor $K^k_r \subseteq R^k_r$. The number of vertices of $H$ mapped to a vertex of $R^k_r$ only differs slightly from the size of the corresponding cluster of $\cV$. The lemma further guarantees that each of the first $\sqrt{\beta}n$ vertices of the bandwidth ordering of $V(H)$ is mapped to $(1,j)$ with $j$ being the colour that the vertex has received by the given colouring of $H$. In case $H$ is $D$-degenerate the next lemma also ensures that for every $(i,j) \in [r] \times [k]$, a constant fraction of vertices mapped to $(i,j)$ have each at most $2D$ neighbours. 

\begin{lemma}[Lemma for~$H$]\label{lem:H2}
Given $D, k, r \geq 1$ and $\xi, \beta > 0 $ the following holds if $\xi \leq 1/(kr)$ and $\beta \leq 10^{-10}\xi^2/(D k^4r)$. Let $H$ be a $D$-degenerate graph on $n$ vertices, let $\mathcal L$ be a labelling of its vertex set of bandwidth at most $\beta n$ and let $\sigma: V(H) \to \{0,\ldots k\}$ be a proper $(k+1)$-colouring that is $(10/\xi, \beta)$-zero-free with respect to $\mathcal L$, where the colour zero does not appear in the first $\sqrt{\beta}n$ vertices of $\mathcal{L}$. Furthermore, let $R^k_r$ be a graph on vertex set $[r] \times [k]$ with $K^k_r \subseteq B^k_r \subseteq R^k_r$ such that for every $i\in [r]$ there exists a vertex $z_i \in \big([r]\setminus\{i\}\big) \times [k]$ with $\big\{z_i, (i,j)\big\} \in E(R^k_r)$ for every $j\in [k]$. Then, given a $k$-equitable integer partition $\{m_{i,j}\}_{i\in[r],j\in[k]}$ of $n$ with $n/(10kr) \leq m_{i,j} \leq 10n/(kr)$ for every $i \in[r]$ and $j\in [k]$, there exists a mapping $f \colon V(H) \to [r]\times[k]$ and a set of special vertices $X \subseteq V(H)$ such that we have for every $i\in [r]$ and $j\in[k]$
%with the following properties, where $W_{i,j} := f^{-1}(i,j)$. 

\begin{enumerate}[label=\itmarab{H}]
\item\label{lemH:H1} $m_{i,j} - \xi n  \leq |f^{-1}(i,j)| \leq m_{i,j} + \xi n$,
\item\label{lemH:H2} $|X| \leq \xi n$,
\item\label{lemH:H3} $\{f(x),f(y)\} \in E(R^k_r)$  for every $\{x,y\} \in E(H)$,
\item\label{lemH:H4} $y,z\in \cup_{j'\in[k]}f^{-1}(i,j')$ for every $x\in f^{-1}(i,j)\setminus X$ and $xy,yz\in E(H)$,
\item\label{lemH:H5} $f(x) = \big(1, \sigma(x)\big)$ for every $x$ in the first $\sqrt{\beta}n$ vertices of $\mathcal{L}$, and
\item\label{lemH:H6} $|\{x\in f^{-1}(i,j): \deg(x) \leq 2D\}| \geq \tfrac{1}{24D} |f^{-1}(i,j)|$. 
\end{enumerate}
\end{lemma}

Lemma~\ref{lem:H2} is a strengthened version of~\cite[Lemma~8]{BoeTarWue}. The proof of~\cite[Lemma~8]{BoeTarWue} is  deterministic; here we use a probabilistic argument to show the existence of a function $f$ that also satisfies the additional property~\ref{lemH:H6}, which is required for Theorem~\ref{thm:degenerate}. However, we still borrow ideas from the proof of~\cite[Lemma~8]{BoeTarWue}. The proof of Lemma~\ref{lem:H2} will be given in Section~\ref{sec:lemH}. 

 During the pre-embeding, we embed a vertex $x$ of $H$ onto a vertex $v$ of $V_0$, and we also embed its neighbours $N_H(x)$. This creates restrictions on the vertices of $G$ to which we can embed the second neighbours, and for application of Lemma~\ref{thm:blowup} we need certain conditions to be satisfied. The next lemma states that we can find vertices in $N_G(v)$, to which we will embed $N_H(x)$, satisfying these conditions.

\begin{lemma}[Common neighbourhood lemma]
\label{lem:common}
For each $d>0$, $k \geq 2$, and $\Delta \geq 2$ there exists $\alpha >0$ such that for every $\eps^\ast \in (0,1)$ there exists $\eps_0 >0$ such that for every $r\geq 1$ and every $0<\eps\le\eps_0$ there exists $\Ca >0$ such that the following is true. If $p \geq \Ca \left(\log n/n\right)^{1/\Delta}$, then $\Gamma = G(n,p)$ a.a.s.~satisfies the following. Let $G=(V,E)$ be a (not necessarily spanning) subgraph of $\Gamma$ and $\{V_i\}_{i\in[k]}\cup \{W\}$ a vertex partition of a subset of $V$ such that the following are true for every $i,i'\in [k]$.
\begin{enumerate}[label=\itmarab{V}]
 \item\label{cnl:bal} $\frac{n}{4kr}\le |V_i|\le \frac{4n}{kr}$,
 \item\label{cnl:Vreg} $(V_i,V_{i'})$ is $(\eps, d, p)_G$-lower-regular,
 \item\label{cnl:W} $|W|\ge 10^{-10}\frac{\eps^4 pn}{k^4r^4}$, and
 \item\label{cnl:Wdeg} $|N_G(w,V_i)| \geq dp|V_i|$ for every $w \in W$. 
\end{enumerate}
Then there exists a tuple $(w_1, \ldots, w_\Delta) \in \binom{W}{\Delta}$ such that for every $\Lambda,\Lambda^\ast\subseteq[\Delta]$, and every $i \neq i' \in [k]$ we have
\begin{enumerate}[label=\itmarab{W}]
 \item\label{cnl:Gsize} $|\bigcap_{j\in \Lambda} N_G(w_j,V_i)|\geq \alpha p^{|\Lambda|}|V_i|$,
 \item\label{cnl:Gasizen} $|\bigcap_{j\in \Lambda} N_{\Gamma}(w_j)| \le (1 + \eps^\ast)p^{|\Lambda|}n$,
 \item\label{cnl:Gasize} $|\bigcap_{j\in \Lambda} N_{\Gamma}(w_j,V_i)| = (1 \pm \eps^\ast)p^{|\Lambda|}|V_i|$, and
 \item\label{cnl:Nreg} $\big(\bigcap_{j\in \Lambda}\NGa(w_j,V_i),\bigcap_{j^\ast\in \Lambda^\ast}\NGa(w_{j^\ast},V_{i'})\big)$ is $(\eps^\ast, d,p)_G$-lower-regular if $|\Lambda|,|\Lambda^\ast| < \Delta$ and either $\Lambda\cap\Lambda^\ast=\varnothing$ or $\Delta\geq 3$ or both.
\end{enumerate} 
\end{lemma}

Let $H'$ and $G'$ denote the subgraphs of $H$ and $G$ that result from removing all vertices that were used in the pre-embedding process. As a last step before finally applying the sparse blow-up lemma, the clusters in $\restr{\cV}{G'}$ need to be adjusted to the sizes of $\restr{W_{i,j}}{H'}$. The next lemma states that this is possible, and that after this redistribution the regularity properties needed for Lemma~\ref{thm:blowup} still hold.

\begin{lemma}[Balancing lemma]
\label{lem:balancing}
For all integers $k\geq 1$, $r_1, \Delta \geq 1$, and reals $\gamma, d >0$ and $0 < \eps < \min\{d,1/(2k)\}$ there exist $\xi >0$ and $\Ca >0$ 
such that the following is true for every $p \geq \Ca \left(\log n/n\right)^{1/2}$ and every $10\gamma^{-1}\le r \leq r_1$ provided that $n$ is large enough. Let $\Gamma$ be a graph on the vertex set $[n]$ and let $G=(V,E)\subseteq \Gamma$ be a (not necessarily spanning) subgraph with vertex partition $\cV = \{\Vij\}_{i\in[r],j\in[k]}$ that satisfies $n/(8kr) \leq |\Vij| \leq 4n/(kr)$ for each $i\in[r]$, $j\in[k]$. 
Let $\{n_{i,j}\}_{i \in [r], j\in [k]}$ be an integer partition of $\sum_{i\in[r],j\in[k]} |V_{i,j}|$. Let $R^k_r$ be a graph on the vertex set $[r] \times [k]$ with minimum degree $\delta(R^k_r) \geq \big((k-1)/k+\gamma/2\big) kr$ such that $K^k_r \subseteq B^k_r \subseteq R^k_r$.
Suppose that the partition $\cV$ satisfies the following properties for each $i\in[r]$, each $j\neq j'\in[k]$, and each $v\in V$. 
\begin{enumerate}[label=\itmarab{B}]
\item \label{lembalancing:sizes} We have $n_{i,j} -  \xi n \leq |V_{i,j}| \leq n_{i,j} +  \xi n$,
\item \label{lembalancing:regular1} $\cV$ is $\big(\tfrac{\eps}{4},d,p\big)_G$-lower-regular on $R^k_r$ and $\big(\tfrac{\eps}{4},d,p\big)_G$-super-regular on $K^k_r$,
\item \label{lembalancing:inheritance1}  both $\big(\NGa(v, V_{i,j}),V_{i,j'}\big)$ and $\big(\NGa(v, V_{i,j}),\NGa(v, V_{i,j'})\big)$ are $\big(\tfrac{\eps}{4},d,p\big)_G$-lower-regular pairs, and
\item \label{lembalancing:gamma1} we have $|\NGa(v,V_{i,j})| = \big(1 \pm \tfrac{\eps}{4}\big)p|\Vij|$.
\end{enumerate}
Then, there exists a partition $\mathcal{V'}= \{V'_{i,j}\}_{i\in[r],j\in[k]}$ of $V$ such that the following properties hold for each $i\in[r]$, each $j\neq j'\in [k]$, and each $v\in V$.
\begin{enumerate}[label=\itmarabp{B}{'}]
\item\label{lembalancing:sizesout} We have $|V'_{i,j}|=n_{i,j}$,
\item\label{lembalancing:symd} We have $|V_{i,j}\symd V'_{i,j}|\le 10^{-10}\eps^4k^{-2}r_1^{-2} n$,
\item \label{lembalancing:regular} $\mathcal{V'}$ is $(\eps,d,p)_G$-lower-regular on $R^k_r$ and $(\eps,d,p)_G$-super-regular on $K^k_r$,
\item \label{lembalancing:inheritance} both $\big(N_{\Gamma}(v,V'_{i,j}), V'_{i,j'}\big)$ and $\big(N_{\Gamma}(v,V'_{i,j}), N_\Gamma(v,V'_{i,j'})\big)$ are $(\eps,d,p)_G$-lower-regular pairs, and
\item\label{lembalancing:gammaout} For each $1\le s\le\Delta$ and vertices $v_1,\ldots,v_s\in[n]$ we have 
\[\big|N_\Gamma(v_1,\ldots,v_s;V_{i,j})\symd N_\Gamma(v_1,\ldots,v_s;V'_{i,j})\big|\le 10^{-10}\eps^4k^{-2}r_1^{-2}\deg_\Gamma(v_1,\ldots,v_s)+\Ca\log n\,.\]
\end{enumerate} 
Furthermore, if for any two disjoint vertex sets $A,A'\subset V(\Gamma)$ with $|A|,|A'|\ge\tfrac{1}{50000kr_1}\eps^2\xi pn$ we have $e_\Gamma(A,A')\le \big(1+\tfrac{1}{100}\eps^2\xi\big)p|A||A'|$, and if `lower-regular' is replaced with `regular' in~\ref{lembalancing:regular1}, and~\ref{lembalancing:inheritance1}, then we can replace `lower-regular' with `regular' in~\ref{lembalancing:regular} and~\ref{lembalancing:inheritance}.
\end{lemma} 

%%%%%%%%%%%%%%%%%%%%%%%%%%%%%%%%%%%%%%%%%%%%%%%%%%%%%%%%%

% Proof of the main lemmas

%%%%%%%%%%%%%%%%%%%%%%%%%%%%%%%%%%%%%%%%%%%%%%%%%%%%%%%%%

\section{The lemma for \texorpdfstring{$G$}{G}}
\label{sec:prooflemG}

In this section we prove the Lemma for~$G$ (Lemma~\ref{lem:G}), which borrows from the proof of \cite[Proposition~17]{bottcher2009proof} and from the proof of \cite[Lemma~9]{bottcher2013almost}. Our strategy is as follows. We first apply Lemma~\ref{lem:regularitylemma} to obtain an equitable partition of $V(G)$ within whose reduced graph we can find a backbone graph by Theorem~\ref{thm:bandwidth}. We let $Z_1$ be the vertices whose $\Gamma$-degrees are `wrong' to this partition, or whose neighbourhoods fail to inherit lower-regularity (plus a few extra to maintain $k$-equitability), and we remove the vertices $Z_1$. Now there may be some vertices in each cluster which destroy super-regularity on the clique factor of the backbone graph. We redistribute these, and the exceptional set of the regular partition, to other clusters. Now we would like to say we are finished, but the moving of vertices may have destroyed some of the regularity inheritance, $\Gamma$-neighbourhood, and super-regularity properties we tried to obtain. However, it is easy to check that a vertex only witnesses failure of these properties if exceptionally many of its $\Gamma$-neighbours were moved from or to a cluster. We let $Z_2$ be the set of all such vertices, and remove them. We will see that $Z_2$ is so small that its removal does not significantly affect the properties we want, so that we can set $V_0=Z_1\cup Z_2$ and we are done.

\begin{proof}[Proof of Lemma~\ref{lem:G}]
We first fix the constants in the proof. Given $\gamma>0$, $k\ge 2$ and $r_0\ge 1$, set $d=\tfrac{\gamma}{32}$. Let $\beta$ and $n_0$ be returned by Theorem~\ref{thm:bandwidth} for input $\tfrac12\gamma$, $3k$ and $k$. Let $r'_0=\max\{n_0,k/d,10k/\beta,r_0\}$. 

Given $\eps\in\big(0,\tfrac{1}{2k}\big]$, let $0<\epsa\le 10^{-10}\eps^2\gamma k^{-2}$ be small enough for both Lemmas~\ref{lem:OSRIL} and~\ref{lem:TSRIL} for input $\tfrac12\eps$ and $d$. Let $C$ be large enough for these applications of Lemmas~\ref{lem:OSRIL} and~\ref{lem:TSRIL}, and also large enough for Proposition~\ref{prop:chernoff} with input $\tfrac{1}{1000}\big(\tfrac{\epsa}{k}\big)^2$.

Now Lemma~\ref{lem:regularitylemma}, with input $\tfrac{1}{k}\epsa$, $\tfrac{k-1}{k}+\gamma$ and $r'_0+k$, returns $r_1$. We set $\Ca=1000k^3r_1^5C/(\epsa)^2$.

Given $p\ge \Ca\big(\tfrac{\log n}{n}\big)^{1/2}$, the random graph $G(n,p)$ a.a.s.\ satisfies the good events of Lemmas~\ref{lem:OSRIL} and~\ref{lem:TSRIL}, and Proposition~\ref{prop:chernoff}. We condition on $\Gamma=G(n,p)$ satisfying these good events.

Given $G\subset\Gamma$ with $\delta(G)\ge\big(\tfrac{k-1}{k}+\gamma\big)pn$, we apply Lemma~\ref{lem:regularitylemma}, with input $\tfrac{1}{k}\epsa$, $\tfrac{k-1}{k}+\gamma$, $r'_0+k$, and $d$, to $G$. We may do this because $G$ is a subgraph of $\Gamma$, and by choice of $\Ca$  we have $Cp^{-1}\log n\le\tfrac{\epsa n}{kr_1}$, so that the condition of Lemma~\ref{lem:regularitylemma} is satisfied because the good event of Proposition~\ref{prop:chernoff} holds for $\Gamma$. The result is a $\big(\tfrac{1}{k}\epsa,p\big)$-lower-regular partition of $V(G)$ into $t'\in [r'_0+k, r_1]$ equally sized clusters, with exceptional set of size at most $\tfrac{1}{k}\epsa n$, whose $(\epsa,d,p)$-reduced graph has minimum degree at least $\big(\tfrac{k-1}{k}+\gamma-d-\tfrac{1}{k}\epsa\big)t'$. We remove at most $k-1$ of these clusters to the exceptional set, obtaining an $(\epsa,p)$-lower-regular partition $\cU$ of $V(G)$ into $kr$ equally sized clusters, where $r'_0\le kr\le r_1$,  with exceptional set $U_0$ of size at most $\epsa n$, whose $(\epsa,d,p)$-reduced graph $R^k_r$ has minimum degree at least $\big(\tfrac{k-1}{k}+\gamma-d-\tfrac{1}{k}\epsa\big)kr-k$.

By choice of $d$ and $\epsa$, and by choice of $r'_0$, we have
\[\Big(\frac{k-1}{k}+\gamma-d-\frac{1}{k}\epsa\Big)kr-k\ge\Big(\frac{k-1}{k}+\frac{\gamma}{2}\Big)kr\,.\]
Observe that $B^k_r$ has bandwidth at most $2k<\beta r'_0$, and maximum degree less than $3k$. Thus Theorem~\ref{thm:bandwidth}, with input $\tfrac{\gamma}{2}$, $3k$, and $k$, in particular states that $R^k_r$ contains a copy of $B^k_r$. We fix one such copy. We let its vertices $\{(i,j)\}_{i\in[r],j\in[k]}$ label the vertices of $R^k_r$, and similarly let the cluster of $\cU$ corresponding to the vertex $(i,j)$ of $B^k_r$ be $U_{i,j}$ for each $i\in[r]$ and $j\in[k]$. The partition $\cU$ is equitable, and thus in particular $k$-equitable.

We now create $Z_1$ as follows. We start with all vertices $v$ of $G$ for which there are $(i,j)$ and $(i',j')$ in $V(R^k_r)$, with $\{(i,j),(i',j')\}$ an edge of $R^k_r$, such that either $\big(N_\Gamma(v,U_{i,j}),U_{i',j'}\big)$ or $\big(N_\Gamma(v,U_{i,j}),N_\Gamma(v,U_{i',j'})\big)$ is not $\big(\tfrac{1}{2}\eps,d,p\big)_G$-lower-regular. We add all vertices $v$ of $G$ for which there exists $U_{i,j}$ with $\deg_\Gamma(v,U_{i,j})\neq(1\pm\epsa)p|U_{i,j}|$, or for which $\deg_\Gamma(v,U_0)>2\epsa p n$. Finally we add a minimum number of vertices to obtain $k$-equitability of the sets $\big\{U_{i,j}\setminus Z_1\big\}_{i\in[r],j\in[k]}$. Note that we have $|U_{i,j}|\ge n/(2kr_1)$ for each $i,j$, and we can estimate the number of vertices with more than $2\epsa p n$ neighbours in $U_0$ by considering a superset of $U_0$ of size $\epsa n$. It follows that for each $i,j$ we have $\log(en/|U_{i,j}|),\log(en/|U_0|)\le\log(ekr_1/\epsa)$. By Lemma~\ref{lem:OSRIL} and Lemma~\ref{lem:TSRIL}, and Proposition~\ref{prop:chernoff}, we have
\begin{equation}\label{eq:sizeZ1}
 |Z_1|\le 4kr_1^2 C\max\big\{p^{-2},p^{-1}\log (ekr_1/\epsa)\big\}\le 8k^2r_1^3Cp^{-2}/\epsa\le\frac{\epsa}{kr_1} n\,,
\end{equation}
where the factor $k$ accounts for vertices removed to maintain $k$-equitability.

We now try to obtain super-regularity on the copy of $K^k_r$ in $B^k_r$. For each $i\in[r]$ and $j\in[k]$ let $W_{i,j}$ be the vertices of $U_{i,j}\setminus Z_1$ which have less than $(d-2\epsa)p|U_{i,j'}|$ neighbours in $U_{i,j'}$ for some $j'\neq j$. Because $(U_{i,j},U_{i,j'})$ is $(\epsa,d,p)$-lower-regular for each $i\in[r]$ and $j\neq j'\in[k]$, we have $|W_{i,j}|\le k\epsa|U_{i,j}|$ for each $i\in[r]$ and $j\in[k]$.

Now let $W$ contain $U_0\setminus Z_1$ together with all $W_{i,j}$ for $i\in[r]$ and $j\in[k]$, and a minimum number of additional vertices from $V(G)\setminus Z_1$ to obtain $k$-equitability of the sets $\big\{U_{i,j}\setminus (Z_1\cup W)\big\}_{i\in[r],j\in[k]}$. By construction, we have $|W|\le\epsa n+kr\cdot k\epsa\tfrac{n}{kr}\le 2k\epsa n$.

Given any $w\in W$, because $w\not\in Z_1$ we have
\[\deg_\Gamma(w,U_0)\le 2\epsa pn\quad\text{and}\quad\deg_\Gamma(w,U_{i,j})\le(1+\epsa)p|U_{i,j}|\]
for each $i\in[r]$ and $j\in[k]$. Now let us consider the edges of $G$ leaving $w$. At most $2\epsa pn$ of these go to $U_0$, and by definition at most $2dpn$ go to sets $U_{i,j}$ such that $\deg_G(w,U_{i,j})\le 2dp|U_{i,j}|$. Since $\deg_G(w)\ge\big(\tfrac{k-1}{k}+\gamma\big)pn$, at least $\big(\tfrac{k-1}{k}+\tfrac{\gamma}{2}\big)pn$ edges leaving $w$ go to sets $U_{i,j}$ with $\deg_G(w,U_{i,j})\ge 2dp|U_{i,j}|$. Since $|U_{i,j}|\le \tfrac{1}{kr}n$, in particular there are at least
\[\frac{\Big(\frac{k-1}{k}+\frac{\gamma}{2}\Big)pn}{(1+\epsa)p\frac{n}{kr}}\ge \Big(\frac{k-1}{k}+\frac{\gamma}{4}\Big)kr\]
sets $U_{i,j}$ with $i\in[r]$ and $j\in[k]$ such that $\deg_G(w,U_{i,j})\ge 2dp|U_{i,j}|$. It follows that there are at least $\tfrac{\gamma}{4}r$ indices $i\in[r]$ such that $\deg_G(w,U_{i,j})\ge 2dp|U_{i,j}|$ for each $j\in[k]$.

We now assign to each $w\in W$ sequentially an index $c(w)\in[r]\times [k]$. For each $w$, we choose $c(w)=(i,j)$ as follows. The index $i$ is chosen minimal in $[r]$ such that $\deg_G(w,U_{i,j'})\ge 2dp|U_{i,j'}|$ for each $j'\in[k]$, but at most $\tfrac{100}{r}k\epsa \gamma^{-1} n$ vertices $w'\in W$ have so far been assigned $c(w')=(i,j')$ for any $j'\in[k]$. We choose $j\in[k]$ minimising the number of vertices $w'\in W$ with $c(w)=(i,j)$. Because $|W|\le 2k\epsa n$, this assignment is always possible.

Next, for each $i\in[r]$ and $j\in[k]$, we let $V'_{i,j}$ consist of $U_{i,j}\setminus (Z_1\cup W_{i,j})$, together with all $w\in W$ such that $c(w)=(i,j)$. By construction, we have
\[|U_{i,j}\symd V'_{i,j}|\le |Z_1|+|W_{i,j}|+\frac{100}{r}k\epsa \gamma^{-1}n\le 1000k^2\epsa\gamma^{-1}|U_{i,j}|\,.\]

Finally, we let $Z_2$ be the vertices $v\in V(G)\setminus Z_1$ with $\deg_\Gamma(v,U_{i,j}\symd V'_{i,j})\ge 2000k^2\epsa\gamma^{-1} p|U_{i,j}|$ for some $i\in[r]$ and $j\in[k]$, together with a minimum number of additional vertices of $V(G)\setminus Z_1$ to obtain $k$-equitability of the sets $V_{i,j}:=V'_{i,j}\setminus Z_2$. We set $V_0=Z_1\cup Z_2$. We claim that $\cV=\{V_{i,j}\}_{i\in[r],j\in[k]}$ is the desired partition of $V(G)\setminus V_0$.

Note that the sets $V'_{i,j}$ and $V'_{i,j'}$ differ in size by at most one for any $i\in[r]$ and $j,j'\in[k]$, by our construction of the assignment $c$. We apply Proposition~\ref{prop:chernoff} to estimate the number of vertices $v\in V(G)\setminus Z_1$ with $\deg_\Gamma(v,U_{i,j}\symd V'_{i,j})\ge 2000k^2\epsa\gamma^{-1} p|U_{i,j}|$ by considering a superset of $U_{i,j}\symd V'_{i,j}$ of size $1000k^2\epsa\gamma^{-1}|U_{i,j}|\ge \epsa n/r_1$. By Proposition~\ref{prop:chernoff} we thus have
\begin{equation}\label{eq:sizeZ2}
 |Z_2|\le r_1+Ckr_1p^{-1}\log (er_1/\epsa)\le 4Ckr_1^2p^{-1}/\epsa\le\frac{\epsa}{kr_1}pn\,.
\end{equation}
This gives
\begin{equation}\label{eq:symUV}
 |U_{i,j}\symd V_{i,j}|\le|U_{i,j}\symd V'_{i,j}|+|Z_2|\le 2000k^2\epsa\gamma^{-1}|U_{i,j}|\,.
\end{equation}

Now given any $v\in V(G)\setminus V_0$, for each $i\in[r]$ and $j\in[k]$, because $v\not\in Z_2$ we have $\deg_\Gamma(v,U_{i,j}\symd V'_{i,j})\le 2000k^2\epsa\gamma^{-1}p|U_{i,j}|$. We thus have
\begin{equation}\label{eq:symdUV}
 \deg_\Gamma(v,U_{i,j}\symd V_{i,j})\le 2000k^2\epsa\gamma^{-1}p|U_{i,j}|+|Z_2|\le 3000k^2\epsa\gamma^{-1}p|U_{i,j}|\,,
\end{equation}
and because $v\not\in Z_1$ we have $\deg_\Gamma(v,U_{i,j})=(1\pm\epsa)p|U_{i,j}|$, and hence by~\eqref{eq:symUV}
\begin{equation}\label{eq:vU}
 \deg_\Gamma(v,V_{i,j})=\big(1\pm 10000k^2\epsa\gamma^{-1}\big)p|V_{i,j}|\,.
\end{equation}

Adding up~\eqref{eq:sizeZ1} and~\eqref{eq:sizeZ2}, we conclude
\begin{equation}\label{eq:sizeV0}
|V_0|\le 8k^2r_1^3Cp^{-2}/\epsa+4Ckr_1^2p^{-1}/\epsa\le \Ca p^{-2}\,,
\end{equation}
as desired. The partition $\cV=\{V_{i,j}\}_{i\in[r],j\in[k]}$ is by construction $k$-equitable, and the graph $R^k_r$ has minimum degree $\big(\tfrac{k-1}{k}+\tfrac{\gamma}{2}\big)kr$ as desired.

For each $i\in[r]$ and $j\in[k]$ we have $|U_{i,j}|=(1\pm\epsa)\frac{n}{kr}$, and so~\eqref{eq:symUV} and our choice of $\epsa$ give~\ref{lemG:size}.

Next, if $\{(i,j),(i',j')\}$ is an edge of $R^k_r$, then $G$ is $(\epsa,d,p)$-lower-regular on $(U_{i,j},U_{i',j'})$ by construction. By~\eqref{eq:symUV}, Proposition~\ref{prop:subpairs3}, and our choice of $\epsa$, $G$ is $(\eps,d,p)$-lower-regular on $(V_{i,j},V_{i',j'})$. Given $i\in[r]$ and $j\neq j'\in[k]$, let $v$ be a vertex of $V_{i,j}$. Observe that since $v\in V_{i,j}$, either we have $v\in U_{i,j}$, in which case, since $v\not\in W$ we have $\deg_G(v,U_{i,j'})\ge (d-2\epsa)p|U_{i,j'}|$, or $v$ is in $W$ and has $c(v)=(i,j)$, in which case $\deg_G(v,U_{i,j'})\ge dp|U_{i,j'}|$. By~\eqref{eq:symUV} and~\eqref{eq:symdUV} we have
\[\deg_G(v,V_{i,j'})\ge(d-2\epsa)p|U_{i,j'}|-3000k^2\epsa\gamma^{-1}p|U_{i,j'}|\ge(d-\eps)p|V_{i,j'}|\,,\]
giving~\ref{lemG:regular}.

If $\{(i,j),(i',j')\}\in E(R^k_r)$, then for any $v\in V(G)\setminus V_0$, since $v\not\in Z_1$, the pairs $\big(N_\Gamma(v,U_{i,j}),U_{i',j'}\big)$ and $\big(N_\Gamma(v,U_{i,j}),N_\Gamma(v,U_{i',j'})\big)$ are $\big(\tfrac12\eps,d,p\big)_G$-lower-regular. Using~\eqref{eq:symUV} and~\eqref{eq:symdUV}, Proposition~\ref{prop:subpairs3} and our choice of $\epsa$, we conclude~\ref{lemG:inheritance}.

Finally,~\ref{lemG:gamma} follows from~\eqref{eq:vU} and our choice of $\epsa$.

Note that if we alter the definition of $Z_1$, removing the condition on $\big(N_\Gamma(v,U_{i,j}),N_\Gamma(v,U_{i',j'})\big)$, then we do not need to use Lemma~\ref{lem:TSRIL} and the bound in~\eqref{eq:sizeZ1} improves to $|Z_1|\le 8k^2r_1^3Cp^{-1}/\epsa$. Thus, if we only require~\ref{lemG:inheritancep}, we obtain $|V_0|\le\Ca p^{-1}$ as claimed.
\end{proof}

\section{The lemma for \texorpdfstring{$H$}{H}}\label{sec:lemH} 

In this section we present the proof of Lemma~\ref{lem:H2}. 
The proof idea is as follows. First, given the zero-free labelling $\cL$ and $(k+1)$-colouring $\sigma$ of $H$, we split $\cL$ into the blocks of the definition of zero-freeness. We partition the blocks into $r$ `sections' of consecutive blocks, such that the $i$-th section contains about $\sum_{j\in[k]}m_{i,j}$ vertices, and furthermore such that the `boundary vertices, namely the first and last $\beta n$ vertices of each section, do not receive colour zero. Now it is easy to check that assigning the vertices of colour $j$ in the $i$-th section to $(i,j)$ for each $i\in[r]$ and $j\in[k]$, and the vertices of colour zero in the $i$-th section to $z_i$, is a graph homomorphism. However it can be very unbalanced, since different colours in $[k]$ may be used with very different frequencies in each section. To fix this, we replace $\sigma$ with a new colouring $\sigma'$, which we obtain as follows. We partition each section into `intervals' of consecutive blocks, and for each interval except the last in each section, we pick a random permutation of $[k]$. We will show that there is a colouring $\sigma'$ such that all but the first few vertices of each interval are coloured according to the permutation applied to $\sigma$, with vertices of colour zero staying coloured zero. We use this colouring $\sigma'$ in place of $\sigma$ to define the mapping $f$. We let $X$ consist of all vertices whose distance is two or less to either boundary vertices, vertices near the start of an interval, or colour zero vertices.

To complete the proof, we show that so few vertices receive colour zero that they do not much affect the desired conclusions. Now the mapping $f$ is in expectation balanced, and using Lemma~\ref{lem:McDiarmid} we can show that it is also with high probability close to balanced. It is also easy to check that, since $H$ is $D$-degenerate, in the $i$-th section of $\cL$ there are many vertices of degree at most $2D$. In expectation these are distributed about evenly over the $\big\{(i,j)\big\}_{j\in[k]}$ by $f$, and again McDiarmid's inequality shows that with high probability the same holds. These two observations give us~\ref{lemH:H1} and~\ref{lemH:H6}, while the other four desired conclusions hold by construction.

\begin{proof}[Proof of Lemma~\ref{lem:H2}]
For given $D\geq 1$, set $\alpha = 1/(24D)$. Let $k, r \geq 1$ and $\xi, \beta >0$ be given, where $\xi \leq 1/(kr)$ and $\beta \leq 10^{-10}\xi^2/(D k^4r)$. Let  $H$ and $K^k_r \subseteq B^k_r \subseteq R^k_r$ be graphs as in the statement of the lemma. Let $\mathcal L$ be the given labelling of $V(H)$ of bandwidth at most $\beta n$. We denote the set of the first $\sqrt{\beta}n$ vertices of $\mathcal L$ by $F$.  Let $\sigma: V(H) \to \{0,\ldots k\}$ be the given proper $(k+1)$-colouring of $V(H)$ that is $(10/\xi, \beta)$-zero-free with respect to $\mathcal L$ and such that $\sigma(F) \subseteq [k]$. 
Also, let $z_1, \ldots, z_n$ be vertices such that $z_i \in \big([r]\setminus\{i\}\big) \times [k]$ with $\big\{z_i, (i,j)\big\} \in E(R^k_r)$ for every $i\in [r]$ and $j\in [k]$.
Finally, set $b=k/\sqrt{\beta}$. 

Let $\{m_{i,j}\}_{i\in[r],j\in[k]}$ be the given $k$-equitable integer partition of $n$ with $n/(10kr) \leq m_{i,j} \leq 10n/(kr)$ for every $i\in[r]$ and $j\in[k]$.

Let us now introduce the notation that we use in this proof. Recall that for every $t\in \big[1/(4k\beta)\big]$ the $i$-th block is defined as 
\[B_t:= \{(t-1)4k\beta n +1, \ldots, t4k\beta n\}.\]
Next we split the labelling $\mathcal L$ into $r$ \emph{sections}, where the first and the last block of each section are zero-free. Each section is partitioned into \emph{intervals}, each of which but possibly the last one consists of $b$ \emph{blocks}. 

Since $\sigma$ is $(10/\xi, \beta)$-zero-free with respect to $\mathcal L$, we can choose indices $0 = t_0 \leq t_1 \leq \ldots \leq t_{r-1} \leq t_r = 1/(4k\beta)$ such that $B_{t_i}$ and $B_{t_{i}+1}$ are zero-free blocks for every $i\in [r]$  and 
\[\sum_{t=1}^{t_i}|B_t| \leq \sum_{t=1}^{i} \sum_{j\in[k]}m_{t,j} < 12k\beta n + \sum_{t=1}^{t_i} |B_t|.\] 
Since $m_{i,j} \geq n/(10kr) > 12k\beta n$, indices $t_0, \ldots, t_r$ are distinct. 
For every $i \in [r]$ we define the $i$-th section $S_i$ as 
\[\bigcup_{t = t_{i-1}+1}^{t_i} B_t.\]
This means by the choice of the indices $t_0, \ldots, t_r$ that the first and last block of each section are zero-free. 
Since $\{m_{i,j}\}_{i\in[r],j\in[k]}$ is a $k$-equitable partition, we have in particular 
\begin{equation}\label{eq:mijSi}
\frac{1}{k} (|S_i|- 12k\beta n) \leq  m_{i,j} \leq \frac{1}{k} \big(|S_i| + 12 k \beta n\big).
\end{equation}
The last $\beta n$ vertices of the blocks $B_{t_i}$ and the first $\beta n$ vertices of the blocks $B_{t_{i+1}}$ are called \emph{boundary vertices} of $H$. Notice that colour zero is never assigned to boundary vertices by $\sigma$. For each $i\in [r]$, we split $S_i$ into $s_i:= \left\lceil (t_i-t_{i-1}-1)/b \right\rceil$ intervals, where each of the first $(s_i-1)$ intervals is the concatenation of exactly $b$ blocks and the last interval consists of $t_i-t_{i-1}-1 - b(s_i-1) \leq b$ blocks. 
Therefore, for every $i\in [r]$, we have 
\begin{equation}\label{eq:sizeSi}
s_i(b-1)4k\beta n +1 \leq |S_i| \leq s_i b 4 k \beta n. 
\end{equation}
Using Equation~\eqref{eq:mijSi}, $b= k/\sqrt{\beta}$, and $n/(10kr)\leq m_{i,j} \leq 10n/(kr)$ we get, for every $i\in[r]$, the following bounds on $s_i$ 
\begin{equation*}\label{eq:si}
\frac{1}{100rk^2\sqrt{\beta}} \leq s_i \leq \frac{10}{rk^2\sqrt{\beta}}.
\end{equation*}
We denote the intervals of the $i$-th section by $I_{i,1}, \ldots, I_{i,s_i}$. Let $\Bs_{i,\ell}$ denote the union of the first two blocks of each interval $I_{i,\ell}$. All of these blocks but $\Bs_{i,1}$ and $\Bs_{i,s_i}$   will be used to switch colours within parts of $H$. Notice that we have $|\Bs_{i,\ell}| = 8k \beta n$ and, since $\sigma$ is $(10/\xi, \beta)$-zero-free with respect to $\mathcal L$, at least one of the two blocks of $\Bs_{i,\ell}$ is zero-free. We will not use $\Bs_{i,1}$ and $\Bs_{i,s_i}$ to switch colours because we will need that the boundary vertices do not receive colour zero.

For every $i\in [r]$ and every $\ell \in \{2,\ldots, s_i-1\}$, we choose a permutation $\pi_{i,\ell}: [k] \to [k]$ uniformly at random. 

The next claim ensures that we can use zero-free blocks to obtain a proper colouring of the vertex set such that vertices before the switching block are coloured according to the original colouring and the colours of the vertices after the switching block are permuted as wished. A proof can be found in \cite{BoeTarWue}. 
\begin{claim}[\cite{BoeTarWue}]
\label{claim:switching}
Let $\sigma: [n] \to \{0,\ldots, k\}$ be a proper $(k+1)$-colouring of $H$, let $B_t$ be a zero-free block and let $\pi$ be any permutation of $[k]$. Then there exists a proper $(k+1)$-colouring $\sigma'$ of $H$ with $\sigma'(x) = \sigma(x)$ for all $x\in \bigcup_{i<t} B_i$ and
\[\sigma'(x) = \begin{cases} \pi(\sigma(x)) & \text{if } \sigma(x) \neq 0 \\
			  0 & \text{otherwise} \end{cases}\]
 for all $x\in \bigcup_{i>t}B_i$.
\end{claim}

We use Claim~\ref{claim:switching} to switch colours at the beginning of each interval except for the first and last interval of each section. More precisely, we switch colours within the sets $\Bs_{i,\ell}$ so that the colouring of the remaining vertices in the interval $I_{i,\ell}$ matches $\pi_{i, \ell}$. Note that we can indeed use $\Bs_{i,\ell}$ to do the switching since one of the two blocks in $\Bs_{i,\ell}$ is zero-free. In particular, we get a proper $(k+1)$-colouring $\sigma'= \sigma'\big(\pi_{1,2}, \ldots, \pi_{r,s_r-1}\big): V(H) \to \{0, \ldots k+1\}$ of $H$ that fulfils the following. 
For every $x\in I_{1,1}$ we have 
\[\sigma'(x) = \sigma(x),\]
for each $i \in [r]$ and $\ell \in \{2,\ldots, s_i-1\}$ and every $x\in I_{i,\ell} \setminus \Bs_{i,\ell}$ we have that 
\[\sigma'(x) = 
\begin{cases} \pi_{i,\ell}\big(\sigma(x)\big) & \text{if } \sigma(x) \neq 0 \\
			  0 & \text{otherwise} \end{cases}\]
and for each $i\in [r]$ and every $x\in I_{i,s_{i}}\cup  I_{i+1,1}$ (where $I_{r+1,1}:= \varnothing$) we have that
\[\sigma'(x) = \pi_{i, s_i-1}\big(\sigma(x)\big).\]
While $\sigma'$ is well-defined on the sets $\Bs_{1,2}, \ldots, \Bs_{r, s_r-1}$ by Claim~\ref{claim:switching}, the definition on these sets is rather complicated as it is depends on which of the two blocks in $\Bs_{i,\ell}$ is zero-free and on the colourings before and after the switching. However, the precise definition on these sets is not important for the remainder of the proof. Hence, we omit it here. Observe that $\sigma'$ never assigns colour zero to boundary vertices. 

Using $\sigma'$ we now define $f= f\big(\pi_{1,2}, \ldots, \pi_{r,s_r-1}\big): V(H) \to [r] \times [k]$ as follows. For each $i\in [r]$ and $x\in S_i$ we set 
\[f(x):= 
\begin{cases} \big(i, \sigma'(x)\big) & \text{if } \sigma'(x) \neq 0 \\
			  z_i & \text{otherwise},\end{cases}\]
where $z_i \in \big([r]\setminus\{i\}\big) \times [k]$ is the vertex defined in the statement of the lemma. 
Let $X$ consist of all vertices at distance two or less from a boundary vertex of $\cL$, from a vertex in any $\Bs_{i,\ell}$, or from a colour zero vertex. We now show that $f$ and $X$ satisfy Properties~\ref{lemH:H2}--\ref{lemH:H5} with probability 1 and Properties~\ref{lemH:H1} and~\ref{lemH:H6} with high probability. In particular, this implies that the desired $f$ and $X$ exist.

We start with Property~\ref{lemH:H1}. 
For each $i\in [r]$ let 
\[S^\ast_i:=S_i \setminus \left( \bigcup_{ \ell \in [s_i]} \Bs_{i,\ell} \cup  I_{i,1} \cup I_{i,s_i} \right)\]
be the set of all vertices in $S_i$ except for the first and last interval and the first two blocks of each interval of $S_i$. 
We will also make use of the following restricted function 
\[f^\ast = f^\ast\big(\pi_{1,2},\ldots, \pi_{r,s_r}\big):= \restr{f}{\bigcup_{i\in[r]} S^\ast_i }.\] 
The basic idea of the proof of Property~\ref{lemH:H1} is to determine bounds on $|{f^\ast}^{-1}(i,j)|$ that hold with positive probability and then deduce the desired bounds on $|f^{-1}(i,j)|$. 
Since the permutations $\pi_{i,\ell}$ were chosen uniformly at random, we have by definition of $f^\ast$ that the expected number of vertices mapped to $(i,j)\in [r]\times[k]$ by $f^\ast$ is 
\begin{multline*}
\Ex\big[|{f^{\ast}}^{-1}(i,j)|\big] = \frac{1}{k} \Big[ (s_i-2) (b-2) 4k \beta n  -\big|\{x \in S^\ast_{i}: \sigma(x) = 0  \}\big| \Big] \\
 + \big|\bigcup_{\iota\in [r]\setminus \{i\}} \{x\in S^\ast_{\iota}: \sigma(x) = 0 \text{ and } z_{\iota} = (i,j) \}\big|\,.
\end{multline*}
In particular, the following bounds on the expected value of $|{f^{\ast}}^{-1}(i,j)|$ hold. 
\begin{equation}\label{eq:fastijlow}
\Ex\big[|{f^{\ast}}^{-1}(i,j)|\big] \leq  (s_i-2) (b-2) 4 \beta n + \frac{\xi}{10} n
\end{equation}
and
\begin{equation}\label{eq:fastijup}
\Ex\big[|{f^{\ast}}^{-1}(i,j)|\big] \geq (1- \xi/ 10) (s_i-2) (b-2) 4 \beta n  \geq  (s_i-2) (b-2) 4 \beta n - \frac{\xi}{10}n.
\end{equation}

If one replaced a permutation $\pi_{i,\ell}$ by some other permutation $\tilde{\pi}: [k] \to [k]$, then $|{f^\ast}^{-1}(i,j)|$ would change by at most $(b-2) 4k\beta n$. Hence, by McDiarmid's Inequality (Lemma~\ref{lem:McDiarmid}) we have
\begin{multline}
\Pr\left[\big| (s_i-2) (b-2) 4 \beta n- |{f^\ast}^{-1}(i,j)|\big| \geq  \frac{\xi}{5} n \right] \overset{\eqref{eq:fastijlow}, \eqref{eq:fastijup}}\leq  \\
\Pr\left[\big| \Ex[|{f^\ast}^{-1}(i,j)] - |{f^\ast}^{-1}(i,j)|\big| \geq  \frac{\xi}{10} n \right] \leq 
2\exp\left\{-\frac{\xi^2 n^2}{50 (s_i-2) \big((b-2)4k\beta n\big)^2}\right\}.
\end{multline}

Taking the union bound over all $j\in [k]$ and using $s_i\leq 10/(rk^2\sqrt{\beta})$ and $b=k/\sqrt{\beta}$ as well as $\beta \leq 10^{-10}\xi^2/(D k^4r)$  yields 
\[ \Pr\left[\big| (s_i-2) (b-2) 4 \beta n- |{f^\ast}^{-1}(i,j)|\big| \geq  \frac{\xi}{5} n \text{ for all } j\in [k] \right] \leq 2k \exp\left\{-\frac{\xi^2r}{8000k^2\sqrt{\beta}}\right\}\leq 2k e^{-k} < 1 .\] 

Observe that $|{f^\ast}^{-1}(i,j)|$ is independent of the choices for $\pi_{i',\ell}$ if $i' \neq i$. Hence, with positive probability we have, for every $i\in [r]$ and $j\in [k]$, that
\[(s_i-2) (b-2) 4 \beta n - \frac{\xi}{5}n \leq |{f^\ast}^{-1}(i,j)| \leq (s_i-2) (b-2) 4 \beta n + \frac{\xi}{5} n. \]
From the definition of $f^\ast$ it follows that $|f^{-1}(i,j)|\geq |{f^\ast}^{-1}(i,j)|$ and 
\[|f^{-1}(i,j)| \leq |{f^\ast}^{-1}(i,j)| + |I_{i,1}| + |I_{i,s_i}| + \sum_{\ell=2}^{s_i-1} |\Bs_{i,\ell}|+ \Big|\big\{x \in\hspace{-2mm} \bigcup_{\iota \in [r] \setminus \{i\}} S_{\iota}\setminus S^\ast_{\iota}: \sigma'(x) = 0 \text{ and } z_{\iota}=(i,j) \big\}\Big|. \]
Using $s_i\leq 10/(rk^2\sqrt{\beta})$ and $b=k/\sqrt{\beta}$ and $\beta \leq 10^{-10}\xi^2/(D k^4r)$, with positive probability we have for every $i\in[r]$ and $j\in[k]$ that 
\begin{align*}
|f^{-1}(i,j)| &\geq |{f^\ast}^{-1}(i,j)| 
\geq (s_i-2) (b-2) 4 \beta n - \frac{\xi}{5} n\\ &\geq 
(s_i-2)(b-2) 4 \beta n - \frac{\xi}{5} n + \left(8 (s_i + b) \beta n - \frac{4}{5}\xi n\right)\\ &\geq 
s_ib4\beta n + 16 \beta n - \xi n\\ &
\overset{\eqref{eq:sizeSi}}\geq \frac{1}{k}\big(|S_i| + 16 k \beta n\big) - \xi n \overset{\eqref{eq:mijSi}}\geq 
m_{i,j} - \xi n.
\end{align*}
On the other hand, 
\begin{align*}
|f^{-1}(i,j)| &\leq |{f^\ast}^{-1}(i,j)| + |I_{i,1}| + |I_{i,s_i}| + \sum_{\ell=2}^{s_i-1} |\Bs_{i,\ell}| \\&\,\,\quad+ \Big|\big\{x \in \bigcup_{\iota \in [r] \setminus \{i\}}  S_{\iota}\setminus S^\ast_{\iota}: \sigma'(x) = 0 \text{ and } z_{\iota}=(i,j) \big\}\Big|\\&
\leq (s_i-2) (b-2) 4 \beta n + \frac{\xi}{5} n +8bk\beta n+ (s_i-2)8k\beta n+  \frac{\xi}{10}n \\&
\leq \frac{1}{k}\big((s_i-2)(b-2)4k\beta n \big) + \xi n\\&
\leq \frac{1}{k} (|S_i|-12k\beta n)  + \xi n \overset{\eqref{eq:mijSi}}\leq m_{i,j}+\xi n,
\end{align*}
 which shows that Property~\ref{lemH:H1} holds with positive probability.

By definition of $X$, since $\mathcal L$ is a $\beta n$-bandwidth ordering, any vertex in $X$ is at distance at most $2\beta n$ in $\cL$ from a boundary vertex, a vertex of some $\Bs_{i,\ell}$, or from a vertex assigned colour zero. Because there are $r$ sections, the boundary vertices form $r-1$ intervals each of length $2\beta n$, and so at most $6r\beta n$ vertices of $H$ are at distance $2$ or less from a boundary vertex. There are $\sum_{i\in[r]} s_i$ intervals and hence $\sum_{i\in[r]} s_i$ switching blocks each of size $8k\beta n$. As $s_i \leq 10/(rk^2\sqrt{\beta})$ for every $i\in[r]$, there are at most $(4+8k)\beta n\cdot 10/(k^2\sqrt{\beta})$ vertices at distance 2 or less from a vertex of some switching block. 
Similarly, because $\cL$ is $(10/\xi,\beta)$-zero-free, in any consecutive $10/\xi$ blocks at most one contains vertices of colour zero, and hence at most $(8+4k)\beta n$ vertices in any such $10/\xi$ consecutive blocks are at distance $2$ or less from a vertex of colour zero. Thus we have
\[|X|\le 6r\beta n+(4+8k)\beta n\left(\frac{10}{k^2\sqrt{\beta}n}\right)+(8+4k)\beta n\left(\frac{n}{4k\beta n\cdot 10/\xi}+1\right)\le 6r\beta n+\frac{1}{4}\xi n+\frac{1}{3}\xi n\le\xi n\,,\]
which gives~\ref{lemH:H2}.

Since $\sigma'$ is a proper colouring, and boundary vertices are not adjacent to colour zero vertices, by definition, $f$ restricted to the boundary vertices is a graph homomorphism to $B^k_r$. On the other hand, on each section $S_i$, again since $\sigma'$ is a proper colouring and since $\big\{(i,j)\big\}_{j\in[k]}\cup\{z_i\}$ forms a clique in $R^k_r$, $f$ is a graph homomorphism to $R^k_r$. Since $\cL$ is a $\beta n$-bandwidth ordering, any edge of $H$ is either contained in a section or goes between two boundary vertices, and we conclude that $f$ is a graph homomorphism from $H$ to $R^k_r$, giving~\ref{lemH:H3}.

Now, given $i\in[r]$ and $j\in[k]$, and $x\in f^{-1}(i,j)\setminus X$, if $\{x,y\}$ and $\{y,z\}$ are edges of $H$, then $y$ and $z$ are at distance two or less from $x$ in $H$. In particular, by definition of $X$ neither $y$ nor $z$ is either a boundary vertex, in any $\Bs_{i,\ell}$, or assigned colour zero. Since boundary vertices appear in intervals of length $2\beta n$ in $\cL$, and $\cL$ is a $\beta n$-bandwidth ordering, it follows that $y$ and $z$ are both in $S_i$. Furthermore, suppose $x\in I_{i,\ell}$ for some $\ell$. By definition $x\not\in\Bs_{i,\ell}$. Because $\Bs_{i,\ell}$ and $\Bs_{i,\ell+1}$ (if the latter exists) are intervals of length $8k\beta n$, both $y$ and $z$ are also in $I_{i,\ell}\setminus \Bs_{i,\ell}$, and in particular both $y$ and $z$ are in $\bigcup_{j'\in[k]}f^{-1}(i,j')$, giving~\ref{lemH:H4}.

Since $\sqrt{\beta}n \leq b4k\beta n \leq |I_{1,1}|$ and $\sigma'(x) \neq 0$ for each $x$ in the first $\sqrt{\beta}n$ vertices of $\mathcal{L}$, it follows directly from the definition of $f$ that $f(x) = \big(1,\sigma(x)\big)$, which shows Property~\ref{lemH:H5}. 

Finally, we show that Property~\ref{lemH:H6} holds with positive probability. 
Let $i\in[r]$ and $j\in[k]$. We define the random variable $\mathcal E_{i,j} := |\{x \in {f^\ast}^{-1}(i,j): \deg(x) \leq 2 D\}|$. Since $H$ is $D$-degenerate and $\mathcal L$ is a labelling of bandwidth at most $\beta n$ we have 
\[e\big(S_i^\ast, V(H)\big) \leq D |S_i^\ast| + D 4\beta n \leq D \big(1+ 1/(4D)\big)  |S_i^\ast|.\]  
 Hence, it must hold that $|\{x\in S_i^\ast: \deg(x) \geq 2D + 1\}| (2D + 1) \leq 2 D \big(1+1/(4D)\big) |S_i^\ast|.$
This yields $|\{x\in S_i^\ast: \deg(x) \leq 2D\}| \geq |S_i^\ast|/(6D)$ and therefore 
\[\Ex[\mathcal E_{i,j}] \geq \frac{1}{6kD} |S_i^\ast| \geq  \frac{1}{6D} (s_i-2) (b-2)4\beta n.\]
By applying Chernoff's Inequality (Theorem~\ref{thm:chernoff}) and using Equations~\eqref{eq:mijSi} and~\eqref{eq:sizeSi} as well as $\alpha= 1/(24D)$ we get
\begin{align*}
&\Pr\Big[\big|\{x \in f^{-1}(i,j):\deg(x) \leq 2 D \}\big| < \alpha |f^{-1}(i,j)|\Big] \overset{\ref{lemH:H1}}\leq 
\Pr\Big[\mathcal E_{i,j} < \alpha (s_i b 4 \beta n + 2\xi n)\Big]\\ &\leq 
\Pr\Big[\mathcal E_{i,j} < 2 \alpha \big((s_i-2)(b-2) 4 \beta n\big)\Big] \leq 
\Pr\Big[\mathcal E_{i,j} < \frac{1}{2}\Ex[\mathcal E_{i,j}]\Big] < 2 \exp\left\{- \frac{(s_i-2)(b-2)4\beta n}{72}\right\}.
\end{align*}
Taking the union bound over all $i\in[r]$ and $j\in [k]$ yields that Property~\ref{lemH:H6} holds with positive probability. 
\end{proof}

\section{The common neighbourhood lemma}

In order to prove Lemma~\ref{lem:common} we need the following version of the Sparse Regularity Lemma, allowing for a partition equitably refining an initial partition with parts of very different sizes. Given a partition $V(G)=V_1\dcup\dots\dcup V_s$, we say a partition $\{V_{i,j}\}_{i\in[s],j\in[t]}$ is an equitable $(\eps,p)$-regular refinement of $\{V_i\}_{i\in[s]}$ if $|V_{i,j}|=|V_{i,j'}|\pm 1$ for each $i\in[s]$ and $j,j'\in[t]$, and there are at most $\eps s^2t^2$ pairs $(V_{i,j},V_{i',j'})$ which are not $(\eps,0,p)$-regular. 

\begin{lemma}
\label{lem:SRLb}
For each $\eps>0$ and $s\in\mathbb{N}$ there exists $t_1\geq 1$ such that the following holds. Given any graph $G$, suppose $V_1\dcup\dots\dcup V_s$ is a partition of $V(G)$. Suppose that $e(V_i)\le 3p|V_i|^2$ for each $i\in[s]$, and $e(V_i,V_{i'})\le 2p|V_i||V_{i'}|$ for each $i\neq i'\in[s]$. Then there exist sets $V_{i,0}\subset V_i$ for each $i\in[s]$ with $|V_{i,0}|<\eps|V_i|$, and an equitable $(\eps,p)$-regular refinement $\{V_{i,j}\}_{i\in[s],j\in[t]}$ of $\{V_i\setminus V_{i,0}\}_{i\in[s]}$ for some $t\le t_1$.
\end{lemma} 
 The proof is standard, following Scott's method~\cite{Scott}. We defer it to Appendix~\ref{app:tools}.

To prove Lemma~\ref{lem:common}, we work as follows. First, we choose a regularity parameter $\epsaa_0$ and apply Lemma~\ref{lem:SRLb} with $\epsaa_0$ and the initial partition $V_1\setminus W,\dots,V_k\setminus W,W$. From this partition, all we need is a part $W'\subset W$ and parts $V'_i\subset V_i\setminus W$ for each $i\in[k]$, such that each pair $(W',V'_i)$ is $(\epsaa_0,d/2,p)$-lower-regular, which we find by averaging. We now choose our vertices $w_1,\dots,w_\Delta$ sequentially (in Claim~\ref{claim:common}), such that the desired~\ref{cnl:Gsize}--\ref{cnl:Nreg} hold for all subsets of the so far chosen vertices at each stage. This is in spirit very much like the usual dense case `Key Lemma' sequential embedding of vertices using regularity, but in the sparse setting here we need to work somewhat harder and use the regularity inheritance lemmas to show that we can choose vertices which give us lower-regular pairs for future embedding (rather than this being automatic from the slicing lemma, as it is in the dense case).

Thus, the proof mainly amounts to showing that the number of vertices which break one of the desired properties and which we therefore cannot choose is always much smaller than $|W'|$. In order to show this for~\ref{cnl:Gsize} we need to maintain some extra properties, specifically sizes of $G$- and $\Gamma$-neighbourhoods of chosen vertices within each $V'_i$, and that these $\Gamma$-neighbourhoods of chosen vertices in each $V'_i$ form lower-regular pairs with $W'$.

Note that the way we choose our various regularity parameters amounts to ensuring that, even after $\Delta-1$ successive applications of regularity inheritance lemmas, we still have sufficient regularity for our argument. Furthermore, it is important to note that the choice of $\epsaa_0$ does not have anything to to with $\epsa$ or $\eps_0$, rather it affects only the returned value of $\alpha$.

\begin{proof}[Proof of Lemma~\ref{lem:common}]
First we fix all constants that we need throughout the proof. 
Given $d>0$, $k\geq 1$, and $\Delta\geq 2$, let $\epsaa_{\Delta}:=8^{-\Delta}\frac{1}{(k+1)^2}\left(\frac d 8\right)^{\Delta}$.
Now, for each $j=1,\dots,\Delta$ sequentially, choose $\epsaa_{\Delta-j}\le\epsaa_{\Delta-j+1}$ not larger than the $\eps_0$ returned by Lemma~\ref{lem:OSRIL} for input $\epsaa_{\Delta-j}$ and $\tfrac{d}{2}$.

Now, Lemma~\ref{lem:SRLb} with input $\epsaa_0$ and $s=k+1$ returns $t_1\ge 1$. We set
\[\alpha:=\frac{1}{2t_1}\Big(\frac{d}{4}\Big)^\Delta\,.\]

Next, given $\epsa>0$, let $\epsa_{\Delta-1,\Delta-1}:=\epsa$, and let $\epsa_{j,\Delta}=\epsa_{\Delta,j}=1$ for each $1\le j\le\Delta$. For each $(j,j')\in[\Delta]^2\setminus\{(1,1)\}$ in lexicographic order sequentially, we choose
\[\epsa_{\Delta-j,\Delta-j'}\le\min\{\epsa_{\Delta-j+1,\Delta-j'},\epsa_{\Delta-j,\Delta-j'+1},\epsa_{\Delta-j+1,\Delta-j'+1}\}\]
not larger than the $\eps_0$ returned by Lemma~\ref{lem:OSRIL} for both input $\epsa_{\Delta-j+1,\Delta-j'}$ and $d$, and for input $\epsa_{\Delta-j,\Delta-j'+1}$ and $d$, and not larger than the $\eps_0$ returned by Lemma~\ref{lem:TSRIL} for input $\epsa_{\Delta-j+1,\Delta-j'+1}$ and $d$.

 We choose $\eps_0$ small enough such that $(1+\eps_0)^{\Delta} \leq 1+\epsa$ and $(1-\eps_0)^{\Delta} \geq 1-\epsa$. Given $r\ge 1$ and $\eps$ with $0<\eps\le\eps_0$, suppose that $C$ is large enough for each of these calls to Lemmas~\ref{lem:OSRIL} and~\ref{lem:TSRIL}, and for Proposition~\ref{prop:chernoff} with input $\eps_0$. Finally, we set
\[\Ca= 10^{12}k^4t_1r^4\eps^{-4}2^{2\Delta}C\,.\]

Given $p\ge \Ca\big(\tfrac{\log n}{n}\big)^{1/\Delta}$, a.a.s.\ the good events of each of the above calls to Lemma~\ref{lem:OSRIL} and~\ref{lem:TSRIL}, and to Proposition~\ref{prop:chernoff} and Lemma~\ref{lem:SRLb}, occur. We condition from now on upon these events occurring for $\Gamma=G(n,p)$.

Let $G=(V,E)$ be a subgraph of $\Gamma$. Suppose $\{V_i\}_{i\in[k]}$ and $W$ satisfy the conditions of the lemma. We first apply Lemma~\ref{lem:SRLb}, with the promised input parameters $\epsaa_0$ and $s=k+1$, to $G[V_1\cup\dots\cup V_k\cup W]$, with input partition $\{V_i\setminus W\}_{i\in[k]}\cup\{W\}$. We can do this because $Cp^{-1}\log n<10^{-10}\frac{\eps^4 pn}{k^4r^4}$, so that the good event of Proposition~\ref{prop:chernoff} guarantees that the conditions of Lemma~\ref{lem:SRLb} are satisfied. This returns a partition refining each set of $\{V_i\setminus W\}_{i\in[k]}\cup\{W\}$ into $1\le t\le t_1$ clusters together with a small exceptional set. Let $W'\subset W$ be a cluster which is in at most $2k\epsaa_0 t$ pairs with clusters in $\big(V_1\cup\dots\cup V_k\big)\setminus W$ which are not $(\epsaa_0,p)_G$-lower-regular. Such a cluster exists by averaging. By Proposition~\ref{prop:chernoff} and~\ref{cnl:bal}, at most $4(k+1)\epsaa_0 p \tfrac{4n}{r}|W'|$ edges lie in the pairs between $W'$ and the $V_i$ which are not lower-regular, and by Proposition~\ref{prop:chernoff} and~\ref{cnl:W} at most $2p|W||W'|<\epsaa_0 p\tfrac{n}{r}|W'|$ edges leaving $W'$ lie in $W$. By~\ref{cnl:Wdeg}, for each $i\in[k]$ each $w\in W'$ has at least $dp|V_i|$ neighbours in $V_i$, and hence there are at least $\tfrac{dp}{2}|V_i||W'|$ edges from $W'$ to $V_i\setminus W$ which lie in $(\epsaa_0,p)_G$-lower-regular pairs. By averaging, for each $i\in[k]$ there exists a cluster $V'_i$ of the partition such that $(W',V_i')$ is $(\epsaa_0, d/2, p)_G$-lower-regular. For the remainder of the proof, we will only need these $k+1$ clusters from the partition.

Notice that for every $i\in[k]$ we have 
\[|V_i| \geq |V_i'| \geq \frac{n}{8kt_1r} \geq \frac{1}{8kt_1r} (\Ca)^{2}p^{-2}\log n \geq \Ca p^{-2} \log n\]
and 
\begin{equation}
\label{eq:sizeW}
|W'| \geq 10^{-11}\frac{\eps^4 pn}{t_1k^4r^4}  \geq 10^{-11}\frac{\eps^4}{t_1k^4r^4}(\Ca)^{2}p^{-1}\log n \geq \Ca p^{-1} \log n
\end{equation}
both by the choice of $\Ca$ and $p$.

 We choose the $\Delta$-tuple $(w_1,\dots,w_\Delta)$ inductively, using the following claim.
\begin{claim}
\label{claim:common}
For each $0\le \ell\le\Delta$ there exists an $\ell$-tuple $(w_1,\ldots,w_\ell) \in \binom{W'}{\ell}$ such that the following holds.

For every $\Lambda, \Lambda^\ast \subseteq [\ell]$, and every $i \neq i' \in [k]$ we have

\begin{enumerate}[label=\itmarab{L}]
 \item\label{cnl:cl:Wreg} $\big(\bigcap_{j\in \Lambda}\NGa(w_j,V_i'),W'\big)$ is $(\epsaa_{|\Lambda|},\frac d 2, p)_G$-lower-regular if $|\Lambda|< \Delta$,
 \item\label{cnl:cl:NGVp} $|\bigcap_{j\in\Lambda} N_G(w_j,V_i')| \geq \big(\frac{d}{4}\big)^{|\Lambda|} p^{|\Lambda|}|V_i'|$,
 \item\label{cnl:cl:NGa} $|\bigcap_{j\in\Lambda}\NGa(w_j)| \leq (1+\eps_0)^{|\Lambda|}p^{|\Lambda|} n$, 
 \item\label{cnl:cl:NGaVp} $|\bigcap_{j\in\Lambda}\NGa(w_j,V_i')| = (1\pm\eps_0)^{|\Lambda|}p^{|\Lambda|} |V_i'|$, 
 \item\label{cnl:cl:NGaV} $|\bigcap_{j\in\Lambda}\NGa(w_j,V_i)| = (1\pm\eps_0)^{|\Lambda|}p^{|\Lambda|} |V_i|$, and
 \item\label{cnl:cl:Vreg} $\big(\bigcap_{j\in\Lambda}\NGa(w_j,V_i),\bigcap_{j^\ast\in\Lambda^\ast}\NGa(w_{j^\ast},V_{i'})\big)$ is $(\epsa_{|\Lambda|,|\Lambda^\ast|},d,p)_G$-lower-regular if\\ $|\Lambda|,|\Lambda^\ast| < \Delta$ and either $\Delta \geq 3$ or $\Lambda \cap \Lambda^\ast = \varnothing$ or both. 
\end{enumerate}
\end{claim}

We prove this claim by induction on $\ell$. Recall that if $\Lambda=\emptyset$ then $\bigcap_{j\in\Lambda}\NGa(w_j,V_i')$ is by definition equal to $V'_i$, and that $[0]=\emptyset$. 

\begin{claimproof}[Proof of Claim~\ref{claim:common}]
	For the base case $\ell=0$, observe that~\ref{cnl:cl:Wreg} follows from our choice of $W'$ and the $V_i'$. For every $i,j\in [k]$, the pair $(V_i,V_j)$ is $(\eps,d,p)_G$-lower-regular by~\ref{cnl:Vreg}, and since $\eps\le\epsa_{0,0}$ this gives~\ref{cnl:cl:Vreg}. The remaining three properties~\ref{cnl:cl:NGVp},~\ref{cnl:cl:NGaVp} and~\ref{cnl:cl:NGaV} are tautologies for $\ell=0$.
	 
	For the inductive step, suppose that for some $0\le\ell<\Delta$ there exists an $\ell$-tuple $(w_1, \ldots, w_{\ell}) \in \binom{W'}{\ell}$ satisfying~\ref{cnl:cl:Wreg}--\ref{cnl:cl:Vreg}. We now find a vertex $w_{\ell+1} \in W'$ such that the $(\ell+1)$-tuple $(w_1, \ldots, w_{\ell+1})$ still satisfies~\ref{cnl:cl:Wreg}--\ref{cnl:cl:Vreg}. We do this by determining, for each of these five conditions, an upper bound on the number of vertices in $W'$ that violate them and show that the sum of these upper bounds is less than $|W'|-\ell$.
	
	Suppose $\Lambda \subseteq [\ell]$ satisfies $|\Lambda| < \Delta -1$, and suppose $i \in [k]$. By the choice of $C$ and $p$ we have for every $i \in [k]$ 
	\begin{equation}\label{eq:common:sizeNGa}
	 \big|\bigcap_{j\in \Lambda}\NGa(w_j,V_i')\big|\geBy{\ref{cnl:cl:NGaVp}} (1-\eps_0)^{|\Lambda|}p^{|\Lambda|}|V_i'|
	\overset{|\Lambda| < \Delta-1}\geq (1-\eps_0)^{\Delta-2} p^{\Delta-2} \frac{n}{8ktr} \geq C p^{-2} \log n\,.
	\end{equation}
	We also have $|W'| \geq \Ca p^{-1} \log n$ by \eqref{eq:sizeW} and
  $\big(\bigcap_{j\in \Lambda} \NGa(w_j,V_i'), W'\big)$ is an
  $(\epsaa_{|\Lambda|}, d/2, p)_G$-lower-regular pair by~\ref{cnl:cl:Wreg}.
  Since the good event of Lemma~\ref{lem:OSRIL} with input
  $\epsaa_{|\Lambda|+1}$ and $\tfrac{d}{2}$ occurs, there exist at most $C
  p^{-1} \log n$ vertices $w$ in $W'$ such that
  \[\bigg(\bigcap_{j\in
      \Lambda}\NGa(w_j,V_i') \cap \NGa(w), W'\bigg) = \bigg(\bigcap_{j\in
      \Lambda}\NGa(w_j,V_i') \cap \NGa(w,V_i'),W'\bigg)\]
  is not $(\epsaa_{|\Lambda|+1},\frac d 2, p)_G$-lower-regular. Summing over all possible choices of $\Lambda \subseteq [l]$ and $i \in [k]$, there are at most $2^\Delta k^2 C p^{-1} \log n$ vertices $w$ in $W'$ such that $(w_1, \ldots, w_l, w)$ does not satisfy~\ref{cnl:cl:Wreg}.
	
	Moving on to~\ref{cnl:cl:NGVp}, let $\Lambda\subset[\ell]$ and $i\in[k]$ be given. We have
	\begin{align*}
	 \big|\bigcap_{j\in\Lambda}N_G(w_j,V'_i)\big|&\geBy{\ref{cnl:cl:NGVp}}\Big(\frac{d}{4}\Big)^{|\Lambda|}p^{|\Lambda|}|V'_i|\quad\text{and}\\
	 \big|\bigcap_{j\in\Lambda}N_\Gamma(w_j,V'_i)\big|&\leBy{\ref{cnl:cl:NGaVp}}(1+\eps_0)^{|\Lambda|}p^{|\Lambda|}|V'_i|\,.
	\end{align*}
	By choice of $\eps_0$ and $\epsaa_{|\Lambda|}$, we thus have $\big|\bigcap_{j\in\Lambda}N_G(w_j,V'_i)\big|\ge\epsaa_{|\Lambda|}\big|\bigcap_{j\in\Lambda}N_\Gamma(w_j,V'_i)\big|$.
	Now by~\ref{cnl:cl:Wreg}, the pair $\big(W',\bigcap_{j\in\Lambda}N_\Gamma(w_j,V'_i)\big)$ is $\big(\epsaa_{|\Lambda|},\tfrac{d}{2},p\big)_G$-lower-regular, and thus the number of vertices $w\in W'$ such that 
	\[\big|N_G(w,V'_i)\cap\bigcap_{j\in\Lambda}N_G(w_j,V'_i)\big|<\Big(\frac{d}{4}\Big)^{|\Lambda|+1}p^{|\Lambda|+1}|V'_i|\]
	is at most $\epsaa_{|\Lambda|}|W'|\le\epsaa_{\Delta}|W'|$. Summing over the choices of $\Lambda\subset[\ell]$ and $i\in[k]$, the number of $w\in W'$ violating~\ref{cnl:cl:NGVp} is at most $2^\Delta k\epsaa_{\Delta}|W'|$.
	
	For~\ref{cnl:cl:NGaVp}, given $\Lambda\subset[\ell]$ and $i\in[k]$, by~\ref{cnl:cl:NGaVp} we have
	\[ \big|\bigcap_{j\in\Lambda}\NGa(w_j,V_i')\big| = (1\pm\eps_0)^{|\Lambda|}p^{|\Lambda|} |V_i'|\,,\]
	and by choice of $\eps_0$ and $p$, in particular $\big|\bigcap_{j\in\Lambda}\NGa(w_j,V_i')\big|\ge C p^{-1}\log n$. Since the good event of Proposition~\ref{prop:chernoff} occurs, the number of vertices $w\in W'$ such that $\big|\NGa(w,V'_i)\cap\bigcap_{j\in\Lambda}\NGa(w_j,V_i')\big|$ is either smaller than $(1-\eps_0)^{|\Lambda|+1}p^{|\Lambda|+1}|V_i'|$
	or larger than $(1+\eps_0)^{|\Lambda|+1}p^{|\Lambda|+1}|V_i'|$ is at most $2C p^{-1}\log n$. Summing over the choices of $\Lambda\subset[\ell]$ and of $i\in[k]$, we conclude that at most $2^{\Delta+1}kC p^{-1}\log n$ vertices of $W'$ violate~\ref{cnl:cl:NGaVp}. Since $n\ge |V_i|\ge|V'_i|$, the same calculation shows that a further at most $2^{\Delta+1}kC p^{-1}\log n$ vertices of $W'$ violate~\ref{cnl:cl:NGaV}, and at most $2^{\Delta+1}kC p^{-1}\log n$ vertices of $W'$ violate~\ref{cnl:cl:NGa}.
	
	Finally, we come to~\ref{cnl:cl:Vreg}. Suppose we are given
  $\Lambda,\Lambda'\subset[\ell]$ and distinct $i,i'\in[k]$. Suppose that
  $|\Lambda|\le\Delta-2$ and $|\Lambda'|\le\Delta-1$. We wish to show that for
  most vertices $w\in W'$, the pair $\big(\NGa(w,V_i)\cap
  \bigcap_{j\in\Lambda}\NGa(w_j,V_i),\bigcap_{j\in\Lambda}\NGa(w_j,V'_i)\big)$
  is $\big(\epsa_{|\Lambda|+1,|\Lambda'|},d,p\big)_G$-lower-regular, and
  furthermore, if $\Delta\ge 3$ and $|\Lambda'|\le\Delta-2$, that the pair
  \[\bigg(\NGa(w,V_i)\cap
  \bigcap_{j\in\Lambda}\NGa(w_j,V_i),\NGa(w,V_{i'})\cap\bigcap_{j\in\Lambda}\NGa(w_j,V'_i)\bigg)\]
  is $\big(\epsa_{|\Lambda|+1,|\Lambda'|+1},d,p\big)_G$-lower-regular.
	
	By~\ref{cnl:cl:NGaV}, and by choice of $\eps_0$, $C$ and $p$, we have
	\begin{align*}
	 \big|\bigcap_{j\in\Lambda}\NGa(w_j,V_i)\big|&\ge(1-\eps_0)^{|\Lambda|} p^{|\Lambda|}|V_i|\ge C p^{|\Lambda|-\Delta}\log n\quad\text{and}\\
	 \big|\bigcap_{j\in\Lambda'}\NGa(w_j,V_{i'})\big|&\ge(1-\eps_0)^{|\Lambda'|} p^{|\Lambda'|}|V_{i'}|\ge C p^{|\Lambda'|-\Delta}\log n\,.
	\end{align*}
	By~\ref{cnl:cl:Vreg}, the pair $\big(\bigcap_{j\in\Lambda}\NGa(w_j,V_i),\bigcap_{j\in\Lambda}\NGa(w_j,V'_i)\big)$ is $\big(\epsa_{|\Lambda|,|\Lambda'|},d,p\big)_G$-lower-regular. Since the good event of Lemma~\ref{lem:OSRIL} with input $\epsa_{|\Lambda|+1,|\Lambda'|}$ and $d$ occurs, there are at most $C p^{-1}\log n$ vertices $w$ of $W'$ such that $\big(\NGa(w,V_i)\cap\bigcap_{j\in\Lambda}\NGa(w_j,V_i),\bigcap_{j\in\Lambda}\NGa(w_j,V'_i)\big)$ is not $\big(\epsa_{|\Lambda|+1,|\Lambda'|},d,p\big)_G$-lower-regular. Furthermore, if $|\Lambda'|\le\Delta-2$, then since the good event of Lemma~\ref{lem:TSRIL} with input $\epsa_{|\Lambda|+1,|\Lambda'|+1}$ and $d$ occurs, there are at most $C p^{-2}\log n$ vertices $w$ of $W'$ such that 
	\[\Big(\NGa(w,V_i)\cap\bigcap_{j\in\Lambda}\NGa(w_j,V_i),\NGa(w,V_{i'})\cap \bigcap_{j\in\Lambda}\NGa(w_j,V'_i)\Big)\text{ is not }\big(\epsa_{|\Lambda|+1,|\Lambda'|},d,p\big)_G\text{-lower-regular.}\]
	Observe that if $\Delta=2$ the property~\ref{cnl:cl:Vreg} does not require this pair to be lower-regular. Summing over the choices of $\Lambda,\Lambda'\subset[\ell]$ and $i,i'\in[k]$, we conclude that if $\Delta=2$ then at most $2^{2\Delta}k^2C p^{-1}\log n$ vertices $w$ of $W'$ cause~\ref{cnl:cl:Vreg} to fail, while if $\Delta\ge 3$, at most $2^{2\Delta}k^2C(p^{-1}+p^{-2})\log n$ vertices $w$ of $W'$ violate~\ref{cnl:cl:Vreg}.
	
	Summing up, if $\Delta=2$ then at most
	\begin{equation}\label{eq:common:bad2}
	 2^{\Delta}k^2C p^{-1}\log n+2^\Delta k\epsaa_{\Delta}|W'|+3\cdot 2^{\Delta+1}kC p^{-1}\log n+2^{2\Delta}k^2C p^{-1}\log n
	\end{equation}
	vertices $w$ of $W'$ cannot be chosen as $w_{\ell+1}$. By choice of $\Ca$ and $\epsaa_{\Delta}$, and by choice of $p$, this is at most $\tfrac12|W'|$, so that there exists a vertex of $W'$ which can be chosen as $w_{\ell+1}$, as desired. If on the other hand $\Delta\ge 3$, then at most
	\begin{equation}\label{eq:common:bad3}
	 2^{\Delta}k^2C p^{-1}\log n+2^\Delta k\epsaa_{\Delta}|W'|+3\cdot 2^{\Delta+1}kC p^{-1}\log n+2^{2\Delta}k^2C (p^{-1}+p^{-2})\log n
	\end{equation}
	vertices of $W'$ cannot be chosen as $w_{\ell+1}$. Again by choice of $\Ca$, $\epsaa_{\Delta}$ and $p$, this is at most $\tfrac12|W'|$, and again we therefore can choose $w_{\ell+1}$ satisfying~\ref{cnl:cl:Wreg}--\ref{cnl:cl:Vreg} as desired.
\end{claimproof}

Finally, let us argue why the lemma is a consequence of Claim~\ref{claim:common}. Let $(w_1,\ldots, w_\Delta) \in \binom{W'}{\Delta}$ be a tuple satisfying~\ref{cnl:cl:Wreg}--\ref{cnl:cl:Vreg}. By~\ref{cnl:cl:NGVp}, for any $\Lambda\subset[\ell]$ and $i\in[k]$ we have
\[\Big|\bigcap_{j\in\Lambda}N_G(w_j,V_i)\Big|\ge \Big(\frac{d}{4}\Big)^{|\Lambda|} p^{|\Lambda|}|V_i'|\ge\Big(\frac{d}{4}\Big)^\Delta p^{|\Lambda|}\frac{|V_i|}{2t_1}\ge\alpha p^{|\Lambda|}|V_i|\,,\]
as required for~\ref{cnl:Gsize}. Properties~\ref{cnl:Gasizen},~\ref{cnl:Gasize} and~\ref{cnl:Nreg} are respectively~\ref{cnl:cl:NGa},~\ref{cnl:cl:NGaV} and~\ref{cnl:cl:Vreg}, by choice of $\eps_0$.
\end{proof}

\section{The balancing lemma}
\label{sec:prooflembalancing}
The statement of Lemma~\ref{lem:balancing} gives us a partition of $V(G)$ with parts $\big(V_{i,j}\big)_{i\in[r],j\in[k]}$, and a collection of `target integers' $\big(n_{i,j}\big)_{i\in[r],j\in[k]}$, with each $n_{i,j}$ close to $|V_{i,j}|$, and with $\sum n_{i,j}=\sum |V_{i,j}|$. Our aim is to find a partition of $V(G)$ with parts $\big(V'_{i,j}\big)_{i\in[r],j\in[k]}$ such that $|V'_{i,j}|=n_{i,j}$  for each $i,j$. This partition is required to maintain similar regularity properties as the original partition, while not substantially changing common neighbourhoods of vertices.

There are two steps to our proof. In a first step, we correct \emph{global imbalance}, that is, we find a partition $\tcV$ which maintains all the desired properties and which has the property that $\sum_i |\tV_{i,j}|=\sum_i n_{i,j}$ for each $j\in[k]$. To do this, we identify some $j^\ast$ such that $\sum_i |V_{i,j^\ast}|>\sum_i n_{i,j^\ast}$ and $j'$ such that $\sum_i |V_{i,j'}|<\sum_i n_{i,j'}$. We move $\sum_i |V_{i,j^\ast}|- n_{i,j^\ast}$ vertices from $V_{1,j^\ast}$ to some cluster $V_{i',j'}$, maintaining the desired properties, and repeat this procedure until no global imbalance remains.

In a second step, we correct \emph{local imbalance}, that is, for each $i=1,\dots,r-1$ sequentially, and for each $j\in[k]$, we move vertices between $\tV_{i,j}$ and $\tV_{i+1,j}$, maintaining the desired properties, to obtain the partition $\cV'$ such that $|V'_{i,j}|=n_{i,j}$ for each $i,j$. Observe that because $\tcV$ is globally balanced, once we know $|V'_{i,j}|=n_{i,j}$ for each $i\in[r-1]$ and each $j\in[k]$ we are guaranteed that $|V'_{r,j}|=n_{r,j}$ for each $j\in[k]$.

The proof of the lemma then comes down to showing that we can move vertices and maintain the desired properties. Because we start with a partition in which $V_{i,j}$ is very close to $n_{i,j}$ for each $i$ and $j$, the total number of vertices we move in any step is at most the sum of the differences, which is much smaller than any $n_{i,j}$. The following lemma shows that we can move any small (compared to all $n_{i,j}$) number of vertices from one part to another and maintain the desired properties.

\begin{lemma}\label{lem:smallmove}
 For all integers $k, r_1, \Delta \geq 1$, and reals $d>0$ and $0<\eps<1/2k$ as well as $0 < \xi < 1/(100kr_1^3)$, there exists $\Ca>0$  such that the following holds for all sufficiently large $n$.
 
 Let $\Gamma$ be a graph on vertex set $[n]$, and let $G$ be a not necessarily spanning subgraph. Let $X,Z_1,\ldots,Z_{k-1} \subseteq V(G)$ be pairwise disjoint subsets, each of size at least $n/(16kr_1)$, such that $(X,Z_i)$ is $(\eps,d,p)_G$-lower-regular for each $i$. Then for each $1\le m\le 2r_1^2\xi n$, there exists a set $S$ of $m$ vertices of $X$ with the following properties.
 \begin{enumerate}[label=\itmarab{SM}]
  \item\label{smallmove:degG} For each $v\in S$ we have $\deg_G(v;Z_i)\ge(d-\eps)p|Z_i|$ for each $i\in[k-1]$, and
  \item\label{smallmove:I} for each $1\le s\le\Delta$ and every collection of vertices $v_1,\ldots,v_s\in[n]$ we have
  \[\deg_\Gamma(v_1,\ldots,v_s;S)\le 100kr_1^3\xi\deg_\Gamma(v_1,\ldots,v_s; X)+\frac{1}{100}\Ca\log n\,.\]
 \end{enumerate}
\end{lemma}

\begin{proof}
 Given $k$, $r_1$, $\Delta$, $d$, $\xi$ and $\eps$, let $C$ be returned by Lemma~\ref{lem:hypgeo} for input $\xi$ and $\Delta$. We set $\Ca=100C$. Given $\Gamma$, $G$ and $X$, $Y$, $Z_1,\ldots,Z_{k-1}$, let $X'$ be the set of vertices $v\in X$ such that $\deg_G(v;Z_i)\ge(d-\eps)p|Z_i|$ for each $i\in[k-1]$. Because each pair $(X,Z_i)$ for $i\in[k-1]$ is $(\eps,d,p)_G$-lower-regular, we have $|X'|\ge |X|-k\eps|X|\ge|X|/2$. 
 
 We now apply Lemma~\ref{lem:hypgeo}, with input $\xi$, $\Delta$, $W=X'$ and the sets $T_i$ being the sets $N_\Gamma(v_1,\ldots,v_s;X')$ for each $1\le s\le\Delta$ and $v_1,\ldots,v_s\in[n]$, to choose a set $S$ of size $m\leq 2r_1^2\xi n\leq |X'|$ in $X'$. We then have
 \begin{align*}
  \deg_\Gamma(v_1,\ldots,v_s;S)&\le \left(\frac{2r_1^2\xi n}{|X'|}+\xi\right)\deg_\Gamma(v_1,\ldots,v_s;X')+C\log n\\
  &\le 100kr_1^3\xi\deg_\Gamma(v_1,\ldots,v_s;X)+\frac{1}{100}\Ca\log n\,,
 \end{align*}
 where the final inequality is by choice of $\Ca$, and since $|X'|\ge|X|/2\ge n/(32kr_1)$. Thus the set $S$ satisfies~\ref{smallmove:I}, and since $S\subseteq X'$ we have~\ref{smallmove:degG}.
\end{proof}

We now prove the balancing lemma.

\begin{proof}[Proof of Lemma~\ref{lem:balancing}]
Given integers $k, r_1, \Delta \geq 1$ and reals $\gamma, d >0$ and $0< \eps < \min\{d, 1/(2k)\}$, we set 
\[\xi =10^{-15} \eps^4d/(k^3r_1^5).\]
 Let $\Ca_1$ be returned by Lemma~\ref{lem:smallmove} with input $k$, $r_1$, $\Delta$, $d$, $\eps/4$ and $\xi$, and let $\Ca_2$ be returned by Lemma~\ref{lem:smallmove} with input $k$, $r_1$, $\Delta$, $d$, $3\eps/4$ and $\xi$. We set $\Ca = \max\{\Ca_1, \Ca_2\}$. 

%Given integers $k, r_1, \Delta \geq 1$ as well as reals $\gamma, d >0$ and $0< \eps < \min\{d, 1/(2k)\}$, we choose $\Ca$ large enough, and $0<\xi<\tfrac{\eps^2d}{100000kr_1^4}$ small enough, to apply Lemma~\ref{lem:smallmove} with input $k$, $r_1$, $\Delta$, $d$ and $\eps/4$, and with input $k$, $r_1$, $\Delta$, $d$ and $3\eps/4$.

Now suppose that $p\ge\Ca\big(\frac{\log n}{n}\big)^{1/2}$, that $10\gamma^{-1}\le r\le r_1$, and that graphs $\Gamma$ and $G$, a partition $\cV$ of $V=V(G)$, and graphs $R^k_r$, $B^k_r$  and $K^k_r$ on $[r]\times [k]$ as in the statement of Lemma~\ref{lem:balancing} are given.

\textbf{First stage (global imbalance):} 

We use the following algorithm.

\begin{algorithm}
\caption{Global balancing}\label{alg:global}
 \While{$\exists j\in[k]$ such that $\sum_{i\in[r]}|V_{i,j}|-n_{i,j}\neq 0$}{
  Choose $j^\ast\in[k]$ maximising $\sum_{i\in[r]}|V_{i,j^\ast}|-n_{i,j^\ast}$ \;
  Choose $i'>1$ such that $V_{i',j}$ is not changed and $(V_{1,j^\ast},V_{i',j})$ is $\big(\tfrac{\eps}{4},d,p\big)_G$-lower-regular $\forall j\in[k]$\;
  Choose $j'\in[k]$ such that $\sum_{i\in[r]}|V_{i,j'}|-n_{i,j'}<0$ \;
  Select $S\subset V_{1,j^\ast}$ with $|S|=\sum_{i\in[r]}|V_{i,j^\ast}|-n_{i,j^\ast}$ \;
  Set $V_{1,j^\ast}:=V_{1,j^\ast}\setminus S$ and $V_{i',j'}=V_{i',j'}\cup S$ \;
  Flag $V_{1,j^\ast}$ and $V_{i',j'}$ as changed \;
 }
\end{algorithm}

In each step where we select $S$, we make use of Lemma~\ref{lem:smallmove} to do so, with input $k$, $r_1$, $\Delta$, $d$, and $\eps/4$, with $X=V_{1,j^\ast}$ and with the $Z_1,\ldots,Z_{k-1}$ being the $V_{i',j''}$ with $j''\neq j'$.

We claim that the algorithm completes successfully, in other words that each of
the choices is possible, and that Lemma~\ref{lem:smallmove} is always applicable. In each While loop, since $\sum_{i,j}|V_{i,j}|-n_{i,j}=0$ and since the While condition is satisfied, $j^\ast$ satisfies $\sum_{i\in[r]}|V_{i,j^\ast}|-n_{i,j^\ast}>0$.

Observe that the While loop is run at most $k$ times, since at the end of the While loop in which we selected some $j=j^\ast$ we have $\sum_{i\in[r]}|V_{i,j^\ast}|-n_{i,j^\ast}=0$ and therefore we do not select $j$ as either $j^\ast$ or $j'$ in future iterations. It follows that the number of $V_{i,j}$ flagged as changed never exceeds $2k$. Now the set $V_{1,j^\ast}$ has degree at least $\big(k-1+\tfrac{\gamma k}{2}\big)r$ in $R^k_r$, and so there are at least $\gamma k r/2$ indices $i\in[r]$ such that $V_{1,j^\ast}$ is adjacent to each $V_{i,j}$ in $R^k_r$. Since $\gamma k r/2>3k$, in particular we can choose $i'$ such that $V_{1,j^\ast}$ is adjacent to each $V_{i',j}$ in $R^k_r$ and no $V_{i',j}$ is flagged as changed. It follows that each pair $(V_{1,j^\ast},V_{i',j})$ is $\big(\tfrac{\eps}{4},d,p\big)_G$-lower-regular and thus it is possible to choose $i'$. It is possible to choose $j'$ since the While condition holds. Finally, we need to show that Lemma~\ref{lem:smallmove} is always applicable with the given parameters. In each application, the sets denoted $X,Z_1,\ldots,Z_{k-1}$ are parts of the partition $\cV$ (so they were not changed by the algorithm yet). It follows that each set has size at least $n/(8kr)>n/(16kr_1)$. Since $\cV$ is $\big(\tfrac\eps4,d,p)$-lower-regular on $B_k^r$, the pairs $(X,Z_1),\dots,(X,Z_{k-1})$ are $\big(\tfrac\eps4,d,p)$-lower-regular as required. Finally, by choice of $j^\ast$ we see that the sizes of the sets $S$ we select in each step are decreasing, so it is enough to show that in the first step we have $|S|\le r\xi n$, which follows from~\ref{lembalancing:sizes}. Thus Lemma~\ref{lem:smallmove} is applicable in each step, and we conclude that the algorithm indeed completes. We denote the resulting vertex partition by $\tcV=\{\tV_{i,j}\}_{i\in[r],j\in[k]}$.

\begin{claim} We have the following properties.
\begin{enumerate}[label=\itmarab{P}]
 \item\label{claimpVt:sizeparts} For each $i\in[r]$ and $j\in[k]$ we have $\big||\tV_{i,j}|-n_{i,j}\big|\le 2r\xi n$,
 \item\label{claimpVt:regular} $\tcV$ is $\big(\tfrac{\eps}{2},d,p\big)_G$-lower-regular on $R^k_r$ and $\big(\tfrac{\eps}{2},d,p\big)_G$-super-regular on  $K^k_r$,
 \item\label{claimpVt:NGa} For each $i\in[r]$, $j\in[k]$ and $1\le s\le\Delta$ and $v_1,\ldots,v_s\in[n]$ we have
 \[|\NGa(v_1,\ldots,v_s,\tV_{i,j}) \symd \NGa(v_1,\ldots,v_s,V_{i,j})| \leq 100kr_1^3\xi\deg_\Gamma\big(v_1,\ldots,v_s;V(G)\big)+\frac{1}{100}\Ca\log n\,.\]
\end{enumerate}
\end{claim}
\begin{claimproof}
 Observe that vertices were removed from or added to each $V_{i,j}$ to form $\tV_{i,j}$ at most once in the running of Algorithm~\ref{alg:global}, and the number of vertices added or removed was at most $r\xi n$. Since $|V_{i,j}|$ satisfies~\ref{lembalancing:sizes}, we conclude that~\ref{claimpVt:sizeparts} holds. Furthermore, the vertices added to or removed from $V_{i,j}$ satisfy~\ref{smallmove:I} and therefore~\ref{claimpVt:NGa} holds.
 
 Since each set $V_{i,j}$ has size at least $n/(8kr)$, we can apply Proposition~\ref{prop:subpairs3} with $\mu=\nu=8kr^2\xi$ to each edge of $R^k_r$, concluding that $\tcV$ is $\big(\tfrac{\eps}{2},d,p\big)_G$-lower-regular on $R^k_r$ since $\tfrac{\eps}{4}+4\sqrt{8kr^2\xi}<\tfrac{\eps}{2}$. Now for any $i\in[r]$ and $j\in[k]$, consider $v\in\tV_{i,j}$. If $v\not\in V_{i,j}$, then we applied Lemma~\ref{lem:smallmove} to select $v$, and when we did so no $V_{i,j'}$ was flagged as changed by Algorithm~\ref{alg:global}. Thus by~\ref{smallmove:degG} we have
 \[\deg_G(v;\tV_{i,j'})=\deg_G(v;V_{i,j'})\ge\Big(d-\frac{\eps}{4}\Big)p|V_{i,j'}|=\Big(d-\frac{\eps}{4}\Big)p|\tV_{i,j'}|\]
 for each $j'\neq j$, since $V_{i,j}$ is then flagged as changed and thus $V_{i,j'}=\tV_{i,j'}$ for each $j'\neq j$. If on the other hand $v\in V_{i,j}$, then by~\ref{lembalancing:regular1} we started with $\deg_G(v;V_{i,j'})\ge\big(d-\tfrac{\eps}{4}\big)p|V_{i,j'}|$. By~\ref{smallmove:I} and~\ref{lembalancing:gamma1}, we have
 \[\deg_G(v;\tV_{i,j'})\ge \Big(d-\frac{\eps}{4}\Big)p|V_{i,j'}|-\frac{\eps^2}{1000kr_1}\Big(1+\frac{\eps}{4}\Big)p|V_{i,j'}|-\frac{1}{100}\Ca\log n
 \ge\Big(d-\frac{\eps}{2}\Big)p|\tV_{i,j'}|\,,\]
 where the final inequality follows by choice of $n$ sufficiently large and since $|\tV_{i,j'}|\le|V_{i,j'}|+r\xi n\le \big(1+\tfrac{\eps d}{100}\big)|V_{i,j'}|$. We conclude that $\tcV$ is $\big(\tfrac{\eps}{2},d,p\big)$-super-regular on $K^k_r$, giving~\ref{claimpVt:regular}.
\end{claimproof}

\textbf{Second stage (local imbalance):}

We use Algorithm~\ref{alg:local} to correct the local imbalances in $\tcV$.

\begin{algorithm}
\caption{Local balancing}\label{alg:local}
 \ForEach{$i=1,\dots,r-1$}{
  \ForEach{$j=1,\dots,k$}{
   \If{$|\tV_{i,j}|>n_{i,j}$}{
    Select $S\subset \tV_{i,j}$ with $|S|=|\tV_{i,j}|-n_{i,j}$ \;
    Set $\tV_{i,j}:=\tV_{i,j}\setminus S$ and $\tV_{i+1,j}:=\tV_{i+1,j}\cup S$ \; 
   }
   \Else{
    Select $S\subset \tV_{i+1,j}$ with $|S|=n_{i,j}-|\tV_{i,j}|$ \;
    Set $\tV_{i+1,j}:=\tV_{i+1,j}\setminus S$ and $\tV_{i,j}:=\tV_{i,j}\cup S$ \; 
   }
  }
 }
\end{algorithm}

Again, in each step when we select $S$ we make use of Lemma~\ref{lem:smallmove} to do so. If we select $S$ from $\tV_{i,j}$, then we use input $k$, $r_1$, $d$, $3\eps/4$ and $\xi$ with $X=\tV_{i,j}$ and the sets $Z_1,\ldots,Z_{k-1}$ being $\tV_{i+1,j'}$ for $j'\neq j$. If on the other hand we select $S$ from $\tV_{i+1,j}$, then we use input $k$, $r_1$, $d$ and $3\eps/4$, with $X=\tV_{i+1,j}$ and the sets $Z_1,\ldots,Z_{k-1}$ being $\tV_{i,j'}$ for $j'\neq j$.

We claim that Lemma~\ref{lem:smallmove} is always applicable. To see that this is true, observe first that the number of vertices which we move between any $\tV_{i,j}$ and $\tV_{i+1,j}$ in a given step is by~\ref{claimpVt:sizeparts} bounded by $2k^2r^2\xi n$. We change any given $\tV_{i,j}$ at most twice in the running of the algorithm, so that in total at most $4k^2r^2\xi n$ vertices are changed. In particular, we maintain $|\tV_{i,j}|\ge n/(16kr_1)$ throughout, and, by Proposition~\ref{prop:subpairs3}, with input $\mu=\nu=\tfrac{4r^2\xi n}{n/(16kr_1)}<100r_1^3k\xi$, and using~\ref{claimpVt:regular}, we maintain the property that any pair in $R^k_r$, and in particular any pair in $B^k_r$, is $\big(\tfrac{3\eps}{4},d,p\big)$-lower-regular throughout.  This shows that Lemma~\ref{lem:smallmove} is always applicable, and therefore the algorithm completes and returns a partition $\cV'$. We claim that this is the desired partition. We need to check that~\ref{lembalancing:sizesout}---\ref{lembalancing:gammaout} hold.

Since for each $j\in[k]$ we have $\sum_i|V'_{i,j}|=\sum_i|\tV_{i,j}|=\sum_in_{i,j}$, and since $|V'_{i,j}|=n_{i,j}$ for each $i\in[r-1]$ and $j\in[k]$, we conclude that $|V'_{i,j}|=n_{i,j}$ for all $i$ and $j$, giving~\ref{lembalancing:sizesout}.

For the first part of~\ref{lembalancing:regular}, we have justified that we maintain $\big(\tfrac{3\eps}{4},d,p\big)_G$-lower-regularity on $R^k_r$ throughout the algorithm. For the second part, we need to show that for each $i\in[r]$ and $j\neq j'\in[k]$, and each $v\in V'_{i,j}$, we have $\deg_G(v;V'_{i,j'})\ge(d-\eps)p|V'_{i,j'}|$. If $v\in\tV_{i,j}$, then by~\ref{claimpVt:regular} we have $\deg_G(v;\tV_{i,j'})\ge\big(d-\tfrac{\eps}{2}\big)p|\tV_{i,j'}|$. We change $\tV_{i,j'}$ at most twice to obtain $V'_{i,j'}$, both times by adding or removing vertices satisfying~\ref{smallmove:I}. As in the proof of Claim~\ref{claimpVt:NGa} above, using~\ref{lembalancing:gamma1} and~\ref{claimpVt:NGa} we obtain $\deg_G(v;\tV_{i,j'})\ge(d-\eps)p|V'_{i,j}|$ as desired. If $v\not\in\tV_{i,j}$, then it was added to the set $\tV_{i,j}$ by Algorithm~\ref{alg:local}, and $\tV_{i,j'}$ was changed at most twice thereafter. Again, using~\ref{smallmove:degG},~\ref{smallmove:I},~\ref{lembalancing:gamma1} and~\ref{claimpVt:NGa} we obtain $\deg_G(v;\tV_{i,j'})\ge(d-\eps)p|V'_{i,j}|$ as desired.

Now~\ref{lembalancing:symd} holds since the total number of vertices moved in Algorithm~\ref{alg:global} is at most $k^2r\xi n$, in Algorithm~\ref{alg:local} at most $4k^2r^2\xi n$ vertices are changed in each cluster, and by choice of $\xi$. To see that~\ref{lembalancing:inheritance} holds, observe that by~\ref{lembalancing:gamma1},~\ref{claimpVt:NGa} and~\ref{smallmove:I} we have
\[\big|N_\Gamma(v;V_{i,j})\Delta N_\Gamma(v;V'_{i,j}) \big|\le \frac{\eps^2}{100kr_1}\deg_\Gamma\big(v;V(G)\big)+\frac{1}{10}\Ca\log n\le \frac{\eps^2}{50}\deg_\Gamma(v;V_{i,j})\]
where the final inequality follows by choice of $p$ and of $n$ sufficiently large. Using~\ref{lembalancing:inheritance1}, we can apply Proposition~\ref{prop:subpairs3}, with $\mu=\nu=\tfrac{\eps^2}{50}$, to deduce~\ref{lembalancing:inheritance}.

For~\ref{lembalancing:gammaout}, observe that for any given $i\in[r]$ and $j\in[k]$ we change $\tV_{i,j}$ at most twice in the running of Algorithm~\ref{alg:local}, both times either adding or removing a set satisfying~\ref{smallmove:I}. By~\ref{claimpVt:NGa} and choice of $\xi$, we conclude that~\ref{lembalancing:gammaout} holds.

Finally, suppose that for any two disjoint vertex sets $A,A'\subset V(\Gamma)$ with $|A|,|A'|\ge \tfrac{1}{50000kr_1}\eps^2\xi pn$ we have $e_\Gamma(A,A')\le \big(1+\tfrac{1}{100}\eps^2\xi\big)p|A||A'|$. In each application of Proposition~\ref{prop:subpairs3} we have $\mu,\nu\ge\tfrac{1}{50}\eps^2\xi$, and, and if we have `regular' in place of `lower-regular' in~\ref{lembalancing:regular1}, and~\ref{lembalancing:inheritance1}, we always apply Proposition~\ref{prop:subpairs3} to a regular pair with sets of size at least $\tfrac{\eps}{1000kr_1}pn$, so it returns regular pairs for~\ref{lembalancing:regular}, and~\ref{lembalancing:inheritance}, as desired.
\end{proof}

%%%%%%%%%%%%%%%%%%%%%%%%%%%%%%%%%%%%%%%%%%%%%%%%%%%%%%%%%

						% Proof of the main theorem

%%%%%%%%%%%%%%%%%%%%%%%%%%%%%%%%%%%%%%%%%%%%%%%%%%%%%%%%%

\section{The Bandwidth Theorem in random graphs}
\label{sec:proofmain}

Before embarking on the proof, we first recall from the proof overview (Section~\ref{subsec:over}) the main ideas. Given $G$, we first use the lemma for~$G$ (Lemma~\ref{lem:G}) to find a lower-regular partition of $V(G)$, with an extremely small exceptional set $V_0$, and whose reduced graph $R^k_r$ contains a spanning backbone graph $B^k_r$, on whose subgraph $K^k_r$ the graph $G$ is super-regular and has one- and two-sided inheritance. Given this, and $H$ together with a $(z,\beta)$-zero-free $(k+1)$-colouring, we use the lemma for~$H$ (Lemma~\ref{lem:H2}) to find a homomorphism $f$ from $V(H)$ to $R^k_r$ almost all of whose edges are mapped to $K^k_r$ and in which approximately the `right' number of vertices of $H$ are mapped to each vertex of $R^k_r$. At this point, if $V_0$ were empty, and if the `approximately' were exact, we would apply the sparse blow-up lemma (Lemma~\ref{thm:blowup}) to obtain an embedding of $H$ into $G$.

Our first aim is to deal with $V_0$. We do this one vertex at a time. Given $v\in V_0$, we choose $x\in V(H)$ from the first $\beta n$ vertices of the supplied bandwidth order $\cL$ which is not in any triangles. We embed $x$ to $v$. We then embed the neighbours of $x$ to carefully chosen neighbours of $v$, which we obtain using Lemma~\ref{lem:common}. Here we use the fact that $N_H(x)$ is independent. This then fixes a clique of $K^k_r$ to which $N^2_H(x)$ must be assigned, and gives image restrictions in the corresponding parts of the lower-regular partition for these vertices. Since $N^2_H(x)$ may have been assigned by $f$ to some quite different clique in $K^k_r$, we have to adjust $f$ to match. This we can do using the fact, which follows from our assumptions on $\cL$, that $x$ is far from vertices of colour zero.

Now the idea is simply to repeat the above procedure, choosing vertices of $V(H)$ to pre-embed which are widely separated in $H$, until we pre-embedded vertices to all of $V_0$. We end up with a homomorphism $f^*$ from what remains of $V(H)$ to $R^k_r$. It is easy to check that this homomorphism still maps about the right number of vertices of $H$ to each vertex of $R^k_r$, simply because $V_0$ is small. We now apply the Balancing Lemma (Lemma~\ref{lem:balancing}) to correct the sizes of the clusters to match $f^*$, and complete the embedding of $H$ using the Sparse Blow-up Lemma (Lemma~\ref{thm:blowup}).

There are two difficulties with this idea, the `subtleties' mentioned in the proof overview (Section~\ref{subsec:over}). First, if $\Delta=2$ we might have $|V_0|\gg pn$, so that we should be worried that at some stage of the pre-embedding we choose $v\in V_0$ and discover most or all of its neighbours have already been pre-embedded to. It turns out to be easy to resolve this: we choose each $v\in V_0$ not arbitrarily, but by taking those which have least available neighbours first. We will show that this is enough to avoid the problem.

More seriously, because we perform the pre-embedding sequentially, we might use up a significant fraction of $N_G(w)$ for some $w\in V(G)$ in the pre-embedding, destroying super-regularity of $G$ on $K^k_r$, or we might use up a significant fraction of some common neighbourhood which defines an image restriction for the sparse blow-up lemma. In order to avoid this, before we begin the pre-embedding we fix a set $S\subset V(G)$ whose size is a very small constant times $n$, chosen using Lemma~\ref{lem:hypgeo} to not have a large intersection with any $N_G(w)$ or with any $\Gamma$-common neighbourhood of at most $\Delta$ vertices of $\Gamma$ (which could define an image restriction). We perform the pre-embedding as outlined above, \emph{except} that we choose our neighbours of each $v$ within $S$. This procedure is guaranteed not to use up neighbourhood sets guaranteed by super-regularity or image restriction sets, since these sets are all contained in $V\setminus V_0$ and even using up all of $S$ would not be enough to do damage.

\begin{proof}[Proof of Theorem~\ref{thm:maink}]
Given $\gamma>0$, $\Delta\ge 2$ and $k\ge 2$, let $d$ be returned by Lemma~\ref{lem:G}, with input $\gamma$, $k$ and $r_0:=10\gamma^{-1}$. Let $\alpha$ be returned by Lemma~\ref{lem:common} with input $d$, $k$ and $\Delta$. We set $D=\Delta$, and let $\eBL>0$ and $\rho>0$ be returned by Lemma~\ref{thm:blowup} with input $\Delta$, $\Delta_{R'}=3k$, $\Delta_J=\Delta$, $\vartheta=\tfrac{1}{100D}$, $\zeta=\tfrac14\alpha$, $d$ and $\kappa=64$. Next, putting $\epsa=\tfrac18\eBL$ into Lemma~\ref{lem:common} returns $\eps_0>0$. We choose $\eps=\min\big(\eps_0,d,\tfrac{1}{4D}\epsa,\tfrac{1}{100k}\big)$. Putting $\eps$ into Lemma~\ref{lem:G} returns $r_1$. Next, Lemma~\ref{lem:balancing}, for input $k$, $r_1$, $\Delta$, $\gamma$, $d$ and $8\eps$, returns $\xi>0$. We assume without loss of generality that $\xi\le 1/(10kr_1)$, and set $\beta=10^{-12}\xi^2/(\Delta k^4r_1^2)$. Let $\mu=\tfrac{\eps^2}{100000kr}$. Finally, suppose $\Ca$ is large enough for each of these lemmas, for Lemma~\ref{thm:blowup}, for Proposition~\ref{prop:chernoff} with input $\eps$, and for Lemma~\ref{lem:hypgeo} with input $\eps\mu^2$ and $\Delta$.

We set $C=10^{10}k^2r_1^2\eps^{-2}\xi^{-1}\Delta^{2r_1+20}\mu^{-\Delta}\Ca$, and $z=10/\xi$. Given $p\ge C\big(\tfrac{\log n}{n}\big)^{1/\Delta}$, a.a.s.\ $G(n,p)$ satisfies the good events of Lemma~\ref{thm:blowup}, Lemma~\ref{lem:G} and Lemma~\ref{lem:common}, and Proposition~\ref{prop:chernoff}, with the stated inputs. Suppose that $\Gamma=G(n,p)$ satisfies these good events.

Suppose $G\subseteq \Gamma$ is any spanning subgraph with $\delta(G) \geq \big(\tfrac{k-1}{k}+\gamma\big)pn$. Let $H$ be a graph on $n$ vertices with $\Delta(H) \leq \Delta$, and $\mathcal{L}$ be a labelling of vertex set $V(H)$, of bandwidth at most $\beta n$, such that the first $\beta n$ vertices of $\mathcal L$ include $Cp^{-2}$ vertices that are not contained in any triangles of $H$, and such that there exists a $(k+1)$-colouring that is $(z, \beta)$-zero-free with respect to $\mathcal L$, and the colour zero is not assigned to the first $\sqrt{\beta}n$ vertices.  

Applying Lemma~\ref{lem:G} to $G$, with input $\gamma$, $k$, $r_0$ and $\eps$, we obtain an integer $r$ with $10\gamma^{-1}\le kr\le r_1$, a set $V_0\subset V(G)$ with $|V_0|\le \Ca p^{-2}$, a $k$-equitable partition $\cV=\{\Vij\}_{i\in[r],j\in[k]}$ of $V(G)\setminus V_0$, and a graph $R^k_r$ on vertex set $[r]\times[k]$ with minimum degree $\delta(R^k_r)\ge\big(\tfrac{k-1}{k}+\tfrac{\gamma}{2}\big)kr$, such that $K^k_r\subset B^k_r\subset R^k_r$, and such that
\begin{enumerate}[label=\itmarabp{G}{a}]
\item\label{main:Gsize} $\frac{n}{4kr}\leq |\Vij| \leq \frac{4n}{kr}$ for every $i\in[r]$ and $j\in[k]$,
\item\label{main:Greg} $\cV$ is $(\eps,d,p)_G$-lower-regular on $R^k_r$ and $(\eps,d,p)_G$-super-regular on $K^k_r$,
\item\label{main:Ginh} both $\big(\NGa(v, V_{i,j}),V_{i',j'}\big)$ and $\big(\NGa(v, V_{i,j}),\NGa(v, V_{i',j'})\big)$ are $(\eps, d,p)_G$-lower-regular pairs for every $\{(i,j),(i',j')\} \in E(R^k_r)$ and $v\in V\setminus V_0$, and
\item\label{main:Ggam} $|\NGa(v,V_{i,j})| = (1 \pm \eps)p|\Vij|$ for every $i \in [r]$, $j\in [k]$ and every $v \in V \setminus V_0$.
\end{enumerate}

Given $i\in[r]$, because $\delta(R^k_r)>(k-1)r$, there exists $v\in V(R^k_r)$ adjacent to each $(i,j)$ with $j\in[k]$. This, together with our assumptions on $H$, allow us to apply Lemma~\ref{lem:H2} to $H$, with input $D$, $k$, $r$, $\tfrac{1}{10}\xi$ and $\beta$, and with $m_{i,j}:=|V_{i,j}|+\tfrac{1}{kr}|V_0|$ for each $i\in[r]$ and $j\in[k]$, choosing the rounding such that the $m_{i,j}$ form a $k$-equitable integer partition of $n$. Since $\Delta(H)\le\Delta$, in particular $H$ is $\Delta$-degenerate. Let $f\colon V(H) \to [r] \times [k]$ be the mapping returned by Lemma~\ref{lem:H2}, let $W_{i,j} := f^{-1}(i,j)$, and let $X \subseteq V(H)$ be the set of special vertices returned by Lemma~\ref{lem:H2}. For every $i\in [r]$ and $j\in [k]$ we have
\begin{enumerate}[label=\itmarabp{H}{a}] 
\item\label{H:size} $m_{i,j} - \tfrac{1}{10}\xi n  \leq |W_{i,j}| \leq m_{i,j} + \tfrac{1}{10}\xi n$,
\item\label{H:sizeX} $|X| \leq \xi n$,
\item\label{H:edge} $\{f(x),f(y)\} \in E(R^k_r)$  for every $\{x,y\} \in E(H)$,
\item\label{H:special} $y,z\in \bigcup_{j'\in[k]}f^{-1}(i,j')$ for every $x\in f^{-1}(i,j)\setminus X$ and $xy,yz\in E(H)$, and
\item\label{H:v1} $f(x)=\big(1,\sigma(x)\big)$ for every $x$ in the first $\sqrt{\beta}n$ vertices of $\mathcal{L}$.
\end{enumerate}
Lemma~\ref{lem:H2} actually gives a little more, which we do not require for this proof. We let $F$ be the first $\beta n$ vertices of $\mathcal{L}$. By definition of $\mathcal{L}$, in $F$ there are at least $C p^{-2}$ vertices whose neighbourhood in $H$ is independent.

Next, we apply Lemma~\ref{lem:hypgeo}, with input $\eps\mu^2$ and $\Delta$, to choose a set $S\subset V(G)$ of size $\mu n$. We let the $T_i$ of Lemma~\ref{lem:hypgeo} be all sets which are common neighbourhoods in $\Gamma$ of at most $\Delta$ vertices of $\Gamma$, together with the sets $V_{i,j}$ for $i\in[r]$ and $j\in[k]$. The result of Lemma~\ref{lem:hypgeo} is that for any $1\le\ell\le\Delta$ and vertices $u_1,\dots,u_\ell$ of $V(G)$, we have
\begin{equation}\label{eq:intS}
 \begin{split}
  \Big|S\cap\bigcap_{1\le i\le\ell}N_\Gamma(u_i)\Big|&=(1\pm\eps\mu)\mu\Big|\bigcap_{1\le i\le\ell}N_\Gamma(u_i)\Big|\pm \eps\mu p^\ell n\,,\quad \text{and}\\
  \big|S\cap V_{i,j}\big|&\le 2\mu|V_{i,j}|\quad\text{for each $i\in[r]$ and $j\in[k]$,}
 \end{split}
\end{equation}
where we use the fact $p\ge C\big(\tfrac{\log n}{n}\big)^{1/\Delta}$ and choice of $C$ to deduce $\Ca\log n<\eps\mu p^\Delta n$.

Our next task is to create the pre-embedding that covers the vertices of $V_0$. We use the following algorithm, starting with $\phi_0$ the empty partial embedding.
\begin{algorithm}
\caption{Pre-embedding}
\label{alg:pre}
 Set $t:=0$ \;
 \While{$V_0\setminus\im(\phi_t)\neq\emptyset$}{
  \lnl{line:choosev} Let $v_{t+1}\in V_0\setminus\im(\phi_t)$ minimise $\big|\big(N_G(v)\cap S\big)\setminus\im(\phi_t)\big|$ over $v\in V_0\setminus\im(\phi_t)$ \;
  Choose $x_{t+1}\in F$ with $N_H(x)$ independent, with $\dist\big(x_{t+1},\dom(\phi_t)\big)\ge 2r+20$ \;
  Let $N_H(x_{t+1})=\{y_1,\dots,y_\ell\}$ \;
  \lnl{line:choosenbs} Choose $w_1,\dots,w_{\ell}\in \big(N_G(v)\cap S\big)\setminus\im(\phi_t)$ \;
  $\phi_{t+1}:=\phi_t\cup\{x_{t+1}\to v_{t+1}\}\cup\{y_1\to w_1\}\cup\dots\cup\{y_\ell\to w_\ell\}$ \;
  $t:=t+1$ \;
 }
\end{algorithm}
Suppose this algorithm does not fail, terminating with $t=t^*$. The final $\phi_{t^*}$ is an embedding of some vertices of $H$ into $V(G)$ which covers $V_0$ and is contained in $V_0\cup S$. Before we specify how exactly we choose vertices at line~\ref{line:choosenbs}, we justify that the algorithm does not fail. In other words, we need to justify that at every time $t$ there are vertices of $F$ whose neighbourhood is independent and which are not close to any vertices in $\dom(\phi_t)$, and that at every time $t$, the set $\big(N_G(v)\cap S\big)\setminus\im(\phi_t)$ is big. For the first, observe that since $|V_0|\le\Ca p^{-2}$, we have $\dom(\phi_t)\le\Ca\Delta p^{-2}$ at every step. Thus the number of vertices at distance less than $2r+20$ from $\dom(\phi_t)$ is at most
\[\big(1+\Delta+\dots+\Delta^{2r+19}\big)\Ca\Delta p^{-2}< 2\Ca\Delta^{2r+20} p^{-2}\]
which by choice of $C$ is smaller than the number of vertices in $F$ with $N_H(x)$ independent. For the second part, suppose that at some time $t$ we pick a vertex $v$ such that $\big|\big(N_G(v)\cap S\big)\setminus\im(\phi_t)\big|<\tfrac14\mu pn$. For each $t-\tfrac{1}{100(\Delta+1)}\mu pn\le t'<t$, we have $\big|\big(N_G(v)\cap S\big)\setminus\im(\phi_{t'})\big|<\tfrac3{10}\mu pn$, yet at each of these times $v$ is not picked, so that the vertex picked at each time has at most as many uncovered neighbours in $S$ as $v$. Let $Z$ be the set of vertices chosen at line~\ref{line:choosev} in each of these time steps. Then for each $z\in Z$ we have $\big|\big(N_G(v)\cap S\big)\setminus\im(\phi_t)\big|\le\tfrac3{10}\mu pn$. But since $\delta(G)>\tfrac12pn$, by~\eqref{eq:intS} and choice of $\eps$ we have $\big|N_G(z)\cap S\big|\ge \tfrac25\mu pn$, so $\big|N_G(z)\cap\im(\phi_t)\big|\ge \tfrac1{10}\mu pn$ for each $z\in Z$. By choice of $C$, we have $|Z|=\tfrac{1}{100(\Delta+1)}\mu pn \ge \Ca p^{-1}\log n$. Since $|\im(\phi)|\le(\Delta+1)|V_0|\le\tfrac{1}{100}\mu n$, by choice of $C$, this contradicts the good event of Proposition~\ref{prop:chernoff}.

We have justified that Algorithm~\ref{alg:pre} completes, and indeed that at each time we reach line~\ref{line:choosenbs} there are at least $\tfrac14\mu pn$ vertices of $\big(N_G(v)\cap S\big)\setminus\im(\phi)$ to choose from. In order to specify how to choose these vertices, we need the following claim.

\begin{claim}\label{cl:chooseW}
 Given any set $Y$ of $\tfrac14\mu pn$ vertices of $V(G)$, there exists $W\subset Y$ of size at least $\tfrac{1}{8r}\mu pn$ and an index $i\in[r]$ with the following property. For each $w\in W$ and each $j\in[k]$, we have $|N_G(w,V_{i,j})|\ge dp|V_{i,j}|$.
\end{claim}
\begin{claimproof}
First let $Y'$ be obtained from $Y$ by removing all vertices $y\in Y$ such that either $|N_\Gamma(y,V_0)|\ge \eps p n$, or for some $i\in[r]$ and $j\in[k]$ we have $\big|N_\Gamma(y,V_{i,j})\big|\neq (1\pm\eps)p|V_{i,j}|$. Because the good event of Proposition~\ref{prop:chernoff} occurs, the total number of vertices removed is at most $2kr \Ca p^{-1}\log n<\tfrac12|Y|$, where the inequality is by choice of $C$. Now given any $y\in Y'$, if for each $i\in[r]$ there is $j\in[k]$ such that $\big|N_G(y,V_{i,j})\big|<dp|V_{i,j}|$, then, since the $\{V_{i,j}\}$ are $k$-equitable, we have $|N_G(y)|\le \eps p n+dpn+(1+\eps)\tfrac{k-1}{k}pn+r<\big(\tfrac{k-1}{k}+\gamma\big)pn$, a contradiction. We conclude that for each $y\in Y'$ there exists $i\in[r]$ such that $|N_G(y,V_{i,j})|\ge dp|V_{i,j}|$ for each $j\in[k]$. We let $W$ be the vertices of $Y'$ giving a majority choice of $i$.
\end{claimproof}

Now at each time $t$, in line~\ref{line:choosenbs} of Algorithm~\ref{alg:pre}, we choose the vertices $w_1,\dots,w_\ell$ as follows. Let $Y=\big(N_G(v_t)\cap S\big)\setminus\im(\phi_t)$. Let $i_t\in[r]$ be an index, and $W\subset Y$ be a set of size $\tfrac{1}{8r}\mu pn$, such that $\big|N_G(w,V_{i_t,j})\big|\ge dpn|V_{i_t,j}|$ for each $j\in[k]$, whose existence is guaranteed by Claim~\ref{cl:chooseW}. By construction, and by our choice of $\mu$, we can apply Lemma~\ref{lem:common} with input $d$, $k$, $\Delta$, $\epsa$, $r$ and $\eps$, with the clusters $\big\{V_{i_t,j}\big\}_{j\in[k]}$ as the $\big\{V_i\big\}_{i\in[k]}$, and inputting a subset of $W$ of size $10^{-10}\tfrac{\eps^4pn}{k^4r^4}$ as requried for~\ref{cnl:W}. This last is possible by choice of $\mu$. To verify the conditions of Lemma~\ref{lem:common}, observe that~\ref{cnl:bal} follows from~\ref{main:Gsize},~\ref{cnl:Vreg} from~\ref{main:Greg}, and~\ref{cnl:Wdeg} from Claim~\ref{cl:chooseW}. We obtain a $\Delta$-tuple of vertices in $W$ satisfying~\ref{cnl:Gsize}--\ref{cnl:Nreg}. We let $w_1,\dots,w_\ell$ be the first $\ell$ vertices of this tuple.

Let $H'=H- \dom(\phi_{t^*})$. We next define image restricting vertex sets and create an updated homomorphism $f^*:V(H')\to [r]\times[k]$. For each $x\in V(H)\setminus\dom(\phi_{t^*})$, let $J_x=\phi_{t^*}\big(N_H(x)\cap\dom(\phi_{t^*})\big)$. Now, since the vertices $\{x_t\}_{t\in[t^*]}$ are by construction at pairwise distance at least $2r+20$, in particular for each $y\in V(H')$ with $J_y\neq\emptyset$ the vertex $y$ is at distance two from one $x_t$, and at distance greater than $r+10$ from all others. Let $j\in [k]$ such that $f(y)=(1,j)$. Then we set $f^*(y):=(i_{t},j)$. Next, for each $t\in[t^*]$ and each $z\in V(H)$ at distance $3,\dots,i_t+1$ from $x_t$, we set $f^*(z)$ as follows. Recall that $f(z)=(1,j)$ for some $j\in[k]$. We set $f^*(z)=\big(i_t+2-\dist(x_t,z),j\big)$. Because the $\{x_t\}$ are at pairwise distance at least $2r+20$, no vertex is at distance $r+5$ or less from any two $x_t$ and $x_{t'}$, so that $f^*$ is well-defined. Because $R^k_r$ contains $B^k_r$, the $f^*$ we constructed so far is a graph homomorphism. Furthermore, for each $x_t$ the set of vertices $z$ at distance $i_t+1$ from $x_t$ are in the first $\sqrt{\beta}n$ vertices of $\cL$, and so by~\ref{H:v1} satisfy $f^*(z)=f(z)$. We complete the construction of $f^*$ by setting $f^*(z)=f(z)$ for each remaining $z\in V(H)\setminus\dom(\phi_{t^*})$. Because $f$ is a graph homomorphism, $f^*$ is also a graph homomorphism whose domain is $V(H')$. For each $i\in[r]$ and $j\in[k]$, let $W'_{i,j}$ be the set of vertices $w\in V(H')$ with $f^*(w)\in V_{i,j}$, and let $X'$ consist of $X$ together with all vertices of $H'$ at distance $r+10$ or less from some $x_t$ with $t\in[t^*]$. The total number of vertices $z\in V(H)$ at distance at most $r+10$ from some $x_t$ is at most $2\Delta^{r+10}|V_0|<\tfrac{1}{100}\xi n$. Since $W_{i,j}\symd W'_{i,j}$ contains only such vertices, we have
\begin{enumerate}[label=\itmarabp{H}{b}]
 \item\label{Hp:sizeWp} $m_{i,j}-\tfrac15\xi n\le |W'_{i,j}|\le m_{i,j}+\tfrac15\xi n$,
 \item\label{Hp:sizeX} $|X'| \leq 2\xi n$, 
 \item\label{Hp:edge} $\{f^*(x),f^*(y)\} \in E(R^k_r)$  for every $\{x,y\} \in E(H')$, and
 \item\label{Hp:special} $y,z\in \bigcup_{j'\in[k]}W'_{i,j'}$ for every $x\in W'_{i,j}\setminus X'$ and $xy,yz\in E(H')$.
\end{enumerate}
where~\ref{Hp:sizeX},~\ref{Hp:edge} and~\ref{Hp:special} hold by~\ref{H:sizeX} and definition of $X'$, by definition of $f^*$, and by~\ref{H:special} and choice of $X'$ respectively.

Furthermore, we have
\begin{enumerate}[label=\itmarabp{G}{a}]
\item $\frac{n}{4kr}\leq |\Vij| \leq \frac{4n}{kr}$ for every $i\in[r]$ and $j\in[k]$,
\item $\cV$ is $(\eps,d,p)_G$-lower-regular on $R^k_r$ and $(\eps,d,p)_G$-super-regular on $K^k_r$,
\item both $\big(\NGa(v, V_{i,j}),V_{i',j'}\big)$ and $\big(\NGa(v, V_{i,j}),\NGa(v, V_{i',j'})\big)$ are $(\eps, d,p)_G$-lower-regular pairs for every $\{(i,j),(i',j')\} \in E(R^k_r)$ and $v\in V\setminus V_0$, and
\item $|\NGa(v,V_{i,j})| = (1\pm \eps)p|\Vij|$ for every $i \in [r]$, $j\in [k]$ and every $v \in V \setminus V_0$.
 \item\label{main:GpI} $\big|V_{f^*(x)}\cap\bigcap_{u\in J_x}N_G(u)\big|\ge\alpha p^{|J_x|}|V_{f^*(x)}|$ for each $x\in V(H')$,
 \item\label{main:GpGI} $\big|V_{f^*(x)}\cap\bigcap_{u\in J_x}N_\Gamma(u)\big|=(1\pm\eps^*)p^{|J_x|}|V_{f^*(x)}|$  for each $x\in V(H')$, and
 \item\label{main:GpIreg} $\big(V_{f^*(x)}\cap\bigcap_{u\in J_x}N_\Gamma(u),V_{f^*(y)}\cap\bigcap_{v\in J_y}N_\Gamma(v)\big)$ is $(\eps^*,d,p)_G$-lower-regular for each $xy\in E(H')$.
 \item\label{main:GaI} $\big|\bigcap_{u\in J_x}N_\Gamma(u)\big|\le(1+\epsa) p^{|J_x|}n$ for each $x\in V(H')$,
\end{enumerate}
Properties~\ref{main:Gsize} to~\ref{main:Ggam} are repeated for convenience. Properties~\ref{main:GpI},~\ref{main:GpGI} and~\ref{main:GaI}, are trivial when $J_x=\emptyset$, and are otherwise guaranteed by Lemma~\ref{lem:common}. Finally~\ref{main:GpIreg} follows from~\ref{main:Greg} when $J_x,J_y=\emptyset$, and otherwise is guaranteed by Lemma~\ref{lem:common}.

For each $i\in[r]$ and $j\in[k]$, let $V'_{i,j}=V_{i,j}\setminus\im(\phi_{t^*})$, and let $\cV'=\{V'_{i,j}\}_{i\in[r],j\in[k]}$. Because $V_{i,j}\setminus V'_{i,j}\subset S$ for each $i\in[r]$ and $j\in[k]$, using~\eqref{eq:intS} and Proposition~\ref{prop:subpairs3}, and our choice of $\mu$, we obtain
\begin{enumerate}[label=\itmarabp{G}{b}]
\item\label{Gp:sizeV} $\frac{n}{6kr}\leq |V'_{i,j}| \leq \frac{6n}{kr}$ for every $i\in[r]$ and $j\in[k]$,
\item\label{Gp:Greg} $\cV'$ is $(2\eps,d,p)_G$-lower-regular on $R^k_r$ and $(2\eps,d,p)_G$-super-regular on $K^k_r$,
\item\label{Gp:Ginh} both $\big(\NGa(v, V'_{i,j}),V'_{i',j'}\big)$ and $\big(\NGa(v, V'_{i,j}),\NGa(v, V'_{i',j'})\big)$ are $(2\eps, d,p)_G$-lower-regular pairs for every $\{(i,j),(i',j')\} \in E(R^k_r)$ and $v\in V\setminus V_0$, and
\item\label{Gp:GsGa} $|\NGa(v,V'_{i,j})| = (1 \pm 2\eps)p|V_{i,j}|$ for every $i \in [r]$, $j\in [k]$ and every $v \in V \setminus V_0$.
 \item\label{Gp:sizeI} $\big|V'_{f^*(x)}\cap\bigcap_{u\in J_x}N_G(u)\big|\ge\tfrac12\alpha p^{|J_x|}|V'_{f^*(x)}|$,
 \item\label{Gp:sizeGa} $\big|V'_{f^*(x)}\cap\bigcap_{u\in J_x}N_\Gamma(u)\big|=(1\pm2\eps^*)p^{|J_x|}|V'_{f^*(x)}|$, and
 \item\label{Gp:Ireg} $\big(V'_{f^*(x)}\cap\bigcap_{u\in J_x}N_\Gamma(u),V'_{f^*(y)}\cap\bigcap_{v\in J_y}N_\Gamma(v)\big)$ is $(2\eps^*,d,p)_G$-lower-regular.
 \item\label{Gp:GaI} $\big|\bigcap_{u\in J_x}N_\Gamma(u)\big|\le(1+2\epsa) p^{|J_x|}n$ for each $x\in V(H')$,
\end{enumerate}

 We are now almost finished. The only remaining problem is that we do not necessarily have $|W'_{i,j}|=|V'_{i,j}|$ for each $i\in[r]$ and $j\in[k]$. Since $|V'_{i,j}|=|V_{i,j}|\pm 2\Delta^{r+10}|V_0|=m_{i,j}\pm 3\Delta^{r+10}|V_0|$, by~\ref{Hp:sizeWp} we have $|V'_{i,j}|=|W'_{i,j}|\pm \xi n$. We can thus apply Lemma~\ref{lem:balancing}, with input $k$, $r_1$, $\Delta$, $\gamma$, $d$, $8\eps$, and $r$. This gives us sets $V''_{i,j}$ with $|V''_{i,j}|=|W'_{i,j}|$ for each $i\in[r]$ and $j\in[k]$ by~\ref{lembalancing:sizesout}. Let $\cV''=\{V''_{i,j}\}_{i\in[r],j\in[k]}$. Lemma~\ref{lem:balancing} guarantees us the following.
\begin{enumerate}[label=\itmarabp{G}{c}]
\item\label{Gpp:sizeV} $\frac{n}{8kr}\leq |V''_{i,j}| \leq \frac{8n}{kr}$ for every $i\in[r]$ and $j\in[k]$,
\item\label{Gpp:Greg} $\cV''$ is $(4\epsa,d,p)_G$-lower-regular on $R^k_r$ and $(4\epsa,d,p)_G$-super-regular on $K^k_r$,
\item\label{Gpp:Ginh} both $\big(\NGa(v, V''_{i,j}),V''_{i',j'}\big)$ and $\big(\NGa(v, V''_{i,j}),\NGa(v, V''_{i',j'})\big)$ are $(4\epsa, d,p)_G$-lower-regular pairs for every $\{(i,j),(i',j')\} \in E(R^k_r)$ and $v\in V\setminus V_0$, and
\item\label{Gpp:GsGa} we have 
 $(1-4\eps)p|V''_{i,j}| \leq |\NGa(v,V''_{i,j})| \leq (1 + 4\eps)p|V''_{i,j}|$ for every $i \in [r]$, $j\in [k]$ and every $v \in V \setminus V_0$.
 \item\label{Gpp:sizeI} $\big|V''_{f^*(x)}\cap\bigcap_{u\in J_x}N_G(u)\big|\ge\tfrac14\alpha p^{|J_x|}|V''_{f^*(x)}|$,
 \item\label{Gpp:sizeGa} $\big|V''_{f^*(x)}\cap\bigcap_{u\in J_x}N_\Gamma(u)\big|=(1\pm4\eps^*)p^{|J_x|}|V'_{f^*(x)}|$, and
 \item\label{Gpp:Ireg} $\big(V''_{f^*(x)}\cap\bigcap_{u\in J_x}N_\Gamma(u),V''_{f^*(y)}\cap\bigcap_{v\in J_y}N_\Gamma(v)\big)$ is $(4\eps^*,d,p)_G$-lower-regular.
\end{enumerate} 
Here~\ref{Gpp:sizeV} comes from~\ref{Gp:sizeV} and~\ref{lembalancing:symd}, while~\ref{Gpp:Greg} comes from~\ref{lembalancing:regular} and choice of $\eps$. \ref{Gpp:Ginh} is guaranteed by~\ref{lembalancing:inheritance}. Now, each of~\ref{Gpp:GsGa},~\ref{Gpp:sizeI} and~\ref{Gpp:sizeGa} comes from the corresponding~\ref{Gp:GsGa},~\ref{Gp:sizeI} and~\ref{Gp:sizeGa} together with~\ref{lembalancing:gammaout}. Finally,~\ref{Gpp:Ireg} comes from~\ref{Gp:Ireg} and~\ref{Gp:GaI} together with Proposition~\ref{prop:subpairs3} and~\ref{lembalancing:gammaout}.

For each $x\in V(H')$ with $J_x=\emptyset$, let $I_x=V''_{f^*(x)}$. For each $x\in V(H')$ with $J_x\neq\emptyset$, let $I_x=V''_{f^*(x)}\cap\bigcap_{u\in J_x}N_G(u)$. Now $\cW'$ and $\cV''$ are $\kappa$-balanced by~\ref{Gpp:sizeV}, size-compatible by construction, partitions of respectively $V(H')$ and $V(G)\setminus\im(\phi_{t^*})$, with parts of size at least $n/(\kappa r_1)$ by~\ref{Gpp:sizeV}. Letting $\widetilde{W}_{i,j}:=W'_{i,j}\setminus X'$, by~\ref{Hp:sizeX}, choice of $\xi$, and~\ref{Hp:special}, $\{\widetilde{W}_{i,j}\}_{i\in[r],j\in[k]}$ is a $\big(\vartheta,K^k_r\big)$-buffer for $H'$. Furthermore since $f^*$ is a graph homomorphism from $H'$ to $R^k_r$, we have~\ref{itm:blowup:H}. By~\ref{Gpp:Greg},~\ref{Gpp:Ginh} and~\ref{Gpp:GsGa} we have~\ref{itm:blowup:G}, with $R=R^k_r$ and $R'=K^k_r$. Finally, the pair $(\cI,\cJ)=\big(\{I_x\}_{x\in V(H')},\{J_x\}_{x\in V(H')}\big)$ form a $\big(\rho,\tfrac14\alpha,\Delta,\Delta\big)$-restriction pair. To see this, observe that the total number of image restricted vertices in $H'$ is at most $\Delta^2|V_0|<\rho|V_{i,j}|$ for any $i\in[r]$ and $j\in[k]$, giving~\ref{itm:restrict:numres}. Since for each $x\in V(H')$ we have $|J_x|+\deg_{H'}(x)=\deg_H(x)\le\Delta$ we have~\ref{itm:restrict:Jx}, while~\ref{itm:restrict:sizeIx} follows from~\ref{Gpp:sizeI}, and~\ref{itm:restrict:sizeGa} follows from~\ref{Gpp:sizeGa}. Finally,~\ref{itm:restrict:Ireg} follows from~\ref{Gpp:Ireg}, and~\ref{itm:restrict:DJ} follows since $\Delta(H)\le\Delta$. Together this gives~\ref{itm:blowup:restrict}. Thus, by Lemma~\ref{thm:blowup} there exists an embedding $\phi$ of $H'$ into $G\setminus\im(\phi_{t^*})$, such that $\phi(x)\in I_x$ for each $x\in V(H')$. Finally, $\phi\cup\phi_{t^*}$ is an embedding of $H$ in $G$, as desired.
\end{proof}

With Theorem~\ref{thm:maink} in hand, we can now present the proof of Theorem~\ref{thm:main}.

\begin{proof}[Proof of Theorem~\ref{thm:main}]
Given $\gamma$, $\Delta$, and $k$, let $\beta>0$, $z>0$, and $C >0$ be returned by Theorem~\ref{thm:maink} with input $\gamma$, $\Delta$, and $k$. Set $\beta^\ast := \beta/2$ and $\Ca := C/\beta$. Let $H$ be a $k$-colourable graph on $n$ vertices with $\Delta(H) \leq \Delta$ such that there exists a set $W$ of at least $\Ca p^{-2}$ vertices in $V(H)$ that are not contained in any triangles of $H$ and such that there exists a labelling $\cL$ of its vertex set of bandwidth at most $\beta^\ast n$. By the choice of $\Ca$ we find an interval $I \subseteq \cL$ of length $\beta n$ containing a subset $F \subseteq W$ with $|F| = C p^{-2}$. Now we can rearrange the labelling $\cL$ to a labelling $\cL'$ of bandwidth at most $2 \beta^\ast n = \beta n$ such that $F$ is contained in the first $\beta n$ vertices in $\cL'$. Then, by Theorem~\ref{thm:maink} we know that $\Gamma = G(n,p)$ satisfies the following a.a.s.~if $p\geq C(\log n/n)^{1/\Delta}$ and in particular if $p\geq \Ca(\log n/n)^{1/\Delta}$. If $G$ is a spanning subgraph of $\Gamma$ with $\delta(G) \geq \big((k-1)/k+\gamma\big)pn$, then $G$ contains a copy of $H$, which finishes the proof. 
\end{proof}

\section{Lowering the probability for degenerate graphs}
\label{sec:proofdegen}
As with Theorem~\ref{thm:main}, we deduce Theorem~\ref{thm:degenerate} from the following more general statement.

\begin{theorem}\label{thm:degen}
For each $\gamma>0$, $\Delta \geq 2$, $D\ge 1$ and $k\geq 1$, there exist constants $\beta >0$, $z>0$, and $C>0$ such that the following holds asymptotically almost surely for $\Gamma = G(n,p)$ if $p\geq C\big(\frac{\log n}{n}\big)^{1/(2D+1)}$. Let $G$ be a spanning subgraph of $\Gamma$ with $\delta(G) \geq\big(\frac{k-1}{k}+\gamma\big)pn$ and let $H$ be a graph on $n$ vertices with $\Delta(H) \leq \Delta$ and degeneracy at most $D$, that has a labelling $\cL$ of its vertex set of bandwidth at most $\beta n$, a $(k+1)$-colouring that is $(z,\beta)$-zero-free with respect to $\cL$ and where the first $\sqrt{\beta} n$ vertices in $\cL$ are not given colour zero and the first $\beta n$ vertices in $\cL$ include $C p^{-2}$ vertices that are not in any triangles or copies of $C_4$ in $H$. Then $G$ contains a copy of $H$.
\end{theorem}

The proof of Theorem~\ref{thm:degen} is quite similar to that of
Theorem~\ref{thm:maink}. We provide only a sketch, highlighting the differences.
(For more details and background on this result see~\cite{JuliaEsDiss}.)
The most important of these differences are that we do not use Lemma~\ref{lem:common} in the pre-embedding, and that we use a version of Lemma~\ref{thm:blowup} whose performance is better for degenerate graphs. In order to state this, we need the following definitions. Given an order $\tau$ on $V(H)$ and a family $\cJ$ of image restricting vertices, we define $\pitau(x):=|J_x|+\big|\{y\in N_H(x):\tau(y)<\tau(x)\}\big|$. Now the condition on $\tau$ we need for our enhanced blow-up lemma is the following.

\begin{definition}[$(\tD,p,m)$-bounded order]\label{def:Dpm_bdd_order} % m=\eps n/r_1
  Let~$H$ be a graph given with buffer sets $\widetilde{\mathcal{W}}$ and
  a restriction pair~$\cI=\{I_i\}_{i\in[r]}$ and~$\cJ=\{J_i\}_{i\in[r]}$.
  Let~$\tW=\bigcup\widetilde{\mathcal{W}}$.
  Let~$\tau$ be an ordering of $V(H)$ and $W^e\subset V(H)$.
  Then~$\tau$ is a \emph{$(\tD,p,m)$-bounded order} for~$H$, $\widetilde{\mathcal{W}}$,
  $\cI$ and $\cJ$
  with \emph{exceptional set} $W^e$ if the following conditions are
  satisfied for each $x\in V(H)$.
  \begin{enumerate}[label=\itmarab{ORD}]
   \item\label{ord:Dx} Define \[
     \tD_x:=\begin{cases}
       \tD-2 & \text{if there is $yz\in E(H)$ with $y,z\in N_H(x)$ and $\tau(y),\tau(z)>\tau(x)$}\\
       \tD-1 & \text{else if there is $y\in N_H(x)$ with $\tau(y)>\tau(x)$}\\
       \tD & \text{otherwise}\,.
     \end{cases}\]
     We have $\pitau(x)\le \tD_x$, and if $x\in N(\tW)$ even $\pitau(x)\le \tD_x-1$. Finally, if $x\in\tW$ we have $\deg(x)\le \tD$.
    \item\label{ord:halfD} One of the following holds:
      \begin{itemize}
        \item $x\in W^e$,
        \item $\pitau(x)\le \frac12 \tD$,
        \item $x$ is not image restricted and every neighbour~$y$ of~$x$
          with $\tau(y)<\tau(x)$ satisfies $\tau(x)-\tau(y)\le p^{\pitau(x)}m$.
      \end{itemize}
    \item\label{ord:NtX} If $x\in N(\tW)$ then all but at most $\tD-1-\max_{z\not\in W^e}\pitau(z)$
      neighbours~$y$ of $x$ with $\tau(y)<\tau(x)$ satisfy
      $\tau(x)-\tau(y)\le p^{\tD} m$.
 \end{enumerate}  
\end{definition}

To obtain the best probability bound, one should choose $\tau$ to minimise $\tD$. In the proof of Theorem~\ref{thm:degen} we will take $\tau$ to be an order witnessing $D$-degeneracy, $W^e$ will contain all image restricted vertices, and we will choose buffer sets containing vertices of degree at most $2D+1$. One can easily check that this allows us to choose $\tD=2D+1$.

\begin{lemma}[{\cite[Lemma~1.23]{blowup}}]
\label{thm:dblow}
  For all $\Delta\ge 2$, $\Delta_{R'}$, $\Delta_J$, $\tD$, 
  $\alpha,\zeta, d>0$, $\kappa>1$ there exist
  $\eps,\rho>0$ such that for all $r_1$ there is a $C$ such that for
  \[p\ge C\bigg(\frac{\log n}{n}\bigg)^{1/\tD}\] 
  the random graph
  $\Gamma=G_{n,p}$ a.a.s.\ satisfies the following.
   
  Let $R$ be a graph on $r\le r_1$ vertices and let $R'\subset R$ be a
  spanning subgraph with $\Delta(R')\leq \Delta_{R'}$.  Let $H$ and $G\subset
  \Gamma$ be graphs with $\kappa$-balanced, size-compatible vertex
  partitions $\cW=\{W_i\}_{i\in[r]}$ and $\cV=\{V_i\}_{i\in[r]}$,
  respectively, which have parts of size at least $m\ge n/(\kappa r_1)$.
  Let $\widetilde{\mathcal{W}}=\{\tW_i\}_{i\in[r]}$ be a family of subsets of $V(H)$, $\cI=\{I_x\}_{x\in
    V(H)}$ be a family of image restrictions, and $\cJ=\{J_x\}_{x\in V(H)}$
  be a family of restricting vertices.  
  Let $\tau$ be an order of $V(H)$ and $W^e\subset V(H)$ be a set of size
  $|W^e|\le\eps p^{{\max_{x\in W^e}\pitau(x)}}n/r_1$. 
  Suppose that
  \begin{enumerate}[label=\itmarab{DBUL}]
  \item\label{dbul:1} $\Delta(H)\leq \Delta$, $(H,\cW)$ is an $R$-partition, 
    and $\widetilde{\mathcal{W}}$ is an $(\alpha,R')$-buffer for $H$,
  \item\label{dbul:2} $(G,\cV)$ is an $(\eps,d,p)$-lower-regular $R$-partition, which is 
    $(\eps,d,p)$-super-regular on $R'$, 
    has one-sided inheritance on~$R'$,
    and two-sided inheritance on~$R'$ for $\widetilde{\mathcal{W}}$,
  \item\label{dbul:3} $\cI$ and $\cJ$ form
    a $(\rho,\zeta,\Delta,\Delta_J)$-restriction pair.
  \item\label{dbul:4} $\tau$ is a $(\tD,p,\eps n/r_1)$-bounded order for~$H$, $\widetilde{\mathcal{W}}$,
    $\cI$, $\cJ$ with
    exceptional set~$W^e$.
  \end{enumerate}
  Then there is an embedding $\psi\colon V(H)\to V(G)$ such that
  $\psi(x)\in I_x$ for each $x\in H$.
\end{lemma}

\begin{proof}[Sketch proof of Theorem~\ref{thm:degen}]
 We set up constants quite similarly as in the proof of Theorem~\ref{thm:maink}. Specifically,
 given $\gamma>0$, $\Delta\ge 2$, $D$ and $k\ge 2$, let $d$ be returned by Lemma~\ref{lem:G}, with input $\gamma$, $k$ and $r_0:=10\gamma^{-1}$. Let $\alpha=\tfrac{d}{2}$. Let $\tD=2D+1$. Now let $\eBL>0$ and $\rho>0$ be returned by Lemma~\ref{thm:dblow} with input $\Delta$, $\Delta_{R'}=3k$, $\Delta_J=\Delta$, $\tD'$, $\vartheta=\tfrac{1}{100D}$, $\zeta=\tfrac14\alpha$, $d$ and $\kappa=64$. Let $\epsa=\tfrac18\eBL$, and then Lemma~\ref{lem:OSRIL}, for input $\epsa$ and $d$, returns $\eps_1>0$. Let $\eps_0>0$ be small enough both for Lemma~\ref{lem:TSRIL} with input $\epsa$ and $d$, and for Lemma~\ref{lem:OSRIL} with input $\eps_1$ and $d$. 
 
  We choose $\eps=\min\big(\eps_0,d,\tfrac{1}{4D}\epsa,\tfrac{1}{2k}\big)$. Putting $\eps$ into Lemma~\ref{lem:G} returns $r_1$. Next, Lemma~\ref{lem:balancing}, for input $k$, $r_1$, $\Delta$, $\gamma$, $d$ and $8\eps$, returns $\xi>0$. We assume without loss of generality that $\xi\le 1/(10kr_1)$, and set $\beta=10^{-12}\xi^2/(\Delta k^4r_1^2)$. Let $\mu=\tfrac{\eps^2}{100000kr}$. Finally, suppose $\Ca$ is large enough for each of these lemmas, for Lemma~\ref{thm:blowup}, for Proposition~\ref{prop:chernoff} with input $\eps$, and for Lemma~\ref{lem:hypgeo} with input $\eps\mu^2$ and $\Delta$.

 We set $C=10^{10}k^2r_1^2\eps^{-2}\xi^{-1}\Delta^{2r_1+20}\mu^{-1}\Ca$, and $z=10/\xi$. Given $p\ge C\big(\tfrac{\log n}{n}\big)^{1/(2D+1)}$, a.a.s.\ $G(n,p)$ satisfies the good events of Lemma~\ref{thm:dblow}, Lemma~\ref{lem:G}, Lemma~\ref{lem:OSRIL} and Lemma~\ref{lem:TSRIL}, and Proposition~\ref{prop:chernoff}, with the stated inputs. Suppose that $\Gamma=G(n,p)$ satisfies these good events. 
 
 Let $G$ be a spanning subgraph of $\Gamma$ with $\delta(G)\ge\big(\tfrac{k-1}{k}+\gamma\big)pn$. Let $H$ be any graph on $n$ vertices with $\Delta(H)\le\Delta$, and let $\cL$ be a labelling of $V(H)$ of bandwidth at most $\beta n$ whose first $\beta n$ vertices include $C p^{-2}$ vertices that are not contained in any triangles or four-cycles of $H$, and such that there exists a $(k+1)$-colouring that is $(z,\beta)$-zero-free with respect to $\cL$, and the colour zero is not assigned to the first $\sqrt{\beta}n$ vertices. Furthermore, let $\tau$ be a $D$-degeneracy order of $V(H)$.
 
 Next, as in the proof of Theorem~\ref{thm:maink}, we apply Lemma~\ref{lem:G} to $G$, obtaining a partition of $V(G)$ with the properties~\ref{main:Gsize}--\ref{main:Ggam}. Note that if $D=1$, in place of~\ref{main:Ginh} we will ask only for the weaker condition
 \begin{enumerate}[label=\itmsol{G}{3'}]
  \item\label{main:Ginhp} $\big(N_\Gamma(v,V_{i,j}),V_{i',j'}\big)$ is an $(\eps,d,p)_G$-lower-regular pair for every $\big\{(i,j),(i',j')\big\}\in E(R^k_r)$ and $v\in V\setminus V_0$,
 \end{enumerate}
 and thus for $D=1$ we have $|V_0|\le\Ca p^{-1}$, while for $D\ge 2$ we have $|V_0|\le\Ca p^{-2}$.
 
 Next, we apply Lemma~\ref{lem:H2} to obtain a partition of $V(H)$. We use the same inputs as in the proof of Theorem~\ref{thm:maink}, with the exception that $D$ is now given in the statement of Theorem~\ref{thm:degen} rather than being set equal to $\Delta$. The result is a function $f:V(H)\to V(R^k_r)$ and a special set $X$ with the same properties~\ref{H:size}--\ref{H:v1}, and in addition
 \begin{enumerate}[label=\itmsol{H}{6a}]
  \item\label{H:buf} $|\{x\in f^{-1}(i,j): \deg(x) \leq 2D\}| \geq \tfrac{1}{24D} |f^{-1}(i,j)|$. 
 \end{enumerate}
 
 We now continue following the proof of Theorem~\ref{thm:maink}, using Lemma~\ref{lem:hypgeo} with input $\eps\mu^2$ and $D+1$ (rather than $\eps\mu^2$ and $\Delta$), to choose a set $S$ satisfying~\eqref{eq:intS} for each $1\le\ell\le D+1$ and vertices $u_1,\dots,u_{\ell}$ of $V(G)$. We use the same pre-embedding Algorithm~\ref{alg:pre}, with the exception that we choose vertices at line~\ref{line:choosenbs} differently. As before, given $v_{t+1}\in V_0\setminus\im(\phi_t)$, we use Claim~\ref{cl:chooseW} to find a set $W\subset N_G(v_{t+1})$ of size at least $\tfrac{1}{8r}\mu p n$ and an index $i\in[r]$ such that for each $w\in W$ we have $\big|N_G(w,V_{i,j})\big|\ge dp|V_{i,j}|$ for each $j\in[k]$. However, rather than applying Lemma~\ref{lem:common}, we let $w_1,\dots,w_\ell$ be distinct vertices of $W$ which satisfy~\ref{main:GpI}--\ref{main:GaI}. We now justify that this is possible. We choose the $w_1,\dots,w_\ell$ successively. Since $x_{t+1}$ is not contained in any triangle or four-cycle of $H$, we have $|J_x|\le 1$ for each $x\in V(H)$, so that~\ref{main:GpI} is automatically satisfied. By Proposition~\ref{prop:chernoff},~\ref{main:GpGI} and~\ref{main:GaI} are satisfied for all but at most $2\Ca kr_1 p^{-1}\log n$ vertices of $W$. It remains to show that we can obtain~\ref{main:GpIreg}, which we do as follows. For $s\in[\ell]$, when we come to choose $w_s$, we insist that for any $\big\{(i,j),(i',j')\big\}\in E(R^k_r)$, the following hold. First, $\big(N_\Gamma(w_s,V_{i,j}),V_{i',j'}\big)$ is $(\eps_1,d,p)_G$-lower-regular. Second, $\big(N_\Gamma(w_s,V_{i,j}),N_\Gamma(w_s,V_{i',j'})\big)$ is $(\epsa,d,p)_G$-lower-regular. Third, for each $1\le t\le s-1$, $\big(N_\Gamma(w_s,V_{i,j}),N_\Gamma(w_t,V_{i',j'})\big)$ is $(\epsa,d,p)_G$-lower-regular. The conditions of respectively Lemma~\ref{lem:OSRIL}, Lemma~\ref{lem:TSRIL}, and Lemma~\ref{lem:OSRIL} are in each case satisfied (in the last case by choice of $w_t$) and thus in total at most $3\Ca k^2r_1^2\max\{p^{-2},p^{-1}\log n\}$ vertices of $W$ are prohibited. Since $5\Ca k^2r_1^2\max\{p^{-2},p^{-1}\log n\}<\tfrac{|W|}{2}<\ell$ by choice of $C$, at each step there is a valid choice of $w_s$. Since for each $x\in V(H')$ we have $|J_x|\le 1$, this construction guarantees~\ref{main:GpIreg}.
 
 We now return to following the proof of Theorem~\ref{thm:maink}. We obtain $\cV'$ by removing the images of pre-embedded vertices, and $\cV''$ by applying Lemma~\ref{lem:balancing}. Note that here~\ref{lembalancing:gammaout} may be trivial, that is, the error term $\Ca\log n$ may dominate the main term when $s$ is large, but we only require it for $s=1$ to obtain~\ref{Gpp:sizeV}--\ref{Gpp:Ireg}.
 
 Finally, we are ready to apply Lemma~\ref{thm:dblow} to complete the embedding. We define $(\cI,\cJ)$ as in the proof of Theorem~\ref{thm:maink}. We however let $\widetilde{W}_{i,j}$ consist of the vertices of $W'_{i,j}\setminus X$ whose degree is at most $2D$. By~\ref{H:buf} there are at least $\tfrac{1}{100D}|W'_{i,j}|$ of these, so that $\widetilde{\mathcal{W}}$ is a $(\vartheta,K^k_r)$-buffer, giving~\ref{dbul:1}. Now~\ref{dbul:2} follows from~\ref{Gpp:Greg} and~\ref{Gpp:Ginh}. Finally, $(\cI,\cJ)$ is a $(\rho,\tfrac14\alpha,\Delta,\Delta)$-restriction pair, giving~\ref{dbul:3}, exactly as in the proof of Theorem~\ref{thm:maink}. However now we need to give an order $\tau'$ on $V(H')$ and a set $W^e\subset V(H')$. The former is simply the restriction of $\tau$ to $V(H')$, and the set $W^e$ consists of all vertices $x\in V(H)$ with $|J_x|>0$.
 
 We now verify the remaining conditions of Lemma~\ref{thm:dblow}. We claim $|W^e|\le \Delta^2|V_0|\le \eps p^{\max_{x\not\in W^e}\pi^{\tau'}(x)}n/r_1$. Observe that $\pi^{\tau'}(x)\le\pitau(x)+|J_x|\le D+1$. For $D=1$, we have $|V_0|\le\Ca p^{-1}$, and by choice of $C$ the desired inequality follows. For $D\ge 2$, we have $|V_0|\le\Ca p^{-2}$, and again by choice of $C$ we have the desired inequality.
 
 The last condition we must verify is~\ref{dbul:4}, that $\tau'$ is a $(\tD,p,\eps n/r_1)$-bounded order. For any vertex $x$ of $H'$, we have $\pi^{\tau'}(x)\le\pitau(x)+1\le D+1$, and furthermore for all vertices not in $W^e$ we have $\pi^{\tau'}(x)=\pitau(x)\le D$. To verify~\ref{ord:Dx}, first note that by construction the vertices of $\bigcup\widetilde{\mathcal{W}}$ have degree at most $2D\le\tD$. Further, observe that if $D=1$ then $H'$ contains no triangles, and $\tD=3=D+2$. Since vertices in $N\big(\bigcup\widetilde{\mathcal{W}}\big)$ are by construction not image restricted, so are not in $W^e$, this is as required for~\ref{ord:Dx}. If on the other hand $D\ge 2$ then $\tD\ge D+3$, and again the conditions of~\ref{ord:Dx} are met.
 Next, if $x\not\in W^e$ then $\pi^{\tau'}(x)\le D$, so that~\ref{ord:halfD} holds.
 Finally, observe that $\max_{z\not\in W^e}\pi^{\tau'}(z)\le D$, and vertices $x\in N\big(\bigcup\widetilde{\mathcal{W}}\big)$ by construction have $\pi^{\tau'}(x)=\pitau(x)\le D$, so that~\ref{ord:NtX} holds.
 
 We can thus apply Lemma~\ref{thm:dblow} to embed $H'$ into $G'$, completing the embedding of $H$ into $G$ as desired.
\end{proof}

The proof of Theorem~\ref{thm:degenerate} from Theorem~\ref{thm:degen} follows the deduction of Theorem~\ref{thm:main} from Theorem~\ref{thm:maink}, and we omit it.

\section{The Bandwidth Theorem in bijumbled graphs}
\label{sec:proofjumbled}
Again, Theorem~\ref{thm:jumbled} is a consequence of the following.
\begin{theorem}\label{thm:jumbledk}
For each $\gamma >0$, $\Delta \geq 2$, and $k \geq 1$, there exists a constant $c >0$ such that the following holds for any $p>0$. Given $\nu\le cp^{\max(4,(3\Delta+1)/2)}n$, suppose $\Gamma$ is a $\big(p,\nu\big)$-bijumbled graph, $G$ is a spanning subgraph of $\Gamma$ with $\delta(G) \geq\big(\tfrac{k-1}{k}+\gamma\big)pn$, and $H$ is a $k$-colourable graph on $n$ vertices with $\Delta(H) \leq \Delta$ and bandwidth at most $c n$. Suppose further that $H$ has a labelling $\cL$ of its vertex set of bandwidth at most $c n$, a $(k+1)$-colouring that is $(z,c)$-zero-free with respect to $\cL$, and where the first $\sqrt{c} n$ vertices in $\cL$ are not given colour zero, and the first $cn$ vertices in $\cL$ include $c^{-1}p^{-6} \nu^2n^{-1}$ vertices in $V(H)$ that are not contained in any triangles of $H$. Then $G$ contains a copy of $H$.
\end{theorem}

The proof of Theorem~\ref{thm:jumbledk} is a straightforward modification of
that of Theorem~\ref{thm:maink}. Rather than repeating the entire proof, we
sketch the modifications which have to be made. Again, for more details and
background on this result see~\cite{JuliaEsDiss}.

Since we are working with bijumbled graphs, we need to work with regular pairs, rather than lower-regular pairs, at all times. In order to use this concept, and to work with bijumbled graphs, we need versions of Lemmas~\ref{thm:blowup},~\ref{lem:OSRIL}, and~\ref{lem:TSRIL}, and Proposition~\ref{prop:chernoff}, which work with regular pairs and with $\Gamma$ a bijumbled graph rather than a random graph. We also need the following easy proposition, which lower bounds the possible $\nu$ for a $(p,\nu)$-jumbled graph with $p>0$.
\begin{proposition}\label{prop:bijn}
 Suppose $\tfrac{16}{n}<p<1-\tfrac{16}{n}$. There does not exist any $(p,\nu)$-bijumbled $n$-vertex graph with $\nu\le\min\big(\sqrt{pn/32},\sqrt{(1-p)n/32}\big)$.
\end{proposition}
\begin{proof}
 Suppose that $\Gamma$ is a $(p,\nu)$-bijumbled graph on $n$ vertices with $p\le\tfrac12$. If $\Gamma$ contains $\tfrac12n$ vertices of degree at least $4pn$, then we have $e(\Gamma)\ge pn^2$, and letting $A,B$ be a maximum cut of $\Gamma$, by bijumbledness we have
 \[\frac12pn^2\le e(A,B)\le p|A||B|+\nu\sqrt{|A||B|}\le\tfrac14pn^2+\frac12\nu n\,,\]
 and thus $\nu\ge pn/2\ge\sqrt{pn/32}$.
 
 If on the contrary $\Gamma$ contains at least $\tfrac12n$ vertices of degree less than $4pn$, then let $A$ be a set of $\tfrac{1}{8p}$ such vertices, and $B$ a set of $\tfrac{n}{4}$ vertices with no neighbours in $A$. By bijumbledness, we have
 \[0\ge p|A||B|-\nu\sqrt{|A||B|}=\frac{n}{32}-\nu\sqrt{n/(32p)}\]
 and thus $\nu\ge \sqrt{pn/32}$. The same argument applied to $\overline{\Gamma}$ proves the $p\ge\tfrac12$ case.
\end{proof}

The following sparse blow-up lemma for jumbled graphs is proved in~\cite{blowup}.
\begin{lemma}[{\cite[Lemma~1.25]{blowup}}]\label{thm:jblowup}
  For all $\Delta\ge 2$, $\Delta_{R'}$, $\Delta_J$, $\alpha,\zeta, d>0$, $\kappa>1$
  there exist $\eps,\rho>0$ such that for all $r_1$ there is a $c>0$ such
  that if $p>0$ and 
  \[\beta\le cp^{\max(4,(3\Delta+1)/2)}n\]
  any $(p,\beta)$-bijumbled graph~$\Gamma$ on $n$ vertices satisfies the following.
   
  Let $R$ be a graph on $r\le r_1$ vertices and let $R'\subset R$ be a spanning
  subgraph with $\Delta(R')\leq \Delta_{R'}$.
  Let $H$ and $G\subset \Gamma$ be graphs given with $\kappa$-balanced,
  size-compatible vertex partitions 
  $\cX=\{X_i\}_{i\in[r]}$ and $\cV=\{V_i\}_{i\in[r]}$, respectively, which have parts of size at
  least $m\ge n/(\kappa r_1)$. Let 
  $\tcX=\{\tX_i\}_{i\in[r]}$ be a family of subsets of $V(H)$,
  $\cI=\{I_x\}_{x\in V(H)}$ be a family of image restrictions, and
  $\cJ=\{J_x\}_{x\in  V(H)}$  be a family of restricting vertices.
  Suppose that
  \begin{enumerate}[label=\itmarab{JBUL}]
  \item $\Delta(H)\leq \Delta$, $(H,\cX)$ is an $R$-partition,
    and $\tcX$ is an
    $(\alpha,R')$-buffer for $H$,
  \item $(G,\cV)$ is an $(\eps,d,p)$-regular $R$-partition, which is 
    $(\eps,d,p)$-super-regular on $R'$, 
    and has one-sided inheritance on~$R'$,
    and two-sided inheritance on~$R'$ for $\tcX$,
  \item $\cI$ and $\cJ$ form
    a $(\rho p^\Delta,\zeta,\Delta,\Delta_J)$-restriction pair.
  \end{enumerate}
  Then there is an embedding $\psi\colon V(H)\to V(G)$ such that $\psi(x)\in
  I_x$ for each $x\in H$.
\end{lemma}
There are three differences between this result and Lemma~\ref{thm:blowup}. First, we assume a bijumbledness condition on $\Gamma$, rather than that $\Gamma$ is a typical random graph. Second, we require regular pairs in place of lower-regular pairs. Third, the number of vertices we may image restrict is much smaller. We will see that these last two restrictions do not affect our proof substantially.

Next, in~\cite{ABSS}, the following regularity inheritance lemmas for bijumbled graphs are proved.
\begin{lemma}[{\cite[Lemma~3]{ABSS}}]\label{lem:pOSRIL}
  For each $\eps',d>0$ there are $\eps,c>0$ such that for all $0<p<1$ the
  following holds.
  Let $G\subset \Gamma$ be graphs and $X,Y,Z$ be disjoint vertex sets in
  $V(\Gamma)$.  Assume that 
  \begin{itemize}
  \item $(X,Z)$ is $(p,cp^{3/2}\sqrt{|X||Z|})$-bijumbled
    in $\Gamma$,
  \item $(X,Y)$ is $\big(p,cp^2(\log_2\tfrac{1}{p})^{-1/2}\sqrt{|X||Y|}\big)$-bijumbled in $\Gamma$, and
   \item $(X,Y)$ is $(\eps,d,p)_G$-regular.  
  \end{itemize}
  Then, for all but at most at most $\eps'|Z|$ vertices~$z$ of~$Z$, the pair
  $\big(N_\Gamma(z)\cap X,Y\big)$ is $(\eps',d,p)_G$-regular.
\end{lemma}

\begin{lemma}[{\cite[Lemma~4]{ABSS}}]\label{lem:pTSRIL}
  For each $\eps',d>0$ there are $\eps,c>0$ such that for all $0<p<1$ the
  following holds.  Let $G\subset \Gamma$ be graphs and $X,Y,Z$ be disjoint
  vertex sets in  $V(\Gamma)$.  Assume that
  \begin{itemize}
  \item $(X,Z)$ is $(p,cp^{2}\sqrt{|X||Z|})$-bijumbled
    in $\Gamma$,
  \item $(Y,Z)$ is $(p,cp^3\sqrt{|Y||Z|})$-bijumbled in $\Gamma$,
  \item $(X,Y)$ is $(p,cp^{5/2}\big(\log_2\tfrac{1}{p}\big)^{-\frac12}\sqrt{|X||Y|})$-bijumbled in $\Gamma$, and
  \item $(X,Y)$ is $(\eps,d,p)_G$-regular.  
  \end{itemize}
  Then, for all but at most $\eps'|Z|$ vertices~$z$ of~$Z$, the pair
  $\big(N_\Gamma(z)\cap X,N_\Gamma(z)\cap Y\big)$ is $(\eps',d,p)_G$-regular.
\end{lemma}

The following two lemmas, which more closely resemble Lemmas~\ref{lem:OSRIL} and~\ref{lem:TSRIL}, are corollaries.

\begin{lemma}
\label{lem:pseudOSRIL}
For each $\eo, \ao >0$ there exist $\eps_0 >0$ and $C >0$ such that for any $0 < \eps < \eps_0$ and $0 < p <1$, if $\Gamma$ is any $(p,\nu)$-bijumbled graph the following holds. For any disjoint sets $X$ and $Y$ in $V(\Gamma)$ with $|X|\geq C p^{-3} \nu$ and $|Y| \geq C p^{-2} \nu$, and any subgraph $G$ of $\Gamma[X,Y]$ which is $(\eps, \ao,p)_G$-regular, there are at most $C p^{-3}\nu^2|X|^{-1}$ vertices $z \in V(\Gamma)$ such that $(X \cap N_{\Gamma}(z),Y)$ is not $(\eo,\ao,p)_G$-regular.
\end{lemma}

\begin{lemma}
\label{lem:pseudTSRIL}
For each $\et,\at>0$ there exist $\eps_0>0$ and
$C >0$ such that for any $0<\eps<\eps_0$ and $0<p<1$, if $\Gamma$ is any $(p,\nu)$-bijumbled graph the following holds. For any disjoint sets $X$
and $Y$ in $V(\Gamma)$ with $|X|,|Y|\ge Cp^{-3}\nu$, and any
subgraph $G$ of $\Gamma[X,Y]$ which is $(\eps,\at,p)_G$-regular, there are
at most $Cp^{-6}\nu^2/\min\big(|X|,|Y|\big)$ vertices $z \in V(\Gamma)$
such that $\big(X\cap N_\Gamma(z),Y\cap N_\Gamma(z)\big)$ is not
$(\et,\at,p)_G$-regular.
\end{lemma}

Note that the bijumbledness requirements of this lemma are such that if $Y$ and $Z$ are sets of size $\Theta(n)$, then $X$ must have size $\Omega\big(p^{-6}\nu^2 n^{-1}\big)$. This is where the requirement of Theorem~\ref{thm:jumbledk} for vertices of $H$ not in triangles comes from.

Finally, we give a bijumbled graphs version of Proposition~\ref{prop:chernoff}. We defer its proof, which is standard, and similar to that of Proposition~\ref{prop:chernoff}, to Appendix~\ref{app:tools}.
\begin{proposition}
\label{prop:pseudchernoff}
For each $\eps>0$ there exists a constant $C >0$ such that for every $p>0$, any graph $\Gamma$ which is $(p,\nu)$-jumbled has the  following property. For any disjoint $X,Y \subseteq V(\Gamma)$ with $|X|,|Y|\ge \eps^{-1}p^{-1}\nu$, we have $e(X,Y)=(1\pm\eps)p|X||Y|$, and $e(X)\le 2p|X|^2$. Furthermore, for every $Y\subset V(\Gamma)$ with $|Y|\ge Cp^{-1}\nu$, the number of vertices $v \in V(\Gamma)$ with $\big||\NGa(v,Y)| - p |Y|\big| > \eps p |Y|$ is at most $Cp^{-2}\nu^2|Y|^{-1}$.  
\end{proposition}

Now, using these lemmas, we can prove bijumbled graph versions of Lemmas~\ref{lem:G} and~\ref{lem:common}, and use these to complete the proof of Theorem~\ref{thm:jumbledk}. All these proofs are straightforward modifications of those in the previous sections. Briefly, the modifications we make are to replace `lower-regular' with `regular' in all proofs, to replace applications of lemmas for random graphs with the bijumbled graph versions above, and to recalculate some error bounds. 

The only one of our main lemmas which changes in an important way is the following Lemma for $G$.

\begin{lemma}[Lemma for $G$, bijumbled graph version] 
\label{lem:pseudG}
For each $\gamma > 0$ and integers $k \geq 2$ and $r_0 \geq 1$ there exists $d > 0$ such that for every $\eps \in \left(0, \frac{1}{2k}\right)$ there exist $r_1\geq 1$ and $c,\Ca>0$ such that the following holds for any $n$-vertex $(p,\nu)$-bijumbled graph $\Gamma$ with $\nu\le c p^3n$ and $p>0$. Let $G=(V,E)$ be a spanning subgraph of $\Gamma$ with $\delta(G) \geq  \left(\frac{k-1}{k} + \gamma\right)pn$. Then there exists an integer $r$ with $r_0\leq kr \leq r_1$, a subset $V_0 \subseteq V$  with $|V_0| \leq \Ca p^{-6}\nu^2 n^{-1}$, 
a $k$-equitable vertex partition $\cV = \{\Vij\}_{i\in[r],j\in[k]}$ of $V(G)\setminus V_0$,
and a graph $R^k_r$ on the vertex set $[r] \times [k]$ with $K^k_r \subseteq B^k_r \subseteq R^k_r$, with $\delta(R^k_r) \geq \left(\frac{k-1}{k} + \frac{\gamma}{2}\right)kr$, and such that the following are true.
\begin{enumerate}[label=\itmarab{G}]
\item \label{plemG:size} $\frac{n}{4kr}\leq |\Vij| \leq \frac{4n}{kr}$ for every $i\in[r]$ and $j\in[k]$,
\item \label{plemG:regular} $\cV$ is $(\eps,d,p)_G$-regular on $R^k_r$ and $(\eps,d,p)_G$-super-regular on $K^k_r$,
\item \label{plemG:inheritance} both $\big(\NGa(v, V_{i,j}),V_{i',j'}\big)$ and
  $\big(\NGa(v', V_{i,j}),\NGa(v', V_{i',j'})\big)$ are $(\eps,d,p)_G$-regular
  pairs for every edge $\{(i,j),(i',j')\} \in E(R^k_r)$, every $v\in V\setminus
  (V_0 \cup V_{i,j})$ and $v'\in V\setminus (V_0 \cup V_{i,j} \cup V_{i',j'})$,
\item \label{plemG:gamma} we have 
 $(1-\eps)p|\Vij| \leq |\NGa(v,V_{i,j})| \leq (1 + \eps)p|\Vij|$ for every $i \in [r]$, $j\in [k]$ and every $v \in V \setminus V_0$.
\end{enumerate}
\end{lemma}

The change here, apart from replacing `lower-regular' with `regular', and working in bijumbled graphs, is that $V_0$ may now be a much larger set. Nevertheless, the proof is basically the same.
\begin{proof}[Sketch proof of Lemma~\ref{lem:pseudG}]
 We begin the proof as in that of Lemma~\ref{lem:G}, setting up the constants in the same way, with the exception that we replace Lemmas~\ref{lem:OSRIL} and~\ref{lem:TSRIL} with Lemmas~\ref{lem:pseudOSRIL} and~\ref{lem:pseudTSRIL}, and Proposition~\ref{prop:chernoff} with Proposition~\ref{prop:pseudchernoff}. We require $C$ to be sufficiently large for Lemmas~\ref{lem:pseudOSRIL} and~\ref{lem:pseudTSRIL}, and for Proposition~\ref{prop:pseudchernoff}. We define $\Ca=100k^2r_1^3C/\epsa$ as in the proof of Lemma~\ref{lem:G}, and set
\[c=10^{-5}(\epsa)^3(kr_1)^{-3}(\Ca)^{-1}\,.\]
 
 We now assume $\Gamma$ is $(p,\nu)$-bijumbled rather than random, with $\nu\le cp^3n$. In particular, by choice of $c$ this implies that 
 \begin{equation}\label{eq:psG:s}
  10k^2r_1^2 Cp^{-2}\nu^2n^{-1}\le\epsa pn\quad\text{and}\quad 10k^2r_1^3 C p^{-6}\nu^2n^{-1}\le\epsa n\,.
 \end{equation}
 
 We obtain a regular partition, with a reduced graph containing $B^k_r$, exactly as in the proof of Lemma~\ref{lem:G}, using Proposition~\ref{prop:pseudchernoff} in place of Proposition~\ref{prop:chernoff} to justify the use of Lemma~\ref{lem:regularitylemma}. The next place where we need to change things occurs in defining $Z_1$, where we replace `lower-regular' with `regular', and in estimating the size of $Z_1$. Using Lemmas~\ref{lem:pseudOSRIL} and~\ref{lem:pseudTSRIL}, and Proposition~\ref{prop:chernoff} with Proposition~\ref{prop:pseudchernoff}, we replace~\eqref{eq:sizeZ1} with
 \[|Z_1|\le kr_1^2Cp^{-6}\nu^2n^{-1}+kr_1^2Cp^{-3}\nu^2n^{-1}+2kr_1Cp^{-2}\nu^2n^{-1}\le 4kr_1^2Cp^{-6}\nu^2n^{-1} \leByRef{eq:psG:s}\frac{\epsa}{kr_1}n\,.\]
 Note that the final conclusion is as in~\eqref{eq:sizeZ1}.
 
 We can now continue following the proof of Lemma~\ref{lem:G} until we come to estimate the size of $Z_2$, where we use Proposition~\ref{prop:pseudchernoff} and replace~\eqref{eq:sizeZ2} with
 \[|Z_2|\le r_1+kr_1 Cp^{-2}\nu^2n^{-1}\leByRef{eq:psG:s}\frac{\epsa}{kr_1}pn\,.\]
 Again, the final conclusion is as in~\eqref{eq:sizeZ2}.
 
 The next change we have to make is in estimating the size of $V_0$, when we replace~\eqref{eq:sizeV0} with
 \[|V_0|\le|Z_1|+|Z_2|\le  4kr_1^2 Cp^{-6}\nu^2n^{-1}+r_1+kr_1Cp^{-2}\nu^2 n^{-1}\le \Ca p^{-6}\nu^2n^{-1}\,.\]
 
 Finally, we need regular pairs in~\ref{plemG:regular} and~\ref{plemG:inheritance}. We obtained regular pairs from Lemma~\ref{lem:regularitylemma} and in the definition of $Z_1$, so that we only need Proposition~\ref{prop:subpairs3} to return regular pairs. We always apply Proposition~\ref{prop:subpairs3} to pairs of sets of size at least $\tfrac{\epsa pn}{r_1}$, altering them by a factor $\epsa$. Now Proposition~\ref{prop:pseudchernoff} shows that if $X$ and $Y$ are disjoint subsets of $\Gamma$ with $|X|,|Y|\le (\epsa p)^{-1}\nu$, then $e_\Gamma(X,Y)\le (1+\epsa)p|X||Y|$, as required. By choice of $c$, we have $(\epsa p)^{-1}\nu\le (\epsa)^2pn/r_1$, so that the condition of Proposition~\ref{prop:subpairs3} to return regular pairs is satisfied.
\end{proof}

The other one of our main lemmas which requires change, Lemma~\ref{lem:common}, only requires changing `lower-regular' to `regular' and replacing the random graph with a bijumbled $\Gamma$. This does require some change in the proof, as we then use the bijumbled graph versions of various lemmas, whose error bounds are different.
\begin{lemma}[Common neighbourhood lemma, bijumbled graph version]
\label{lem:pseudcommon}
For each $d>0$, $k \geq 1$, and $\Delta \geq 2$ there exists $\alpha >0$ such that for every $\eps^\ast \in (0,1)$ there exists $\eps_0 >0$ such that for every $r\geq 1$ and every $0<\eps\le\eps_0$ there exists $c>0$ such that the following is true. For any $n$-vertex $(p,cp^{\Delta+1}n)$-bijumbled graph $\Gamma$ the following holds. Let $G=(V,E)$ be a (not necessarily spanning) subgraph of $\Gamma$ and $\{V_i\}_{i\in[k]}\cup \{W\}$ a vertex partition of a subset of $V$ such that the following are true for every $i,i'\in [k]$.
\begin{enumerate}[label=\itmarab{G}]
 \item\label{pcnl:bal} $\frac{n}{4kr}\le |V_i|\le \frac{4n}{kr}$,
 \item\label{pcnl:Vreg} $(V_i,V_{i'})$ is $(\eps, d, p)_G$-regular,
 \item\label{pcnl:W} $|W|\ge\frac{\eps pn}{16kr^2}$, and
 \item\label{pcnl:Wdeg} $|N_G(w,V_i)| \geq dp|V_i|$ for every $w \in W$. 
\end{enumerate}
Then there exists a tuple $(w_1, \ldots, w_\Delta) \in \binom{W}{\Delta}$ such that for every $\Lambda,\Lambda^\ast\subseteq[\Delta]$, and every $i \neq i' \in [k]$ we have
\begin{enumerate}[label=\itmarab{W}]
 \item\label{pcnl:Gsize} $|\bigcap_{j\in \Lambda} N_G(w_j,V_i)|\geq \alpha p^{|\Lambda|}|V_i|$,
 \item\label{pcnl:Gasizen} $|\bigcap_{j\in \Lambda} N_{\Gamma}(w_j)| \le (1 + \eps^\ast)p^{|\Lambda|}n$,
 \item\label{pcnl:Gasize} $ (1-\eps^\ast)p^{|\Lambda|}|V_i| \leq |\bigcap_{j\in \Lambda} N_{\Gamma}(w_j,V_i)| \leq (1 + \eps^\ast)p^{|\Lambda|}|V_i|$, and
 \item\label{pcnl:Nreg} $\big(\bigcap_{j\in \Lambda}\NGa(w_j,V_i),\bigcap_{j^\ast\in \Lambda^\ast}\NGa(w_{j^\ast},V_{i'})\big)$ is $(\eps^\ast, d,p)_G$-regular if $|\Lambda|,|\Lambda^\ast| < \Delta$ and either $\Lambda\cap\Lambda^\ast=\varnothing$ or $\Delta\geq 3$ or both.
\end{enumerate} 
\end{lemma}
The main modifications we make to the proof of Lemma~\ref{lem:common} are to replace Lemmas~\ref{lem:OSRIL} and~\ref{lem:TSRIL} with Lemmas~\ref{lem:pseudOSRIL} and~\ref{lem:pseudTSRIL}, and Proposition~\ref{prop:chernoff} with Proposition~\ref{prop:pseudchernoff}, and to replace all occurrences of `lower-regular' with `regular'. We sketch the remaining modifications below.
\begin{proof}[Sketch proof of Lemma~\ref{lem:pseudcommon}]
We begin the proof by setting constants as in the proof of Lemma~\ref{lem:common}, but appealing to Lemmas~\ref{lem:pseudOSRIL} and~\ref{lem:pseudTSRIL}, and Proposition~\ref{prop:pseudchernoff}, rather than their random graph equivalents. 

We set $c=10^{-20}2^{-2\Delta}\eps^5(Ct_1kr)^{-4}$. Suppose $\nu\le cp^{\Delta+2}n$, and that $\Gamma$ is an $n$-vertex $(p,\nu)$-bijumbled graph rather than a random graph.

In order to apply Lemma~\ref{lem:SRLb} to $G$, we need to observe that its condition is satisfied by Proposition~\ref{prop:pseudchernoff} and because $\eps^{-1}p^{-1}\nu<10^{-10}\tfrac{\eps^4pn}{k^4r^4}$ by choice of $c$. The same inequality justifies further use of Proposition~\ref{prop:pseudchernoff} to find the desired $W'$. Estimating the size of $W'$, we replace~\eqref{eq:sizeW} with
\begin{equation}\label{eq:psizeW}
 |W'|\ge 10^{-11}\frac{\eps^4pn}{t_1k^4r^4}\ge 10^5 Cp^{-2}\nu\,,
\end{equation}
where the final inequality is by choice of $c$.

We only need to change the statement of Claim~\ref{claim:common} by replacing `lower-regular' with `regular' in~\ref{cnl:cl:Wreg} and~\ref{cnl:cl:Vreg}. However we need to make rather more changes to its inductive proof. The base case remains trivial. In the induction step, we need to replace~\eqref{eq:common:sizeNGa} with
\[\big|\bigcap_{j\in\Lambda}N_\Gamma(w_j,V'_i)\big|\ge (1-\eps_0)^{\Delta-2}p^{\Delta-2}\frac{n}{8tr}\ge 10^5Cp^{-4}\nu\,,\]
where the final inequality is by choice of $c$. This, together with $|W'|\ge 10^5Cp^{-2}\nu$ from~\eqref{eq:psizeW}, justifies that we can apply Lemma~\ref{lem:pseudOSRIL}. We obtain that at most $2^\Delta k^2 Cp^{-3}\nu^2\tfrac{8krt_1}{n}$ vertices $w$ in $W$ violate~\ref{cnl:cl:Wreg}.

The estimate on the number of vertices violating~\ref{cnl:cl:NGVp} does not change.

For~\ref{cnl:cl:NGaVp}, we need to observe that $\big|\bigcup_{j\in\Lambda}N_\Gamma(w_j,V'_i)\big|=(1\pm\eps_0)^{|\Lambda|}p^{|\Lambda|}|V'_i|$, and in particular by choice of $\eps_0$ and $c$ this quantity is at least $Cp^{-1}\nu$. Then Proposition~\ref{prop:pseudchernoff} then gives that at most $2^{\Delta+1}kCp^{-2}\nu^2\tfrac{8krt_1}{n}$ vertices destroy~\ref{cnl:cl:NGaVp}, and the same calculation gives the same bound for the number of vertices violating~\ref{cnl:cl:NGaV} and~\ref{cnl:cl:NGa}.

Finally, for~\ref{cnl:cl:Vreg}, we need to use the inequality $(1-\eps_0)^{\Delta-1}p^{\Delta-1}\tfrac{n}{4kr}\ge Cp^{-2}\nu$, which holds by choice of $c$, to justify that Lemmas~\ref{lem:pseudOSRIL} and~\ref{lem:pseudTSRIL} can be applied as the corresponding random graph versions are in Lemma~\ref{lem:common}. We obtain quite different bounds from these lemmas, however. If $\Delta=2$, then we only use Lemma~\ref{lem:pseudOSRIL}, with an input regular pair having both sets of size at least $\tfrac{n}{4kr}$, so that the number of vertices violating~\ref{cnl:cl:Vreg} in this case is at most $2^{2\Delta}k^2Cp^{-3}\nu^2\tfrac{4kr}{n}$. If $\Delta\ge 3$, we use both Lemma~\ref{lem:pseudOSRIL} and~\ref{lem:pseudTSRIL}. The set playing the r\^ole of $X$ in Lemma~\ref{lem:pseudOSRIL} has size at least $(1-\eps_0)^{\Delta-2}p^{\Delta-2}\tfrac{n}{4kr}$, while we apply Lemma~\ref{lem:pseudTSRIL} with both sets of the regular pair having at least this size. As a consequence, the number of vertices violating~\ref{cnl:cl:Vreg} is at most $2^{2\Delta+1}k^2Cp^{-6}\nu^2 (1-\eps_0)^{2-\Delta}p^{2-\Delta}\tfrac{4kr}{n}$ for the case $\Delta\ge 3$.

Putting this together, for the case $\Delta=2$ we replace~\eqref{eq:common:bad2} with the following upper bound for the number of vertices $w\in W'$ which cannot be chosen as $w_{\ell+1}$.
\[2^\Delta k^2Cp^{-3}\nu^2\frac{8krt_1}{n}+2^\Delta k\epsaa_\Delta|W'|+3\cdot 2^{\Delta+1} kCp^{-2}\nu^2\frac{8krt_1}{n}+2^{2\Delta}k^2Cp^{-3}\nu^2\frac{4kr}{n}\]
By choice of $c$ and $\epsaa_\Delta$, this quantity is at most $\tfrac12|W'|$, completing the induction step for $\Delta=2$. For $\Delta\ge 3$, we replace the upper bound~\eqref{eq:common:bad3} with
\[2^\Delta k^2Cp^{-3}\nu^2\frac{8krt_1}{n}+2^\Delta k\epsaa_\Delta|W'|+3\cdot 2^{\Delta+1} kCp^{-2}\nu^2\frac{8krt_1}{n}+2^{2\Delta+1}k^2Cp^{-6}\nu^2(1-\eps_0)^{2-\Delta}p^{2-\Delta}\frac{4kr}{n}\]
which by choice of $c,\eps_0$ and $\epsaa_\Delta$ is at most $\tfrac12|W'|$, completing the induction step for $\Delta\ge 3$.

We conclude that the modified Claim~\ref{claim:common} continues to hold, and this implies the statement of Lemma~\ref{lem:pseudcommon} as in the proof of Lemma~\ref{lem:common}.
\end{proof}

The proof of Theorem~\ref{thm:jumbledk} is similar to that of Theorem~\ref{thm:maink}. Again, we sketch the modifications.

\begin{proof}[Sketch proof of Theorem~\ref{thm:jumbledk}]
 We begin as in the proof of Theorem~\ref{thm:maink}, setting up constants as there, but replacing Lemma~\ref{lem:G} with Lemma~\ref{lem:pseudG}, Lemma~\ref{lem:common} with Lemma~\ref{lem:pseudcommon}, Lemma~\ref{thm:blowup} with Lemma~\ref{thm:jblowup}, and Proposition~\ref{prop:chernoff} with Proposition~\ref{prop:pseudchernoff}. In addition to the constants defined in the proof of Theorem~\ref{thm:maink} we require $0<c\le 10^{-50}\eps^8\mu\rho\xi^2(\Delta k r_1C)^{-10}$ to be small enough for Lemmas~\ref{lem:pseudG} and~\ref{lem:pseudcommon}.
 
 Now, instead of assuming $\Gamma$ to be a typical random graph, suppose $\nu\le cp^{\max\{4,(3\Delta+1)/2\}}n$, and let $\Gamma$ be an $n$-vertex $(p,\nu)$-bijumbled graph. By Proposition~\ref{prop:bijn} we have
 \begin{equation}\label{eq:jsizep}
  p\ge \Ca\Big(\frac{\log n}{n}\Big)^{1/2}\,.
 \end{equation}
 
 We continue following the proof of Theorem~\ref{thm:maink}. We now assume the first $\beta n$ vertices of $\cL$ include $Cp^{-6}\nu^2n^{-1}$ vertices that are not contained in any triangles of $H$. We appeal to Lemma~\ref{lem:pseudG} rather than Lemma~\ref{lem:G} to obtain a partition of $V(G)$. This partition has $|V_0|\le \Ca p^{-6}\nu^2n^{-1}$ (which is different to the upper bound in the proof of Theorem~\ref{thm:maink}), but still satisfies~\ref{main:Gsize} and~\ref{main:Ggam}, and~\ref{main:Greg} and~\ref{main:Ginh} when `lower-regular' is replaced by `regular' in both statements.
 
 The application of Lemma~\ref{lem:H2} is identical. The application of Lemma~\ref{lem:hypgeo} is also identical, and the deduction of~\eqref{eq:intS} is still valid by~\eqref{eq:jsizep}. The pre-embedding is also identical, except that we replace each occurrence of $\Ca\max\{p^{-2},p^{-1}\log n\}$ with $\Ca p^{-6}\nu^2n^{-1}$, and that we replace the application of Proposition~\ref{prop:chernoff} justifying that at each visit to Line~\ref{line:choosev} we have at least $\tfrac14\mu p n$ choices with an application of Proposition~\ref{prop:pseudchernoff}. To verify the condition of the latter, and to see that this yields a contradiction we use the inequality $|Z|\ge\tfrac{1}{100(\Delta+1)}\mu p n\ge 2\Ca p^{-2}\nu^2\tfrac{8r}{\eps n}$, which  holds by choice of $c$.
 
 Moving on, we justify Claim~\ref{cl:chooseW} by observing that $\tfrac{\eps n}{4kr_1}\ge Cp^{-1}\nu$, which allows us to apply Proposition~\ref{prop:pseudchernoff} in place of Proposition~\ref{prop:chernoff}, and that $2kr\Ca p^{-2}\nu^2\tfrac{4kr_1}{\eps n}\le\tfrac{|Y|}{2}$, both inequalities following by choice of $c$.
 
 Now Lemma~\ref{lem:pseudcommon}, in place of Lemma~\ref{lem:common}, finds $w_1,\dots,w_\ell$. Our construction of $f^*$, and its properties, is identical, while Lemma~\ref{lem:pseudcommon} gives~\ref{main:Gsize}--\ref{main:GaI}, with `lower-regular' replaced by `regular' in~\ref{main:Greg},~\ref{main:Ginh} and~\ref{main:GpIreg}. The deduction of~\ref{Gp:sizeV}--\ref{Gp:GaI} is identical, except that we use the `regular' consequence of Proposition~\ref{prop:subpairs3}. To justify this, observe that each time we apply Proposition~\ref{prop:subpairs3}, we apply it to a regular pair with sets of size at least $(1-\epsa)p^{\Delta-1}\tfrac{n}{4kr}$ by~\ref{main:Gsize} and~\ref{main:GpGI}, and we change the set sizes by a factor $(1\pm 2\mu)$, so that Proposition~\ref{prop:pseudchernoff} gives the required condition. To check this in turn, we need to observe that $2\mu(1-\epsa)p^{\Delta-1}\tfrac{n}{4kr}\ge100\mu^{-1}p^{-1}\nu$, which follows by choice of $c$. We can thus replace `lower-regular' with `regular' in~\ref{Gp:Greg},~\ref{Gp:Ginh} and~\ref{Gp:Ireg}.
 
 Next, we still have $3\Delta^{r+10}|V_0|\le\tfrac{1}{10}\xi n$, so that $|V'_{i,j}|=|W'_{i,j}|\pm\xi n$ is still valid for each $i\in[r]$ and $j\in[k]$. This, together with~\eqref{eq:jsizep}, Proposition~\ref{prop:pseudchernoff}, and the inequality $\tfrac{1}{50000kr_1}\eps^2\xi pn\le 100\eps^{-2}\xi^{-1}p^{-1}\nu$, justifies that we can apply Lemma~\ref{lem:balancing} to obtain~\ref{Gpp:sizeV}--\ref{Gpp:sizeGa}, with `lower-regular' replaced by `regular' in~\ref{Gpp:Greg} and~\ref{Gpp:Ginh}.  Finally, to obtain~\ref{Gpp:Ireg} with `lower-regular' replaced by `regular', we use Proposition~\ref{prop:subpairs3}, with the condition to output regular pairs guaranteed by the inequality $10^{-20}\eps^4k^{-3}r_1^{-3}p^{\Delta-1}n\ge 10^{20}\eps^{-4}k^3r_1^3Cp^{-1}\nu$, which follows by choice of $c$, and Proposition~\ref{prop:pseudchernoff}.
 
 Finally, we verify the conditions for Lemma~\ref{thm:jblowup}. The only point where we have to be careful is with the number of image restricted vertices. The total number of image restricted vertices in $H'$ is at most $\Delta^2|V_0|\le\Delta^2\Ca p^{-6}\nu^2n^{-1}$, which by choice of $c$ and by~\ref{Gpp:sizeV} is smaller than $\rho p^\Delta|V_{i,j}|$ for any $i\in[r]$ and $j\in[k]$, justifying that $(\cI,\cJ)$ is indeed a $(\rho p^\Delta,\tfrac14\alpha,\Delta,\Delta)$-restriction pair. The remaining conditions of Lemma~\ref{thm:jblowup} are verified as in the proof of Theorem~\ref{thm:maink}, and applying it we obtain an embedding $\phi$ of $H'$ into $G\setminus\im(\phi_{t^*})$, so tha $\phi\cup\phi_{t^*}$ is the desired embedding of $H$ into $G$.
\end{proof}

Finally, the deduction of Theorem~\ref{thm:jumbled} from Theorem~\ref{thm:jumbledk} is essentially the same as that of Theorem~\ref{thm:main} from Theorem~\ref{thm:maink}, and we omit it.

\section{Concluding remarks}
\label{sec:remarks}
\subsection{General spanning subgraphs}
Our main theorems place restrictions on the graphs $H$ with respect to whose containment random or pseudorandom graphs have local resilience. As was shown by Huang, Lee and Sudakov~\cite{huang2012}, such restrictions are necessary. Given $\eps>0$, if $\Gamma$ is either a typical random graph $G(n,p)$ or a pseudorandom graph with density $p$, and $p$ is sufficiently small, then one can delete edges from $\Gamma$ in order to remove all triangles at a given vertex $v$, without deleting more than $\eps p n$ edges at any vertex. Thus if $H$ is any graph all of whose vertices are in triangles, if $p=o(1)$ the local resilience of $\Gamma$ with respect to containment of $H$ is $o(1)$.

This leads to the question: if we instead restrict $G$, requiring in addition to the conditions of Theorem~\ref{thm:main} that $G$ contains a positive proportion of the copies of $K_{\Delta+1}$ in $\Gamma$ at each vertex, is it true that $G$ will contain any $k$-colourable, bounded degree spanning subgraph $H$ with sublinear bandwidth without further restriction? We study this question in a forthcoming companion note to this paper, together with Schnitzer~\cite{ABEST}.

\subsection{Optimality of Theorem~\ref{thm:main}}
Recall that Huang, Lee and Sudakov~\cite{huang2012} proved that the restriction on $H$ that $\Ca p^{-2}$ vertices should not be in triangles is necessary for all $p$. For $p$ constant, they proved a version of Theorem~\ref{thm:main}, but the number of vertices in $H$ they require to have independent neighbourhood grows as a tower type function of $p^{-1}$, and they also require these vertices to be well-distributed in the bandwidth order, so that our result is strictly stronger than theirs.

\medskip

On the other hand, we do not believe that the lower bound on $p$ in Theorem~\ref{thm:main} is optimal. For $\Delta=2$, the statement is certainly false for $p\ll n^{-1/2}$, since then $G(n,p)$ has a.a.s.~local resilience $o(1)$ with respect to containing even one triangle. It seems likely that the statement is true down to this point, a log factor improvement on our result. For $\Delta=3$, the statement as written is false for $p\ll n^{-1/3}$. Briefly, the reason for this is that in expectation a vertex is in $O\big(p^6n^3\big)$ copies of $K_4$ in $G(n,p)$, and (with some work) this implies that there is a.a.s.\ a subgraph of $G(n,p)$ with minimum degree very close to $pn$ and $p^{-5}n^{-1}$ vertices not in copies of $K_4$. For $p\ll n^{-1/3}$, $p^{-5}n^{-1}\gg p^{-2}$, so that we would also have to insist on many vertices of $H$ not being in copies of $K_4$ to accommodate this. Generalising this, we obtain the following conjecture.

\begin{conjecture}
For each $\gamma >0$, $\Delta \geq 2$, and $k \geq 1$, there exist constants $\beta^\ast >0$ and $\Ca >0$ such that the following holds asymptotically almost surely for $\Gamma = G(n,p)$ if $p \geq \Ca n^{-2/(\Delta+2)}$. Let $G$ be a spanning subgraph of $\Gamma$ with $\delta(G) \geq\left(\frac{k-1}{k}+ \gamma\right)pn$ and let $H$ be a $k$-colourable graph on $n$ vertices with $\Delta(H) \leq \Delta$, bandwidth at most $\beta^\ast n$, there are at least $\Ca p^{-2}$ vertices in $V(H)$ that are not contained in any triangles of $H$, and at least $\Ca p^{-(\Delta+2)(\Delta-1)/2}n^{2-\Delta}$ vertices in $V(H)$ which are not in $K_{\Delta+1}$. Then $G$ contains a copy of $H$.  
\end{conjecture}

This conjecture seems to be hopelessly out of reach with our current state of knowledge. We cannot even prove that $G(n,p)$ itself is universal for graphs on $\tfrac{n}{2}$ vertices with maximum degree~$\Delta$. The best current result in this direction is due to Conlon, Ferber, Nenadov and \v{S}kori\'c~\cite{CFNS}, who show that for $\Delta\ge 3$, if $p\gg n^{-1/(\Delta-1)}\log^5 n$ then $G(n,p)$ is a.a.s.\ universal for graphs on $\big(1-o(1)\big)n$ vertices of maximum degree $\Delta$, finally breaking the $n^{-1/\Delta}$ barrier which is reached by several papers, but still far from the conjectured truth. It is possible that their methods could be used to prove a version of Theorem~\ref{thm:main} for almost-spanning $H$ in sparser random graphs, but this does not appear to be straightforward.

\subsection{Optimality of Theorem~\ref{thm:degenerate}}
The `extra' restriction we place in Theorem~\ref{thm:degenerate}, of having many vertices of $H$ which are neither in triangles nor four-cycles, is an artifact of our proof. It would be possible to remove the stipulation regarding four-cycles---one can prove a version of Lemma~\ref{lem:common} capable of embedding vertices in a degeneracy order. However this comes at the cost of a worse lower bound on $p$. It seems likely that one would be able to obtain a result for $p\gg\big(\tfrac{\log n}{n}\big)^{1/(2D+2)}$, but we did not check the details.

As with Theorem~\ref{thm:main}, we expect that the bound $p\ge\big(\tfrac{\log n}{n}\big)^{1/(2D+1)}$ in Theorem~\ref{thm:degenerate} is far from the truth: again the exponent is most likely a factor of roughly $2$ too small. Again, however, proving such a statement in general seems hopeless. Nevertheless, in one interesting case we can substantially improve on Theorem~\ref{thm:degenerate}. Specifically, if $H$ is an $F$-factor for some fixed $F$, then we can follow the proof of Theorem~\ref{thm:degenerate}, but set $\tD=D+3$. We can do this because we choose a degeneracy order on $H$ in which the copies of $F$ are segments. We obtain a version of Theorem~\ref{thm:degenerate} in which $H$ is required to be an $F$-factor, where $F$ is $D$-degenerate, but the lower bound on $p$ improves to $p\ge\Ca\big(\tfrac{\log n}{n}\big)^{1/(D+3)}$. This is still not optimal, but at least the exponent is asymptotically optimal as $D$ grows, rather than being off by a factor of two in the limit. For some specific $F$ one can improve this bound further; moreover for $F$-factors one can slightly improve on Lemma~\ref{thm:dblow} (see the concluding remarks of~\cite{blowup}).

\subsection{Optimality of Theorem~\ref{thm:jumbled}}
The requirement of $\Ca p^{-6}\nu^2 n^{-1}$ vertices of $H$ not in triangles
comes from Lemma~\ref{lem:pTSRIL}. This lemma is proved in~\cite{ABSS}, where it
is conjectured that the bijumbledness requirement is not optimal. What exactly
the optimal result should be is not clear. When $|X|=|Y|=|Z|=\tfrac{n}{3}$, a
construction of Alon~\cite{AlonConstr} shows that
$\big(p,cp^2n\big)$-bijumbledness is necessary for some $c>0$, but in our
application we are interested in $Y$ and $Z$ being of order $n$, and $X$ much
smaller.

We also do not believe that the bijumbledness requirement of
Theorem~\ref{thm:jumbled} is optimal. This requirement comes from
Lemma~\ref{thm:jblowup}, and it is suggested there that the statement could
still hold given only $\big(p,cp^{\Delta+C}\big)$-bijumbledness for some $C$.
Such an improvement there would immediately improve the results here
correspondingly. It is generally conjectured that substantial further
improvement is not possible, in the strong form that it is likely that for some
$C>0$ and all $\Delta$ there exists $c>0$ such that for all large $n$ an
$n$-vertex $\big(p,cp^{\Delta-C}\big)$-bijumbled graph exists which does not
contain $K_{\Delta+1}$ at all.

%%%%%%%%% APPENDIX %%%%%%%%%%%%%%%%%%%%%%%%%%%%%%%%%%%%%%%%%%%

%%% AUTHOR: optional appendix here
\appendix %% you may comment this out if no Appendix
\section*{Appendix}
\section{Tools}
\label{app:tools}

We collect in this appendix proofs of results which are more or less standard but which we could not find in the form we require in the literature. We begin by showing that small alterations to regular pairs give us regular pairs.

\begin{proof}[Proof of Proposition~\ref{prop:subpairs3}]
Let $A \subseteq \hat X$ and $B \subseteq \hat Y$ such that $|A| \geq \hat \eps |\hat X|$ and $|B| \geq \hat \eps |\hat Y|$ be given. Define $A' := A \cap X$ and $B':= B \cap Y$ and note that 
\begin{equation*}
|A'| \geq |A| - \mu |X| \geq \hat \eps |\hat X| - \mu |X| \geq \hat{\eps}(1-\mu) |X| - \mu |X| \geq \big(\hat \eps - 2 \sqrt{\mu}\big) |X| \geq \eps |X|
\end{equation*}
by the definition of $\hat\eps$. Analogously, one can show that $|B'| \geq \eps |Y|$. Since $(X,Y)$ is an $(\eps, d,p)$-regular pair, we know that $d_p(A',B') \geq d- \eps$. Furthermore, we have
\[|A'| \geq |A| - \mu |X| \geq |A| - \mu \frac{|A|}{\hat\eps} \geq \big(1- \sqrt{\mu}\big)|A|\]
and by an analogous calculation we get $|B'| \geq \big(1- \sqrt{\nu}\big)|B|$.
For the number of edges between $A$ and $B$ we get
\begin{align*}
e(A,B) &\geq e(A',B') \geq (d- \eps) p|A'| |B'| \geq (d-\eps)p \big(1-\sqrt{\mu}\big)\big(1-\sqrt{\nu}\big) |A| |B|\\
& \geq \big(d- \eps - 2 \sqrt{\mu} - 2\sqrt{\nu}\big) p |A| |B| \geq (d-\hat\eps) p |A| |B|.
\end{align*}
Therefore we have
\[d_p(A,B) \geq d-\hat\eps,\]
which finishes the proof. 

Now suppose that $(X,Y)$ is $(\eps,d,p)$-fully-regular. Let $d'$ be such that $d_p(A',B')=d'\pm\eps$ for any $A'\subset X$ and $B'\subset Y$ with $|A'|\ge\eps|X|$ and $|B'|\ge\eps |Y|$. Let $A\subset\hat{X}$ and $B\subset \hat{Y}$ with $|A| \geq \hat \eps |\hat X|$ and $|B| \geq \hat \eps |\hat Y|$ be given. As above, we obtain $e(A,B)\ge (d'-\hat\eps) p |A| |B|$. We also have
\begin{align*}
 e(A,B)&\le e(A',B')+e(A',B\setminus B')+e(A\setminus A', B)\\
 &\le (d'+\eps)p|A'||B'|+(1+\mu+\nu)p|A'|\nu|B|+(1+\mu+\nu)p\mu|A||B|\\
 &\le (d'+\hat{\eps})|A||B|\,,
\end{align*}
so that $(\hat{X},\hat{Y})$ is $(\eps,d,p)$-fully-regular, as desired.
\end{proof}

Next, we prove the Sparse Regularity Lemma variant Lemma~\ref{lem:SRLb}, whose proof follows~\cite{Scott}.

\begin{proof}[Proof of Lemma~\ref{lem:SRLb}]
Given $\eps>0$ and $s$, let $L=100s^2\eps^{-1}$. Let $n_1=1$, and for each $j\ge 2$ let $n_j=10000\eps^{-1}n_j2^{sn_j}$. Let $t_1=n_{1000\eps^{-5}(L^2+16Ls^2)+1}$.
 
 We define the energy of a pair of disjoint sets $P,P'$ contained in respectively $V_i$ and $V_{i'}$ to be
 \[\cE(P,P'):=\frac{|P||P'|\min\big(d_p(P,P')^2,2Ld_p(P.P')-L^2\big)}{|V_i||V_{i'}|}\,.\]
 Note that this quantity is convex in $d_p(P,P')$.
 Now given a partition $\cP$ refining $\{V_i\}_{i\in[s]}$, we define the energy of $\cP$ to be
 \[\cE(\cP):=\sum_{\{P,P'\}\subset\cP}\cE(P,P')\,.\]
 We now construct a succession of partitions $\cP_{j+1}$ for each $j\ge 1$, refining $\cP_1:=\{V_i\}_{i\in[s]}$. We claim that for each $j$, the following hold.
 \begin{enumerate}[label=\itmarab{R}]
  \item\label{srl:r1} $\cP_j$ partitions $V_i$ into between $n_j$ and $\big(1+\tfrac{1}{100}\eps\big)n_j$ sets, of which the largest $n_j$ are equally sized.
  \item\label{srl:r2} $\cE(\cP)\ge \tfrac{1}{1000}\eps^5j$.
 \end{enumerate}
 We stop if $\cP_j$ is $\big(\tfrac12\eps,p\big)$-regular. If not, we apply the following procedure. 
 
 For each pair of $\cP_j$ which is not $\big(\tfrac12\eps,0,p\big)$-regular, we take a witness of its irregularity, consisting of a subset of each side of the pair. We let $\cP'_j$ be the union of the Venn diagrams of all witness sets in each part of $\cP_j$. Since $\cP_j$ is not $\big(\tfrac12\eps,p\big)$-regular, there are at least $\tfrac12\eps s^2 n_j^2$ pairs which are not $\big(\tfrac12\eps,0,p\big)$-regular. By choice of $L$ and by~\ref{srl:r1}, at least $\tfrac14\eps s^2 n_j^2$ of these pairs have density not more than $\tfrac12L$. By the defect Cauchy-Schwarz inequality, just from refining these pairs we conclude that $\cE(\cP'_j)\ge\cE(\cP_j)+\tfrac{1}{1000}\eps^5$. Note that, by convexity of $\cE(P,P')$ in $d_p(P,P)$, refining the other pairs does not affect $\cE(P'_j)$ negatively.
 
 We now let $\cP_{j+1}$ be obtained by splitting each set of $\cP'_j$ within each $V_i$ into sets of size $\tfrac{1000-\eps}{1000n_{j+1}}|V_i|$ plus at most one smaller set. Again by Jensen's inequality,  we have $\cE(\cP_{j+1})\ge\cE(\cP'_j)$, giving~\ref{srl:r2}. Since $\cP'_j$ partitions each $V_i$ into at most $n_j2^{sn_j}=\tfrac{1}{10000}\eps n_{j+1}$, the total number of smaller sets is at most $\tfrac{1}{10000}\eps n_{j+1}$. This gives~\ref{srl:r1}. 
 
 Now observe that for any partition $\cP$ refining $\cP_1$, we have $\cE(\cP)\le L^2+16Ls^2$. It follows that this procedure must terminate with $j\le 1000\eps^{-5}(L^2+16Ls^2)+1$. The final $\cP_j$ is thus $\big(\tfrac12\eps,p\big)$-regular. For each $i\in[s]$, let $V_{i,0}$ consist of the union of all but the largest $n_j$ parts of $\cP_j$. Let $\cP$ be the partition of $\bigcup_{i\in[s]}V_i\setminus V_{i,0}$ given by $\cP_j$. This is the desired equitable $(\eps,p)$-regular refinement of $\{V_i\setminus V_{i,0}\}_{i\in[s]}$.
\end{proof}

Using Lemma~\ref{lem:SRLb} (purely in the interests of self-containment, as we could also use the results of~\cite{kohayakawa1997}), we now prove Lemma~\ref{lem:regularitylemma}.

\begin{proof}[Proof of Lemma~\ref{lem:regularitylemma}]
 Given $\eps>0$ and $r_0$, without loss of generality we assume $\eps\le\tfrac1{10}$. Let $t_1$ be returned by Lemma~\ref{lem:SRLb} for input $\tfrac{1}{1000}\eps^2s^{-1}$ and $s=100r_0\eps^{-1}$. Let $r_1=st_1$.
 
 Given $\alpha>0$, let $G$ be an $n$-vertex graph with minimum degree $\alpha p n$. Let $\{V_i\}_{i\in[s]}$ be an arbitrary partition of $V(G)$ into sets of as equal as possible size. By assumption, we have $e(V_i,V_{i'})\le 2p|V_i||V_{i'}|$ for each $i\neq i'$. Furthermore, if $V_i$ is a part with $e(V_i)\ge 3p|V_i|^2$, then taking a maximum cut $A,A'$ of $V_i$ we have $e(A,A')\ge\tfrac32 p|V_i|^2$. Enlarging the smaller of $A$ and $A'$ if necessary, we have a pair of sets both of size at most $|V_i|$ between which there are at least $\tfrac32p|V_i|^2$ edges, again contradicting the assumption of Lemma~\ref{lem:regularitylemma}. Thus $G$ satisfies the conditions of Lemma~\ref{lem:SRLb} with input $\tfrac1{1000}\eps^2s^{-1}$ and $s$. Applying that lemma, we obtain a collection $\{V_{i,0}\}_{i\in[s]}$ of sets, and an $(\eps,p)$-regular partition $\cP$ of $\bigcup_{i\in[s]}V_i\setminus V_{i,0}$ which partitions each $V_i\setminus V_0$ into $t\le t_1$ sets. Note that $s\le |\cP|\le r_1$ by construction.
 
Now let $V'_0$ be the union of the $V_{i,0}$ for $i\in[s]$, any sets $W\in\cP$ that lie in more than $\tfrac14\eps s t$ pairs which are not $(\tfrac{1}{1000}\eps,p)$-regular, and at most two vertices from each set $W\in\cP$ in order that the partition of $V(G)\setminus V'_0$ induced by $\cP$ is an equipartition. Because the total number of pairs which are not $(\tfrac1{1000}\eps,p)$-regular is at most $\tfrac{1}{1000}\eps^2s^{-1}(r_0 t)^2$, the number of such sets in any given $V_i$ is at most $\tfrac{1}{100}\eps t$, so $|V'_{i,0}|$ has size at most $\tfrac{1}{50}\eps |V_i|$, and the number of parts of $\cP$ in $V_i\setminus V'_{i,0}$ is larger than $\tfrac{t}{2}$. Thus the partition $\cP'$ of $V(G)\setminus V'_0$ induced by $\cP$ is an $(\eps,p)$-regular equipartition of $V(G)\setminus V'_0$, and we have $|V'_0|\le\eps n$.
 
 We claim that this partition $\cP'$ has all the properties we require. It remains only to check that for each $d\in[0,1]$, the $d$-reduced graph of $\cP'$ has minimum degree at least $(\alpha-d-\eps)t'$. Suppose that $P$ is a part of $\cP'$. Now we have $e(P)\le 3p|P|^2$, since otherwise, as before, a maximum cut $A,A'$ of $P$ has at least $\tfrac32p|P|^2<\tfrac{1}{20}\eps p|P|n$ edges, yielding a contradiction to the assumption on the maximum density of pairs of $G$. By construction, $P$ lies in at most $\tfrac12\eps t'$ pairs which are not $(\eps,p)$-regular, and these contain at most $(1+\tfrac{1}{10}\eps)p|P|\big(\tfrac12\eps t'|P|\big)<\tfrac{3}{4}\eps p |P|n$ edges of $G$. We conclude that at least $\alpha p |P|n-\tfrac{7}{8}\eps p |P|n$ edges of $G$ leaving $P$ lie in $(\eps,p)$-regular pairs of $\cP'$. Of these, at most $dp|P|n$ can lie in pairs of density less than $p$, so that the remaining at least $\big(\alpha-d-\tfrac{7}{8}\eps\big)p|P|n$ edges lie in $(\eps,d,p)$-regular pairs. If so many edges were in less than $(\alpha-d-\eps)t'$ pairs leaving $P$, this would contradict our assumption on the maximum density of $G$, so that we conclude $P$ lies in at least $(\alpha-d-\eps)t'$ pairs which are $(\eps,d,p)$-regular, as desired.
\end{proof}

\begin{proof}[Proof of Proposition~\ref{prop:pseudchernoff}]
 Given $\eps>0$, set $C'=100\eps^{-2}$ and $C=200C'\eps^{-1}$. Suppose that $\Gamma$ is $(p,\nu)$-bijumbled.
 
 First, given disjoint $X,Y\subset V(\Gamma)$ with $|X|,|Y|\ge \eps^{-1}p^{-1}\nu$, $(p,\nu)$-bijumbledness of $\Gamma$ we have
 $e(X,Y)=p|X|||Y|\pm\nu\sqrt{|X||Y|}$, and we need only verify that
 $\nu\sqrt{|X||Y|}\le\eps p|X||Y|$, which follows from the lower bound on $|X|,|Y|$.
 
 For the second property, let $(A,B)$ be a maximum cut of $X$. We have $e(A,B)\ge\tfrac12e(X)$, and $|A||B|\le\tfrac14|X|^2$. By $(p,\nu)$-bijumbledness of $\Gamma$, we conclude
 \[e(X)\le 2e(A,B)\le 2p|A||B|+2\nu\sqrt{|A||B|}\le \frac12p|X|^2+\nu|X|\]
 so that it is enough to verify $\nu|X|\le p|X|^2$, which duly follows from the lower bound on $|X|$.
 
  Now let $Y\subset V(\Gamma)$ have size at least $Cp^{-1}\nu$. We first show that there are at most $C'p^{-2}\nu^2|Y|^{-1}$ vertices in $\Gamma$ which have less than $(1-\eps)p|Y|$ neighbours in $Y$. If this were false, then we could choose a set $X$ of $C'p^{-2}\nu^2|Y|^{-1}$ vertices in $\Gamma$ which have less than $(1-\eps)p|Y|$ neighbours in $Y$. Since by choice of $C$ we have $(1-\eps)p|Y|\le \big(1-\tfrac\eps2\big)p|Y\setminus X|$, we see that $e(X,Y\setminus X)<\big(1-\tfrac{\eps}{2}\big)p|X||Y\setminus X|$. Since
  \[\nu\sqrt{|X||Y|}=\nu\sqrt{C'p^{-2}\nu^2}=\sqrt{C'}\nu^2p^{-1}<\frac{\eps}{2}p|X||Y\setminus X|\]
  this is a contradiction to $(p,\nu)$-bijumbleness of $\Gamma$.
  
Next we show that there are at most $2C'p^{-2}\nu^2|Y|^{-1}$ vertices of $\Gamma$ which have more than $(1+\eps)p|Y|$ neighbours in $Y$. Again, if this is not the case we can let $X$ be a set of $2C'p^{-2}\nu^2|Y|^{-1}$ vertices of $\Gamma$ with more than $(1+\eps)p|Y|$ neighbours in $Y$.

If there are more than $\tfrac12|X|$ vertices of $X$ with more than $\tfrac12\eps p|Y|$ neighbours in $X$, then we have $e(X)\ge\tfrac18\eps p|X||Y|$. Taking a maximum cut $A,B$ of $X$, we have $e(A,B)\ge\tfrac{1}{16}\eps p|X||Y|$, and by $(p,\nu)$-bijumbledness of $\Gamma$ we therefore have
\[\frac{1}{16}\eps p|X||Y|\le p|A||B|+\nu\sqrt{|A||B|}\le\frac14p|X|^2+\frac12\nu|X|\,,\]
and since $|X|\le\tfrac{1}{100}\eps |Y|$, we conclude $|Y|\le 100\eps^{-1}p^{-1}\nu$, a contradiction to the choice of $C$.

We conclude that there are $\tfrac12|X|$ vertices $X'$ of $X$ have at most $\tfrac12\eps p|Y|$ neighbours in $X$, and hence at least $\big(1+\tfrac12\eps\big)p|Y|$ neighbours in $Y\setminus X$. By $(p,\nu)$-bijumbledness of $\Gamma$ we have
\[\frac12|X|\Big(1+\frac12\eps\Big)p|Y|\le e(X',Y\setminus X)\le \frac12 p|X||Y|+\nu\sqrt{\frac{1}{2}p|X||Y|}\,,\]
from which we have $\eps C'p^{-1}\nu^2\le 2\sqrt{C'}\nu^2p^{-1}$, a contradiction to the choice of $C'$.
\end{proof}

%%%%%%% ACKNOWLEDGE %%%%%%%%%%%%%%%%%%%%%%%%%%%%%%%%%%%%

%%% AUTHOR: optional acknowledgments here
\section*{Acknowledgements} %%  you may comment this out if no Ackno

We thank two anonymous referees for valuable comments that improved our exposition.

%%%%%%% BIBLIOGRAPHY %%%%%%%%%%%%%%%%%%%%%%%%%%%%%%%%%%%%

%%% AUTHOR:
%%% Bibliography goes here. Note that the arXiv cannot process bibtex
%%% or biber bibliographies.  Example of acceptable bibliograpy format:
\bibliographystyle{amsplain}
\bibliography{references}

%%%%%%% ADDRESSES %%%%%%%%%%%%%%%%%%%%%%%%%%%%%%%%%%%%

%%% AUTHOR: Include a short description of each author following the
%%% structure below. Use the same short tags used previously.  
%%% Use \imageat{} and \imagedot{} instead of "@" and "." in
%%% email addresses-this replaces the symbols with graphics to avoid 
%%% e-mail address harvesting from the .pdf file
\begin{aicauthors}
\begin{authorinfo}[PA]
  Peter Allen\\
  London School of Economics, \\
  Department of Mathematics, \\
  Houghton Street, \\
  London WC2A 2AE, UK \\
  p\imagedot{}d\imagedot{}allen\imageat{}lse\imagedot{}ac\imagedot{}uk
\end{authorinfo}
\begin{authorinfo}[JB]
  Julia B\"ottcher\\
  London School of Economics, \\
  Department of Mathematics, \\
  Houghton Street, \\
  London WC2A 2AE, UK \\
  j\imagedot{}boettcher\imageat{}lse\imagedot{}ac\imagedot{}uk
\end{authorinfo}
\begin{authorinfo}[JE]
  Julia Ehrenm\"uller\\
  Technische Universit\"at Hamburg, \\
  Institut f\"ur Mathematik, \\
  Am Schwarzenberg-Campus 3, \\
  21073 Hamburg, Germany \\
  julia\imagedot{}ehrenmueller\imageat{}gmail\imagedot{}com
\end{authorinfo}
\begin{authorinfo}[AT]
  Anusch Taraz\\
  Technische Universit\"at Hamburg, \\
  Institut f\"ur Mathematik, \\
  Am Schwarzenberg-Campus 3, \\
  21073 Hamburg, Germany \\
  taraz\imageat{}tuhh\imagedot{}de
\end{authorinfo}
\end{aicauthors}

\end{document}